\documentclass[a4paper,11pt,reqno]{amsart}

\usepackage[a4paper,width={16cm},left=2.5cm,bottom=3.5cm, top=3.5cm]{geometry}

\usepackage{color}
\usepackage{hyperref}
\usepackage[english]{babel} 
\usepackage[utf8]{inputenc} 
\usepackage[T1]{fontenc} 
\usepackage{amsmath}
\usepackage{amsfonts}
\usepackage{amssymb}
\usepackage{amsthm}
\usepackage{comment}
\usepackage{lmodern}
\usepackage{mathrsfs}
\usepackage{graphics}
\usepackage{enumerate}
\usepackage{pgf,tikz}
\usepackage{cases}
\usepackage{subfigure}
\usepackage{xfrac}

\numberwithin{equation}{section} 
\numberwithin{table}{section}

{ 
\theoremstyle{plain}
\newtheorem{theorem}{Theorem}[section]
\newtheorem{proposition}[theorem]{Proposition}
\newtheorem{lemma}[theorem]{Lemma}
\newtheorem{corollary}[theorem]{Corollary}
\newtheorem{remark}[theorem]{Remark}
\newtheorem{definition}[theorem]{Definition}
\newtheorem{claim}{Claim}
}

\def\R{\mathbb{R}}
\def\Z{\mathbb{Z}}
\def\C{\mathbb{C}}
\def\N{\mathbb{N}}

\def\eps{\varepsilon}

\renewcommand{\Re}{\mathop{\mathfrak{Re}}}
\renewcommand{\Im}{\mathop{\mathfrak{Im}}}

\newenvironment{proofclaim}[1][Proof of the claim]{\begin{proof}[#1]}{\end{proof}}
\newcommand{\resetclaim}{\setcounter{claim}{0}}

\begin{document}

\title[On multiple AB eigenvalues]{On multiple eigenvalues for Aharonov--Bohm operators in planar domains}

\author{Laura Abatangelo}
\address{Laura Abatangelo 
\newline \indent Dipartimento di Matematica e Applicazioni, Università degli Studi di Milano-Bicocca,
\newline \indent  Via Cozzi 55, 20125 Milano, Italy.}
\email{laura.abatangelo@unimib.it}

\author{Manon Nys}
\address{Manon Nys 
\newline \indent Dipartimento di Matematica Giuseppe Peano, Università degli Studi di Torino, 
\newline \indent Via Carlo Alberto 10, 10123 Torino, Italy.}
\email{manonys@gmail.com}

\thanks{}

\date{\today}

\begin{abstract}
  We study multiple eigenvalues of a magnetic Aharonov-Bohm
  operator with Dirichlet boundary
  conditions in a planar domain. 
  In particular, we study the structure of the set of the couples position of the pole-circulation
  which keep fixed the multiplicity of a double eigenvalue of the operator with the pole at the origin 
  and half-integer circulation. 
  We provide sufficient conditions for which this set is made of an isolated point.
  The result confirms and validates a lot of numerical simulations available in preexisting literature.
\end{abstract}

\subjclass[2010]{}


\maketitle

\section{Introduction} 

\subsection{Presentation of the problem and main results}

An infinitely long and infinitely thin solenoid, perpendicular to 
the plane $(x_1,x_2)$ at the point $a=(a_1,a_2)\in\R^2$ produces 
a point-like magnetic field whose flux remains constantly equal to 
$\alpha \in \R$ as the solenoid's radius goes to zero. Such a 
magnetic field is a $2 \pi \alpha$-multiple of the Dirac delta at 
$a$, orthogonal to the plane $(x_1,x_2)$; it is generated by the 
Aharonov--Bohm vector potential singular at the point $a$
\begin{equation} \label{eq:magnetic-pot}
A_a^\alpha (x_1, x_2) = \alpha 
\left( - \frac{x_2 - a_2}{(x_1 - a_1)^2 + (x_2 - a_2)^2}, 
\frac{x_1 - a_1}{(x_1 - a_1)^2 + (x_2 - a_2)^2} \right),
\end{equation}
see e.g. \cite{AB,MOR,AdamiTeta1998}. We are interested in the 
spectral properties of the Schrödinger operator with Aharonov--Bohm 
vector potential
\begin{equation} \label{eq:operator}
(i\nabla + A_a^\alpha)^2 u := -\Delta u + 2 i A_a^\alpha \cdot 
\nabla u + |A_a^\alpha|^2 u,
\end{equation}
acting on functions $u \, : \, \R^2 \to \C$. If the circulation 
$\alpha$ is an integer number, the magnetic potential $A_a^\alpha$ 
can be gauged away by a phase transformation, so that the operator 
$(i \nabla + A_a^\alpha)^2$ becomes spectrally equivalent to the 
standard Laplacian. On the other hand, if $\alpha \not \in \Z$ the 
vector potential $A_a^\alpha$ cannot be eliminated by a gauge 
transformation and the spectrum of the operator is modified by the 
presence of the magnetic field. We refer to Section~\ref{sec:gauge} 
for more details.
This produces the so-called 
Aharonov--Bohm effect: a quantum charged particle is affected by 
the presence of the magnetic field, through the circulation of the 
magnetic potential, even if it moves in a region where the magnetic 
field is zero almost everywhere. 

From standard theory, and as detailled in Section~\ref{sec:preli}, 
if $\Omega$ is an open, bounded and simply connected set of $\R^2$,
when considering Dirichlet boundary conditions, the spectrum of the 
operator \eqref{eq:operator} consists of a diverging sequence of 
positive eigenvalues, that we denote $\lambda_k^{(a,\alpha)}$, $k 
\in \mathbb{N} \setminus\{0\}$, to emphasize the dependance on the 
position of the singular pole and the circulation. As well, we denote 
$\varphi_k^{(a,\alpha)}$ the corresponding eigenfunctions normalized 
in $L^2(\Omega,\C)$. Morever, every eigenvalue has a finite multiplicity. 
In the present paper, we begin to study the possible multiple eigenvalues 
of this operator with respect to the two parameters $(a,\alpha)$.

\medbreak

In the set of papers \cite{BonnaillieNorisNysTerracini2014,
NorisNysTerracini2015,Lena2015,AbatangeloFelli2015,AbatangeloFelli2016,
AbatangeloFelliNorisNys2016,AbatangeloFelliNorisNys2017} the authors 
study the behavior of the eigenvalues of operator \eqref{eq:operator} 
when the singular pole $a$ moves in the domain, letting the 
circulation $\alpha$ fixed. In particular, they focused their attention on 
the asymptotic behavior of simple eigenvalues, which are known to be 
analytic functions of the position of the pole, see \cite{Lena2015}. 
We also recall that in the case of multiple eigenvalues, such a 
map is no more analytic but still continuous, as established in 
\cite{BonnaillieNorisNysTerracini2014,Lena2015}. We then recall the two following results.

\begin{theorem}(\cite[Theorem 1.1, Theorem 1.3]{BonnaillieNorisNysTerracini2014},
\cite[Theorem 1.2, Theorem 1.3]{Lena2015}) \label{thm:reg-a}
Let $\alpha \in \R$ and $\Omega \subset \R^2$ be open, bounded 
and simply connected. Fix any $k \in \N \setminus\{0\}$. The map 
$a \in \Omega \mapsto \lambda_k^{(a,\alpha)}$ has a continuous 
extension up to the boundary $\partial \Omega$, that is 
\[
\lambda_k^{(a,\alpha)} \to \lambda_k^{(b,\alpha)} 
\quad \text{ as } a \to b \in \Omega 
\qquad \text{ and } \qquad
\lambda_k^{(a,\alpha)} \to \lambda_k 
\quad \text{ as } a \text{ converges to } \partial \Omega,
\]
where $\lambda_k$ is the $k$-th eigenvalue of the Laplacian with 
Dirichlet boundary conditions. 

Moreover, if $b \in \Omega$ and if $\lambda_k^{(b,\alpha)}$ is a 
simple eigenvalue, the map $a \in \Omega \mapsto \lambda_k^{(a,\alpha)}$ 
is analytic in a neighborhood of $b$.
\end{theorem}
This Theorem implies an immediate corollary.
\begin{corollary}(\cite[Corollary 1.2]{BonnaillieNorisNysTerracini2014}) \label{coro:minmax}
Let $\alpha \in \R$ and $\Omega \subset \R^2$ be open, bounded 
and simply connected. Fix any $k \in \N \setminus\{0\}$. The map 
$a \in \Omega \mapsto \lambda_k^{(a,\alpha)}$ has an extremal point 
inside $\Omega$, i.e.\ a minimum or maximum point.
\end{corollary}

\smallskip

The above results hold for any circulation $\alpha$ of the 
magnetic potential. The case $\alpha \in \{\tfrac{1}{2}\} + 
\mathbb{Z}$ presents some special features, see \cite{HHOO1999,
FelliFerreroTerracini2011} and Section~\ref{sec:gauge} for more 
details. Indeed, through a correspondance between the magnetic 
problem and a real Laplacian problem on a double covering manifold, 
the operator \eqref{eq:operator} with $\alpha \in \{\tfrac{1}{2}\} 
+ \mathbb{Z}$ behaves as a \emph{real} operator. In particular, 
the nodal set of the eigenfunctions of operator \eqref{eq:operator}, 
i.e.\ the set of points where they vanish, is made of curves and 
not of isolated points as we could expect for complex valued functions. 
More specifically, the magnetic eigenfunctions always have an \emph{odd} 
number of nodal lines ending at the singular point $a$, and therefore 
at least one. This indeed constitues the main difference with the 
eigenfunctions of the Laplacian. From \cite[Theorem 1.3]{FelliFerreroTerracini2011}, 
\cite[Theorem 2.1]{HHOO1999} (see also \cite[Proposition 2.4]{BonnaillieNorisNysTerracini2014}), 
for any $k \in \mathbb{N} \setminus\{0\}$ and $a \in \Omega$, there 
exist $c_k,\,d_k \in \R$ such that 
\begin{equation} \label{eq:expansion-1/2}
\varphi_k^{(a,\alpha)} (a + r (\cos t,\sin t)) 
= e^{i \frac{t}{2}} \, r^{1/2} 
\left( c_k \cos \frac{t}{2} + d_k \sin \frac{t}{2} \right) 
+ f_k(r,t), 
\end{equation}
where $(x_1, x_2) = a + r (\cos t, \sin t)$, $f_k(r, t) = O(r^{3/2})$ 
as $r \to 0^+$ uniformly with respect to $t \in [0,2\pi]$. We remark that 
the eigenfunction has exactly one nodal line ending at $a$ if and only if 
$c_k^2 + d_k^2 \neq 0$, while it is zero for more than one nodal 
line. Moreover, in the first case, the values of $c_k$ and $d_k$ are 
related to the angle which the nodal line leaves $a$ with (this is detailled 
in Subsection~\ref{subsec:proofmain}).

The study of the exact asymptotic behavior of simple eigenvalues at an 
interior point in case $\alpha \in \{\tfrac{1}{2}\} + \mathbb{Z}$ is 
the aim of the two articles \cite{AbatangeloFelli2015,AbatangeloFelli2016}. 
Therein the authors show that such a behavior depends strongly on the local 
behavior of the corresponding eigenfunction \eqref{eq:expansion-1/2}. 
We then recall two particular results of the aforementioned papers.
\begin{theorem}(\cite[Theorem 1.2]{AbatangeloFelli2015}) \label{thm:critical}
Let $\alpha \in \{\tfrac{1}{2}\} + \mathbb{Z}$ and $\Omega \subset 
\R^2$ be open, bounded and simply connected. Fix any $k \in \N 
\setminus\{0\}$. Let $b \in \Omega$ be such that $\lambda_k^{(b,\alpha)}$ 
is a simple eigenvalue. Then, the map $a \in \Omega \mapsto 
\lambda_k^{(a,\alpha)}$ has a critical point at $b$ if and only 
if the corresponding eigenfunction $\varphi_k^{(b,\alpha)}$ 
has more than one nodal line ending at $b$. In particular, this 
critical point is a saddle point.
\end{theorem}
Among many other results we also find the following consequence.
\begin{corollary}(\cite[Corollary 1.5]{AbatangeloFelli2015}) \label{prop:multiple}
Let $\alpha \in \{\tfrac{1}{2}\} + \mathbb{Z}$ and $\Omega \subset \R^2$ 
be open, bounded and simply connected. Fix any $k \in \N \setminus \{0\}$. 
If $b \in \Omega$ is an interior extremal (i.e.\ maximal or minimal) point 
of the map $a \in \Omega \mapsto \lambda_k^{(a,\alpha)}$, then 
$\lambda_k^{(b,\alpha)}$ cannot be a simple eigenvalue.
\end{corollary}

Therefore, when $\alpha \in \{\tfrac{1}{2}\} + \mathbb{Z}$, the 
combination of Corollary~\ref{coro:minmax}, Theorem~\ref{thm:critical} 
and Corollary~\ref{prop:multiple} implies that there always exist points 
of multiplicity higher than one, corresponding to extremal points of the 
map $a \in \Omega \mapsto \lambda_k^{(a,\alpha)}$.

\smallskip

When the circulation $\alpha$ is neither integer nor half-interger, 
i.e.\ $\alpha \in \R \setminus \tfrac{\mathbb{Z}}{2}$, we find much 
less results in literature. The lack of structure particular to 
$\alpha \in \{\tfrac{1}{2}\} + \mathbb{Z}$ does not allow us to find 
as complete results as in \cite{AbatangeloFelli2015,AbatangeloFelli2016}, 
see e.g.\ \cite{AbatangeloFelliNorisNys2017}. However, in \cite{HHOHOO2000} 
the authors show that if $\alpha \not\in \tfrac{\mathbb{Z}}{2}$, 
the multiplicity of the first eigenvalue of operator \eqref{eq:operator} 
is always one, for any position of the singular pole, while it can 
be two when $\alpha \in \{\tfrac{1}{2}\} + \mathbb{Z}$ for specific position of the pole $a$, as already said.

\medbreak

The above considerations let us think that the half-integer case 
can be viewed as a special case among other circulations. Indeed, 
the operator behaves as a real one, the eigenfunctions present the 
special form \eqref{eq:expansion-1/2}. Moreover, when $\alpha \in 
\{\tfrac{1}{2}\} + \mathbb{Z}$, the multiplicity of the first eigenvalue 
must sometimes be higher than one, which is not the case for $\alpha 
\not\in \tfrac{\mathbb{Z}}{2}$.

As already mentioned, in this present paper we investigate the eigenvalues 
of Aharonov--Bohm operators of multiplicity two. In particular, we want to 
understand \emph{how many} are those points of higher multiplicity, that is 
we want to detect the dimension of the intersection manifold (once proved it 
is a manifold, see Section~\ref{sec:abstract}) between the graphs of two 
subsequent eigenvalues. 

Since we want to analyse the multiplicity of the eigenvalues with respect to 
$(a, \alpha) \in \Omega \times \R$, we first need a stronger regularity 
result for the map $(a,\alpha) \mapsto \lambda_k^{(a,\alpha)}$ involving also the circulation $\alpha$, and not only $a \in 
\Omega$ as in Theorem~\ref{thm:reg-a}. 
\begin{theorem} \label{thm:reg-a-alpha}
Let $\Omega \subset \R^2$ be a bounded, open and simply connected 
domain such that $0 \in \Omega$. Fix $k \in \N \setminus\{0\}$. Then,
\[
\text{The map } (a,\alpha) \mapsto \lambda_k^{(a,\alpha)} \text{ is continuous in } 
\Omega \times ( \R \setminus \Z ).
\]
Moreover let $\alpha_0 \in \R \setminus \Z$. If $\lambda_k^{(0,\alpha_0)}$ 
is a simple eigenvalue, then
\[
\text{the map } (a, \alpha) \mapsto \lambda_k^{(a,\alpha)} \text{ is locally } 
C^\infty \text{ in a neighborhood of } (0,\alpha_0).
\]
\end{theorem}

Concerning multiple eigenvalues, our main result is the following.

\begin{theorem} \label{t:main}
Let $\Omega \subset \R^2$ be a bounded, open, simply connected Lipschitz 
domain such that $0 \in \Omega$. Let $\alpha_0 \in \{\tfrac{1}{2}\} 
+ \mathbb{Z}$. Let $n_0 \geq 1$ be such that the $n_0$-th eigenvalue 
$\lambda := \lambda_{n_0}^{(0,\alpha_0)}$ of $(i\nabla + A_0^{\alpha_0})^2$ 
with Dirichlet boundary conditions on $\partial \Omega$ has multiplicity 
two. Let $\varphi_1$ and $\varphi_2$ be two orthonormal in $L^2(\Omega,\C)$ 
and linearly independent eigenfunctions corresponding to $\lambda$. Let 
$c_k, \, d_k \in \R$ be the coefficients in the expansions of $\varphi_k$, 
$k = 1,2$, in \eqref{eq:expansion-1/2}. If $\varphi_1$ and $\varphi_2$ 
satisfy both the following
\begin{itemize}
\item[$(i)$] $c_k^2 + d_k^2 \neq 0$ for $k=1,2$;
\item[$(ii)$] there does not exist $\gamma \in \R$ such that $(c_1, d_1) 
= \gamma (c_2,d_2)$;
\item[$(iii)$] $\int_{\Omega} (i \nabla + A_0^{\alpha_0}) \varphi_1 
\cdot A_0^{\alpha_0} \overline{\varphi_2} \neq 0$; 
\end{itemize}
then there exists a neighborhood $U \subset \Omega \times \R$ of 
$\left(0,\alpha_0 \right)$ such that the set 
\begin{align*}
\{(a,\alpha) \in U \, : \, (i\nabla + A_a^\alpha)^2 \text{ admits an eigenvalue of 
multiplicity two close to } \lambda \} 
= \left\{ (0,\alpha_0 ) \right\}. 
\end{align*}
\end{theorem}

First of all, we make some comments on the conditions appearing in 
Theorem~\ref{t:main}. Condition $(i)$ means that 
both $\varphi_1$ and $\varphi_2$ have a unique nodal line ending at $0$. 
Condition $(ii)$ is related to the relative angle between the two 
nodal lines of $\varphi_1$ and $\varphi_2$: it means that they cannot 
have their nodal line leaving $0$ in a tangential way (see Section~\ref{sec:proof}
for more details). Conditions $(i)$--$(ii)$ can be rephrased 
\begin{itemize}
\item[$(i')$] 
for any $L^2(\Omega,\C)$-orthonormal system of eigenfunctions, the eigenspace 
related to $\lambda$ does not contain any eigenfunction with more than one nodal 
line ending at $0$.
\end{itemize}
Indeed, if condition $(ii)$ is not satisfied, i.e.\ if there exists some 
$\gamma \in \R$ with $(c_1, d_1) = \gamma (c_2,d_2)$, we can  consider the 
linear combinations $\psi_1 = -\gamma \varphi_1 + \varphi_2$ and $\psi_2 = 
\varphi_1 + \gamma \varphi_2$ (up to normalization): they are eigenfunctions 
associated to $\lambda$ and $\psi_1$ has a vanishing first order term in 
\eqref{eq:expansion-1/2}, i.e. strictly more than one nodal line at $0$. 

Condition $(iii)$ is more implicit, since it seems not to be immediately 
related to local properties of the eigenfunctions. Nevertheless, the present 
authors are currently investigating the special case when $\Omega$ is the unit 
disk, in the flavor of \cite{BonnaillieHelffer2013}. This will be an example 
where the assumptions of Theorem \ref{t:main} are sometimes satisfied.

\medbreak

The proof of Theorem~\ref{t:main} relies on two main ingredients. 
First it uses an abstract result drown by \cite{LupoMicheletti1993}. 
By transversality methods, in \cite{LupoMicheletti1993}, the authors 
consider a family of self-adjoint compact operators $\mathcal T_b$ 
parametrized on a Banach space $B$. They provide a sufficient 
condition such that when $\lambda$ is an eigenvalue of $\mathcal T_0$ 
of a given multiplicity $\mu>1$, the set of $b$'s in a small neighborhood 
of $0$ in $B$ for which $\mathcal T_b$ admits an eigenvalue $\lambda_b$ 
(near $\lambda_0$) of the same multiplicity is a manifold in $B$. They 
are also able to compute exactly the codimension in $B$ of this manifold. 
To our aim, a complex version of the result in \cite{LupoMicheletti1993}
will be needed. It is provided in Section~\ref{sec:abstract}.

In order to apply this abstract result, we need to work with \emph{fixed} 
functional spaces, i.e.\; depending neither on the position of the pole 
$a$ nor on the circulation $\alpha$. In this way, the family $\mathcal T_b$ 
can be defined on the very same functions space. Since in general a suitable 
variational setting for this kind of operators depends strongly on the position 
of the pole $a$ (see Section~\ref{sec:preli}), we introduce a suitable family 
of domain perturbations parametrized by $a$ when it is sufficiently close to 
$0$ in $\Omega$. Such perturbations move the pole $a$ into the fixed pole $0$ 
and they produce isomorphisms between the functions spaces dependent on $a$ and a 
fixed one. These domain perturbations will transform the operator 
$(i \nabla + A_a^{\alpha})^2$ (as well as its inverse operator) into a different 
but spectrally equivalent operator, which turns to be the sum 
$(i \nabla + A_0^{\sfrac{1}{2}})^2$ plus a small perturbation with respect to 
$(a, \alpha)$, for $(a,\alpha)$ close to $(0,1/2)$ (see Sections~\ref{sec:modified-operator} 
and \ref{sec:equivalent-operator}).

The second part of the proof relies on the explicit evaluation 
of the perturbation of those operators, applied to eigenfunctions 
of the unperturbed operator $(i \nabla + A_0^{\sfrac{1}{2}})^2$
in order to apply the aforementioned abstract complex result.

\subsection{Motivations}

In order to better understand the general problem and particularly the conditions $(i)$ and 
$(ii)$ in Theorem \ref{t:main}, as well as to support our result, we introduce here some 
numerical simultions by Virginie Bonnaillie--Noël, whom the present 
authors are in debt to. The subsequent simulations are partially 
shown in \cite{BonnaillieHelffer2011,BonnaillieNorisNysTerracini2014} 
and they concern the case when the domain $\Omega$ is the angular sector
\[
\Sigma_{\pi/4} = \left\{ (x_1, x_2) \in \mathbb{R}^2 \, :  \, x_1 > 0, \, 
|x_2| < x_1 \tan \frac{\pi}{8} \, x_1^2 + x_2^2 < 1 \right\}, 
\]
and the square. We also mention the work \cite{BonnaillieHelffer2013} 
treating the case of the unit disk. We remark that all those simulations 
are made in the case of half-integer circulation $\alpha \in \{\tfrac{1}{2}\} 
+ \mathbb{Z}$, since in this case numerical computations can be done for 
eigenfunctions which are in fact real valued functions.

Figure~\ref{fig:sector-square} represents the nine firth magnetic 
eigenvalues for the angular sector and the square, respectively 
when the magnetic pole is moving on the symmetry axis of the sector, 
on the diagonal and on the mediane of the square. We remark that 
therein the points of higher multiplicity correspond to the meeting 
points between the coloured lines, and each coloured line represents 
a different eigenvalue. Next, Figures~\ref{fig:3D-sec} and \ref{fig:3D-square} 
give a three-dimensional vision of the first three 
magnetic eigenvalues in the case of the angular sector, and of the 
four first eigenvalues in the square, respectively.

\begin{figure} 
 \begin{center}
  \subfigure[$a$ moving on the symmetry axis of the angular sector]
  {\includegraphics[scale=0.42]{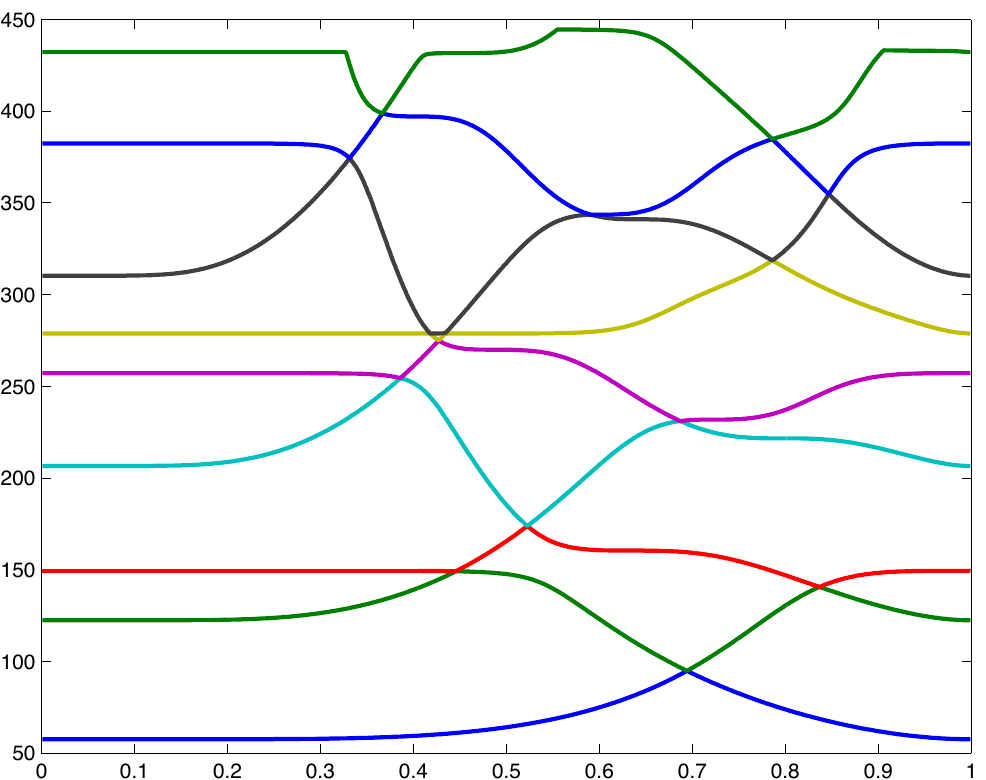}} \\
  \subfigure[$a$ moving on the diagonal of the square]
  {\includegraphics[scale=0.42]{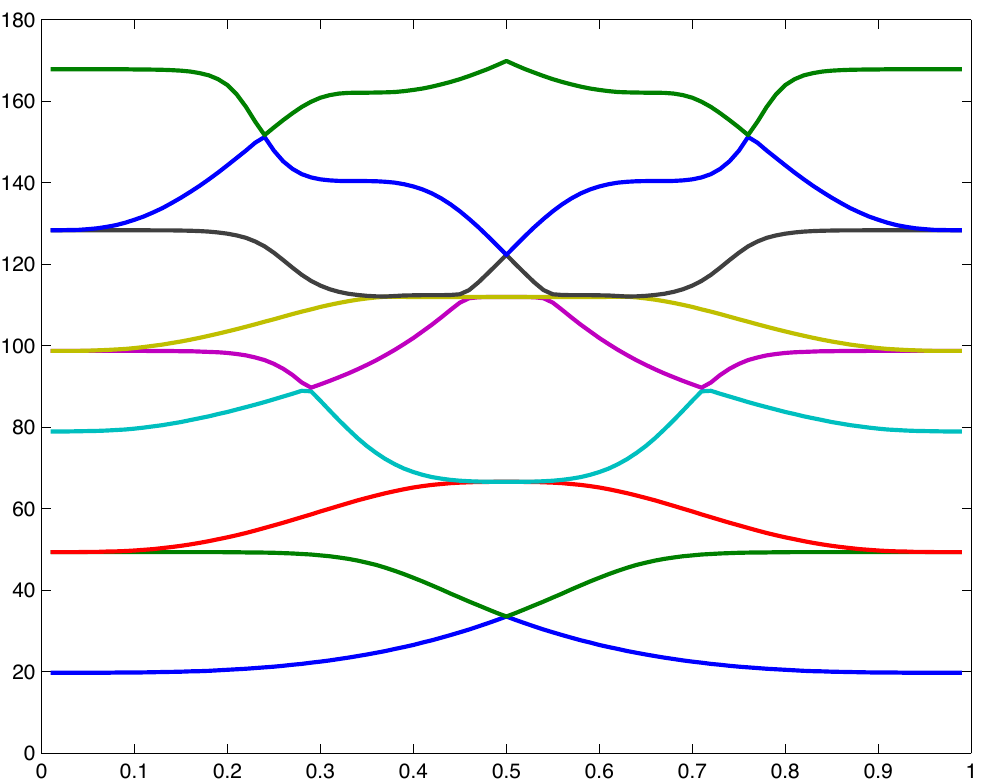}}
  \subfigure[$a$ moving on the mediane of the square]
  {\includegraphics[scale=0.42]{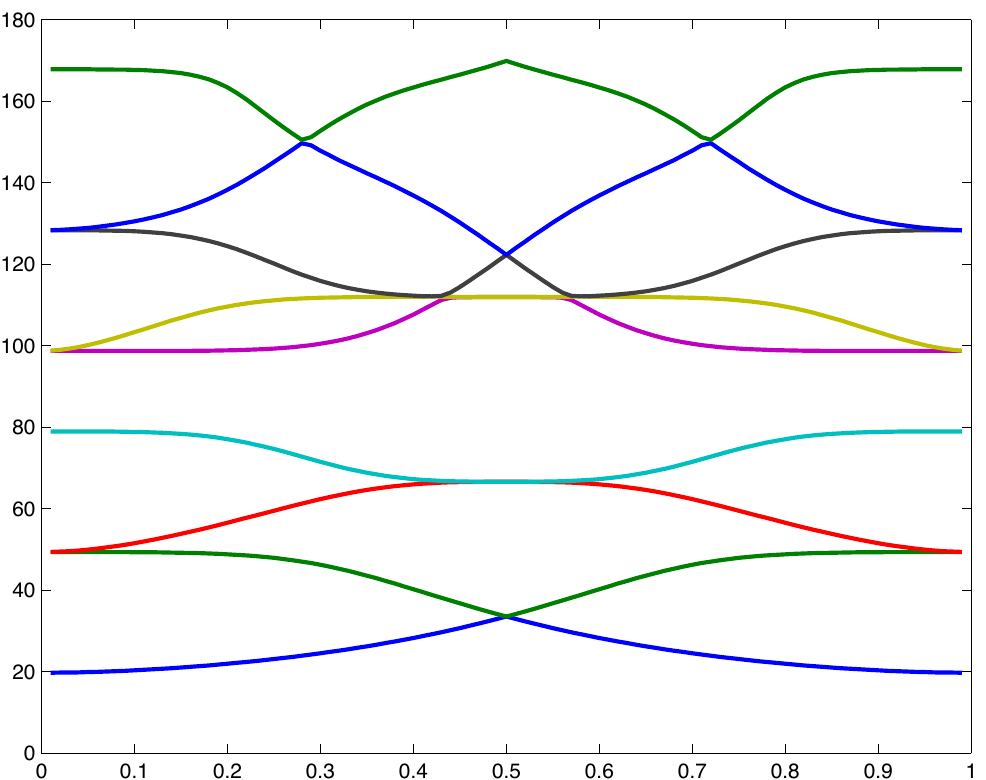}}
  \caption{$a \mapsto \lambda_k^{(a,\frac{1}{2})}$, $k = 1,\ldots,9$, 
  in the angular sector and the square}
  \label{fig:sector-square}
 \end{center}
\end{figure}
\begin{figure}[!ht] 
 \begin{center}
  \subfigure[$a \mapsto \lambda_1^{(a,\frac{1}{2})}$]
  {\includegraphics[scale=0.27]{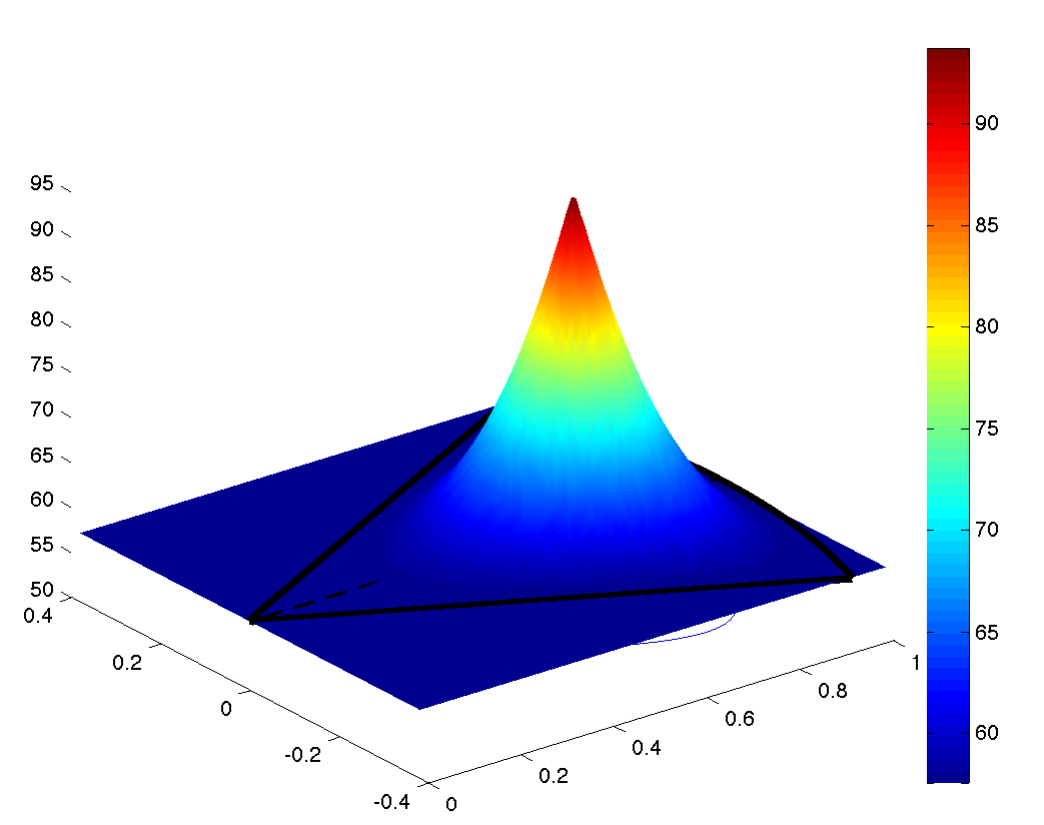}}
  \subfigure[$a \mapsto \lambda_2^{(a,\frac{1}{2})}$]
  {\includegraphics[scale=0.27]{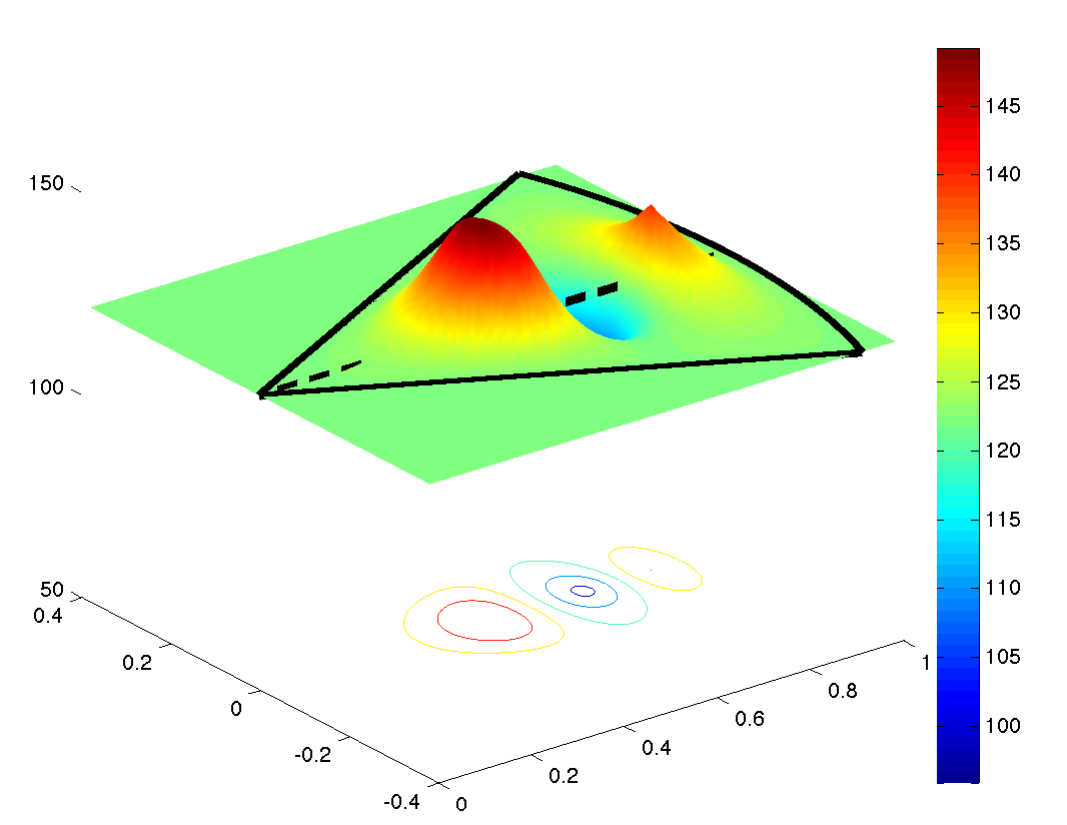}}
  \subfigure[$a \mapsto \lambda_3^{(a,\frac{1}{2})}$]
  {\includegraphics[scale=0.27]{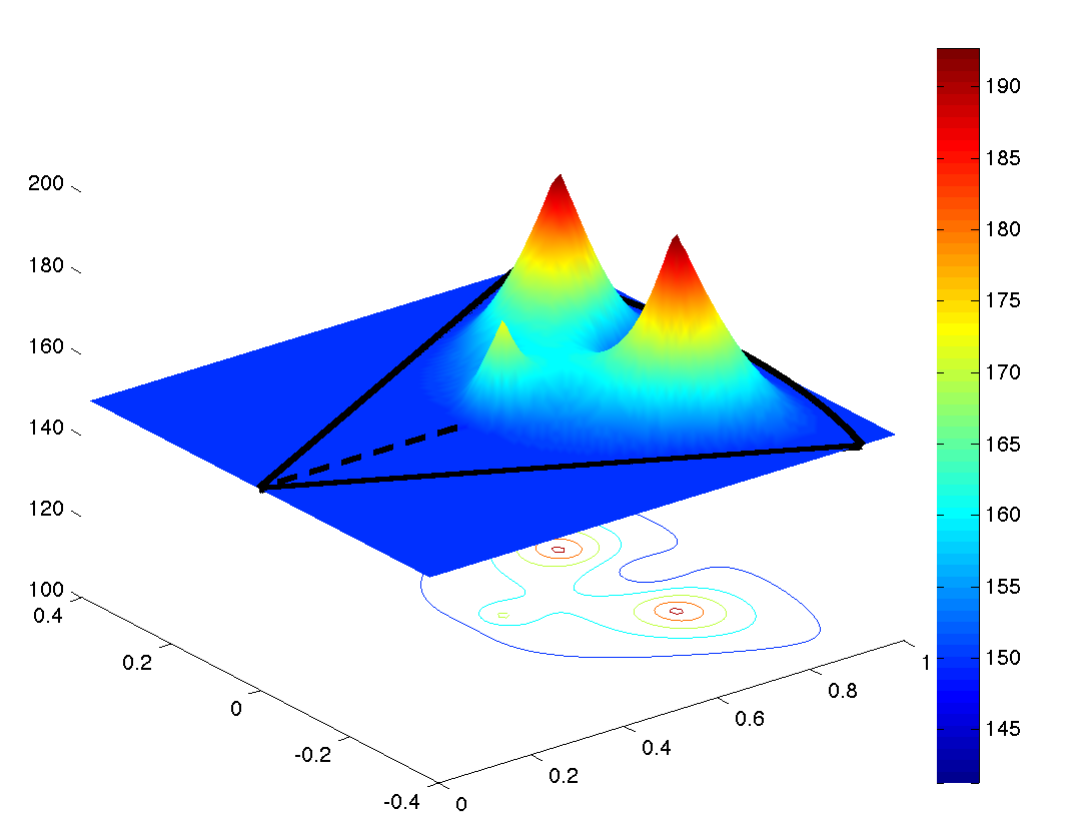}}
  \caption{Three-dimensional vision of $\lambda_k^{(a,\frac{1}{2})}$, 
  $k = 1, \ldots, 3$, for the angular sector}
  \label{fig:3D-sec}
 \end{center}
\end{figure}
\begin{figure} 
 \begin{center}
  \subfigure[$a \mapsto \lambda_1^{(a,\frac{1}{2})}$]
  {\includegraphics[scale=0.25]{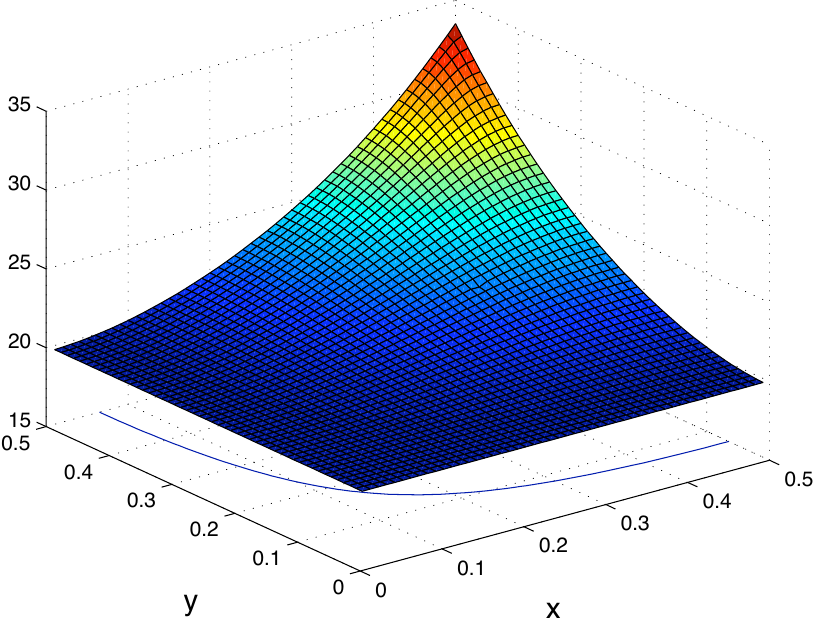}}
  \subfigure[$a \mapsto \lambda_2^{(a,\frac{1}{2})}$]
  {\includegraphics[scale=0.25]{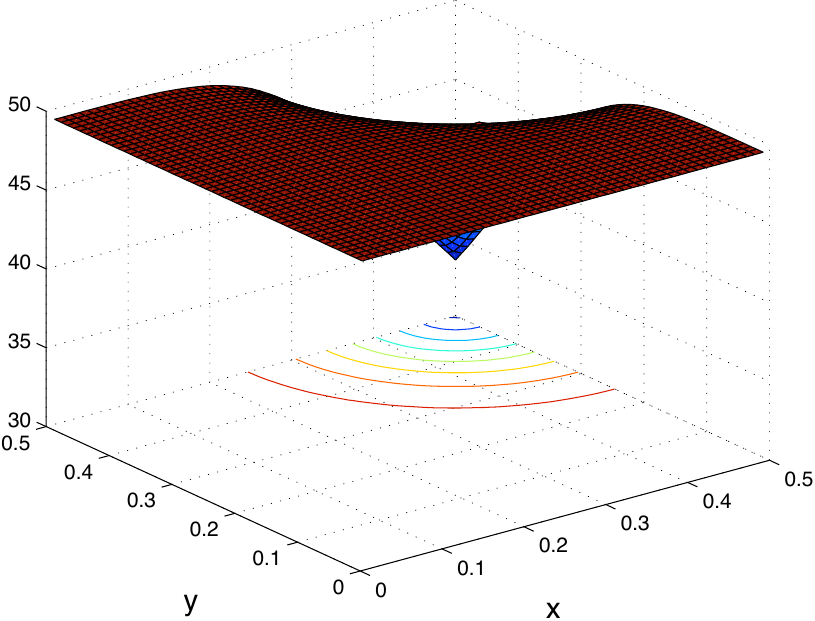}}
  \subfigure[$a \mapsto \lambda_3^{(a,\frac{1}{2})}$]
  {\includegraphics[scale=0.25]{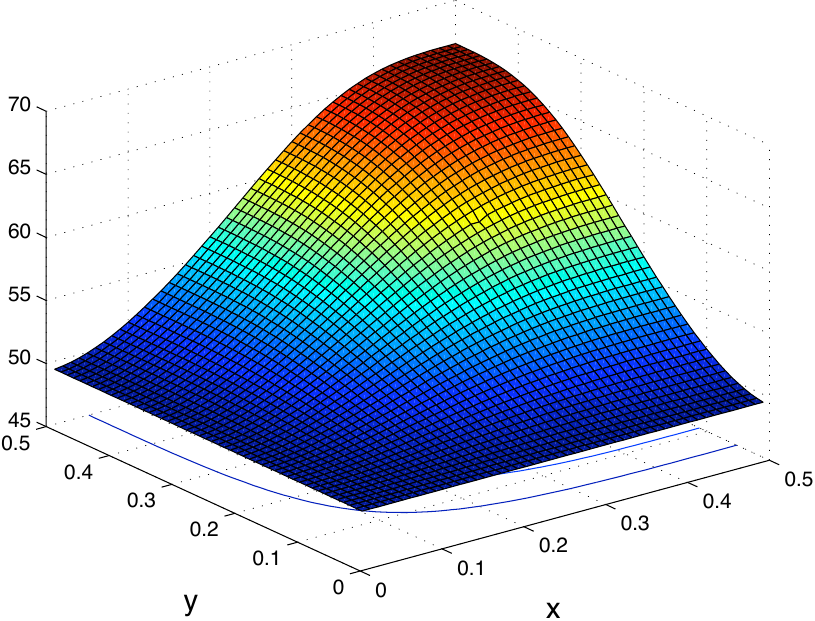}}
  \subfigure[$a \mapsto \lambda_4^{(a,\frac{1}{2})}$]
  {\includegraphics[scale=0.25]{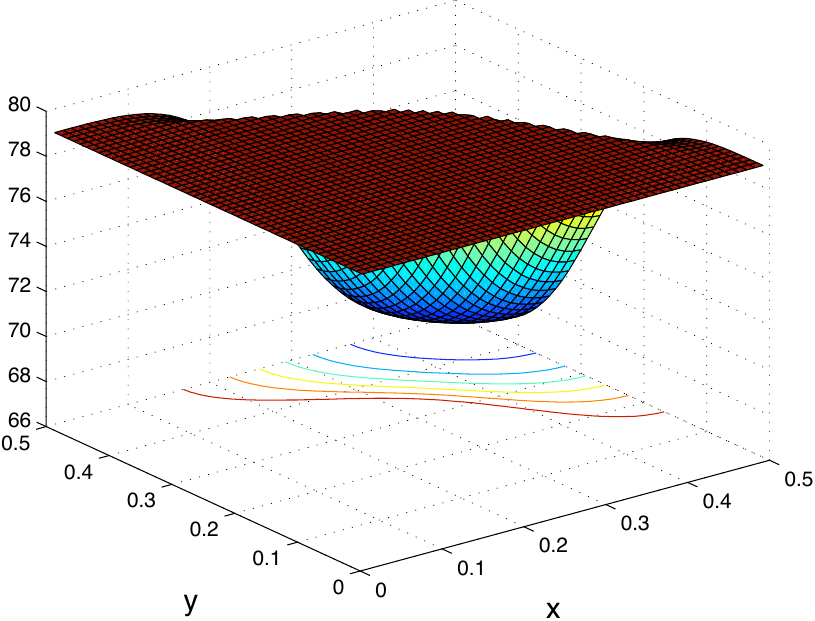}}
  \caption{Three-dimensional vision of $\lambda_k^{(a,\frac{1}{2})}$, 
  $k = 1, \ldots, 4$, for the square}
  \label{fig:3D-square}
 \end{center}
\end{figure}

Since Theorem~\ref{t:main} is related to local properties of the associated 
eigenfunctions by means of conditions $(i)$--$(ii)$, we also present in 
Figure~\ref{fig:graph} the graphs of the nodal set of eigenfunctions in 
the square when the singular pole is at its center. In the case of the disk, 
we refer to \cite[Figures 7 and 8]{BonnaillieHelffer2013}. 

\begin{figure} 
 \begin{center}
  \subfigure[$\varphi_1^{(0,\frac{1}{2})}$]
  {\includegraphics[scale=0.25]{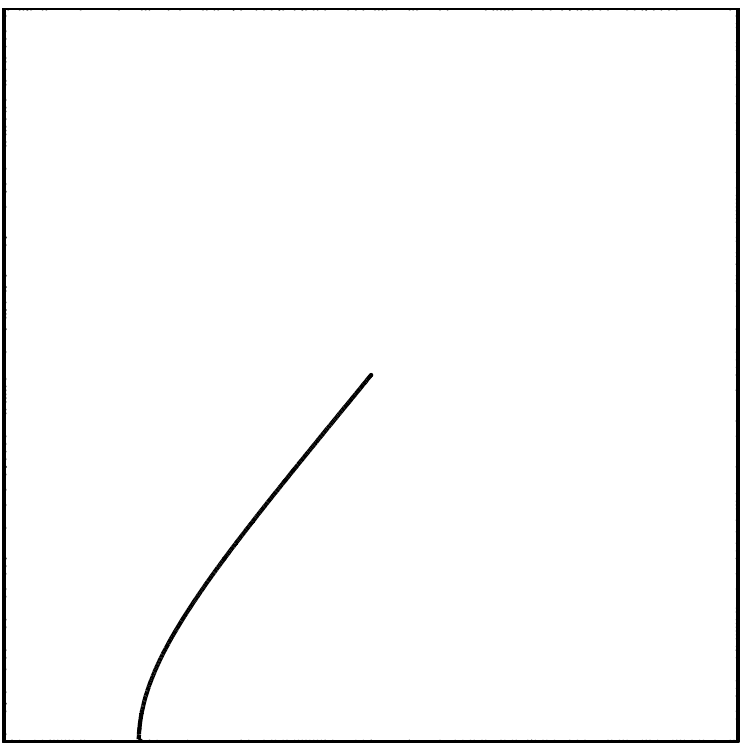}}
  \subfigure[$\varphi_2^{(0,\frac{1}{2})}$]
  {\includegraphics[scale=0.25]{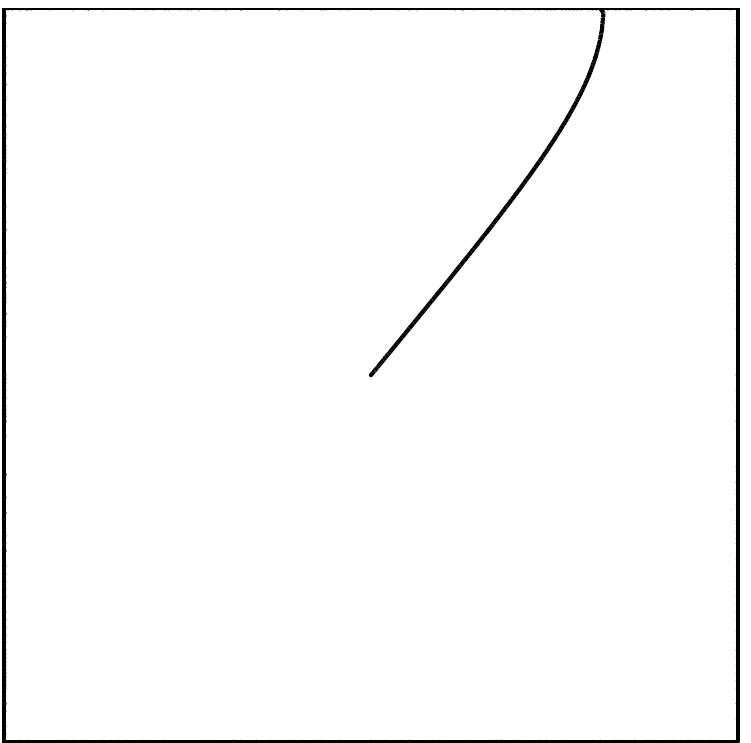}}
  \subfigure[$\varphi_3^{(0,\frac{1}{2})}$]
  {\includegraphics[scale=0.25]{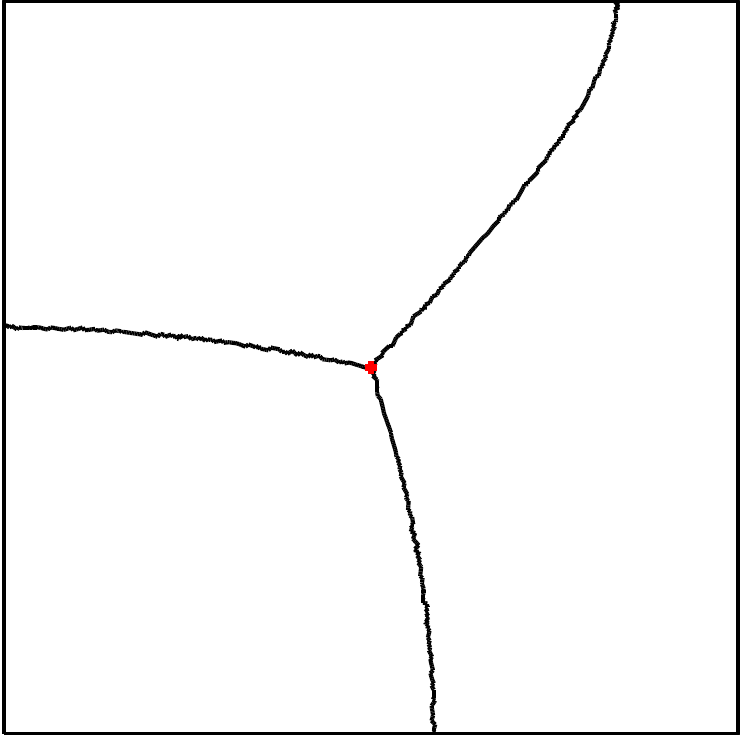}}
  \subfigure[$\varphi_4^{(0,\frac{1}{2})}$]
  {\includegraphics[scale=0.25]{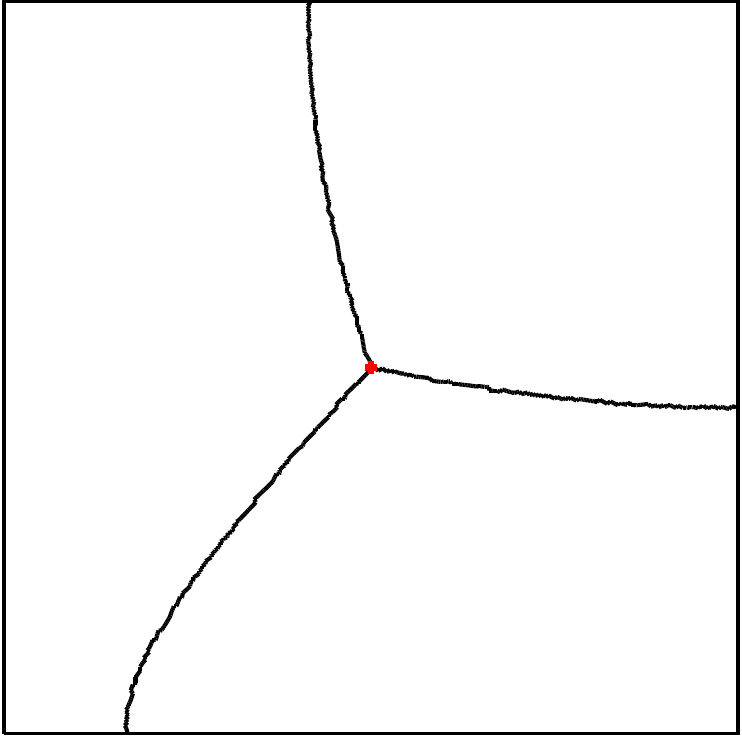}}
  \caption{Nodal set of the first four eigenfunctions in the square with 
  singular pole at the center}
  \label{fig:graph}
 \end{center}
\end{figure}

We present here a collection of observations on these simulations.

\begin{itemize}
\item[(a)] When the singular pole is at the center of the square, or the disk, 
see \cite[Figures 7 and 8]{BonnaillieHelffer2013}, the eigenvalues are always 
of multiplicity two. 

\vspace{0.1cm}

\item[(b)] For the angular sector, the points where eigenvalues are not simple 
correspond to points where they are not differentiable. Indeed, we see in 
Figure~\ref{fig:sector-square}(a) and Figure~\ref{fig:3D-sec} that the meeting 
points present a structure of a singular cone. Moreover, when the same thing 
happens in the square, at those points there is only one nodal line for the 
corresponding eigenfunctions (see Figure~\ref{fig:sector-square}(b) and (c) and 
Figure~\ref{fig:graph}).

\vspace{0.1cm}

\item[(c)] For the square and the disk, if the singular pole is at 
the center, two cases occur. In the first case eigenvalues are not 
differentiable at that point and the corresponding eigenfunctions 
have exactly one nodal line. This is for example the case for the 
first and second eigenvalues, where we note the structure of non 
differentiable cone in Figures~\ref{fig:sector-square}, \ref{fig:3D-square} 
(a) and (b), and the unique nodal line in Figure~\ref{fig:graph} 
(a) and (b). In the second case eigenvalues are differentiable, and 
the corresponding eigenfunctions have more nodal lines. This happens 
for instance for the third and fourth eigenvalues, see Figures~\ref{fig:sector-square}, 
\ref{fig:3D-square} (c) and (d) and Figure~\ref{fig:graph} (c) and (d).

\vspace{0.1cm}

\item[(d)] The two linearly independent eigenfunctions corresponding to 
the same eigenvalue always have the same number of nodal lines ending at 
the singular point; moreover, those lines leave the point in opposite directions, 
therefore never in a tangential way. This can be seen in Figure~\ref{fig:graph} 
and in \cite[Figure 8]{BonnaillieHelffer2013}.

\vspace{0.1cm}

\item[(e)] When considering only variations of the position of the pole, 
it seems that the set 
\[
\{a \in \Omega \, : \, \lambda_k^{(a,\sfrac{1}{2})} \text{ is an eigenvalue 
of } (i \nabla + A_a^{\sfrac{1}{2}})^2 \text{ of multiplicity two} \}
\]
is a finite collection of points in $\Omega$.

\vspace{0.1cm}

\item[(f)] Points of multiplicity higher than two seem not to happen.
\end{itemize}

Observations (b), (c) and (d) suggest that there may be a relation 
between the number of nodal lines of the two eigenfunctions and the 
way the graphs of two subsequent eigenvalues meet at the multiple 
point. Indeed, when they both have one nodal line leaving the point 
in opposite directions, the lines in Figure~\ref{fig:sector-square} 
meet \emph{transversally} and with a non vanishing derivative on the 
two branches, while if there is more than one nodal line, the lines 
in Figure~\ref{fig:sector-square} meet not \emph{transversally} and 
with a vanishing derivative. This reminds us Theorem~\ref{thm:critical}, 
where the criticality of the simple eigenvalues is related to the 
number of nodal lines of the corresponding eigenfunctions. For the 
derivatives at multiple eigenvalues when the domain is perturbed by 
means of a regular vector field, we refer the reader to the book of Henrot 
\cite{Henrot}.

By observation (a), we foresee that all the eigenvalues are double 
because of the strong symmetries of the domain. As well, we think that 
the existence of multiple points such that the eigenvalue is differentiable 
(and then where the eigenfunctions have more than one nodal lines) can 
also be explained by those symmetries.

\medbreak

In view of observation (e), our main Theorem \ref{t:main} provide a 
stronger analysis involving the combined parameters $(a,\alpha)$. 
However, it provides a local result around the point $(0,\tfrac{1}{2})$, 
and there chances to extend it in order to obtain a global result for 
every circulation $\alpha$. The analysis performed in our paper and the 
results achieved leave several other open questions. At the present time, 
we are not able to consider vanishing orders of the eigenfunctions greater 
than $1/2$ (i.e.\ the cases where the eigenfunctions have more than one 
nodal lines), but simulations suggest that multiple points are isolated 
even in these situations. By the way, simulations performed in 
\cite[Section 7]{BonnaillieNorisNysTerracini2014} and \cite{BonnaillieHelffer2011} 
suggest that high orders of vanishing for eigenfunctions may only occur when 
the domain $\Omega$ has strong symmetries (e.g. the disk or the square). 
On the other hand, as far as we know, two linearly independent orthogonal 
eigenfunctions corresponding to the same eigenvalue may have different 
orders of vanishing, but the available simulations do not show such a 
situation. The other fundamental assumption which plays a role in Theorem~\ref{t:main} 
is condition $(ii)$, which prevents the eigenfunctions nodal lines to be tangent. 
How general the assumptions of Theorem~\ref{t:main} can be is currently under 
investigation by the authors.

\medbreak

The paper is organized as follows. We devote Section~\ref{sec:gauge} to illustrate 
the main differences between half-integer and non half-integer circulations. In 
Section~\ref{sec:abstract} we prove the complex version of the abstract theorem 
from \cite{LupoMicheletti1993}. Section~\ref{sec:modified-operator} gives us the 
perturbation of the domain used to obtain new operators with fixed definition domains. 
In Section~\ref{sec:continuity}, we prove Theorem~\ref{thm:reg-a-alpha} which gives 
first a continuity result for the eigenvalues with respect to the combined parameters 
$(a, \alpha)$, and next a regularity results for simple 
eigenvalues, with respect to the same parameters. Section~\ref{sec:equivalent-operator} 
is (with Section~\ref{sec:modified-operator}) the most technical of the paper. We give 
therein the explicit form of the spectrally equivalent operators to $(i \nabla + A_a^\alpha)^2$ 
(and its inverse). Finally, in Section~\ref{sec:proof}, we make all the explicit 
computations using the local asymptotic behavior of the eigenfunctions \eqref{eq:expansion-1/2} 
and use the abstract Theorem to prove our main result Theorem~\ref{t:main}.

\section{Preliminaries}  \label{sec:preli}

\subsection{Functional spaces} \label{subsec:functional}

We notice that throughout the paper (except for Subsection~\ref{subsec:reg}), 
all the Hilbert spaces are \emph{complex} Hilbert spaces, i.e.\ they have complex 
scalar products.

If $\Omega \subset \R^2$ is open, bounded and simply connected, for $a \in \Omega$, 
we define the functional space $H^{1,a}(\Omega,\C)$ as the completion of $\{ u \in 
H^1(\Omega, \C) \cap C^\infty(\Omega, \C) \, : \, u \text{ vanishes in a neighborhood 
of } a \}$ with respect to the norm
\[
\| u \|_{H^{1,a}(\Omega,\C)} := \left( \| \nabla u \|_{L^2(\Omega,\C^2)}^2 
+ \| u \|_{L^2(\Omega,\C)}^2 + \left\| \frac{u}{|x-a|} 
\right\|_{L^2(\Omega,\C)}^2 \right)^{1/2}.
\]
When the circulation $\alpha$ is not an integer, i.e.\ $\alpha \in \R \setminus \Z$, 
the latter norm is equivalent to the norm
\begin{equation*}
\| u \|_{H^{1,a}(\Omega,\C)} = \left( \left\| ( i \nabla + A_a^\alpha ) u 
\right\|^2_{ L^2(\Omega,\C^2) } + \| u \|^2_{ L^2(\Omega,\C) } \right)^{\!\!1/2},
\end{equation*}
in view of the Hardy type inequality proved in \cite{LaptevWeidl1999} (see also 
\cite{Balinsky} and
\cite[Lemma 3.1 and Remark 3.2]{FelliFerreroTerracini2011})
\[
\int_{ D_r(a) } | (i\nabla + A_a^\alpha) u |^2 \,dx \geq \Big( \min_{j \in \Z} 
|j - \alpha| \Big)^2 \int_{ D_r(a) } \frac{ |u(x)|^2 }{|x-a|^2} \, dx,
\]
which holds for all $r > 0$, $a \in \R^2$  and $u \in H^{1,a}(D_r(a),\C)$. Here we 
denote as $D_r(a)$ the disk of center $a$ and radius $r$.

As well, the space $H^{1,a}_0(\Omega, \C)$ is defined as the completion of $C^\infty_c
(\Omega \setminus \{a\}, \C)$ with respect to the norm $\| u \|_{H^{1,a}(\Omega,\C)}$. 
By a Poincaré type inequality, see e.g.\ \cite[A.3]{AbatangeloFelliNorisNys2016}, we 
can consider the following equivalent norm on $H^{1,a}_0(\Omega, \C)$
\[
\| u \|_{H^{1,a}_0(\Omega,\C)} := \left( \| (i \nabla + A_a^\alpha) u 
\|_{L^2(\Omega,\C^2)}^2 \right)^{1/2}.
\] 

Finally, $(H^{1,a}_0(\Omega,\C))^\star$ is the space dual to $H^{1,a}_0(\Omega,\C)$. 
We emphasize that as long as $\alpha$ is not an integer, those spaces are independent 
of $\alpha$.

\subsection{Eigenvalues and eigenfunctions}

We look at the operator defined in \eqref{eq:operator}
\[
(i \nabla + A_a^\alpha)^2 \, : \, H^{1,a}_0(\Omega,\C) \to ( H^{1,a}_0(\Omega,\C))^\star.
\]
In a standard way, for any $u \in H^{1,a}_0(\Omega,\C)$, 
$(i \nabla + A_a^\alpha)^2 u$ acts in the following way
\begin{equation} \label{eq:magn-lap}
\phantom{a}_{( H^{1,a}_0(\Omega,\C))^\star} \left\langle 
(i \nabla + A_a^\alpha)^2 u, v 
\right\rangle_{H^{1,a}_0(\Omega,\C)} 
:= \int_{\Omega} (i \nabla + A_a^\alpha) u \cdot 
\overline{(i \nabla + A_a^\alpha) v}.
\end{equation}
By standard spectral theory the inverse operator 
\[
[(i \nabla + A_a^\alpha)^2]^{-1} \circ \text{Im}_{H^{1,a}_0(\Omega,\C) 
\to ( H^{1,a}_0(\Omega,\C))^\star} \, : \, H^{1,a}_0(\Omega,\C) \to H^{1,a}_0(\Omega,\C)
\]
is compact because of the compactness of the embedding 
$\text{ Im}_{H^{1,a}_0(\Omega,\C) \to ( H^{1,a}_0(\Omega,\C))^\star}$ 
coming from the compact embedding
\[
H^{1,a}_0(\Omega, \C) \hookrightarrow L^2(\Omega, \C),
\]
and the continuity of the immersion $L^2(\Omega,\C) \hookrightarrow 
( H^{1,a}_0(\Omega,\C))^\star$, see e.g. \cite{Salsa}.

We are  considering the eigenvalue problem
\begin{equation}\label{eq:eige_equation_a} \tag{$E_{a,\alpha}$}
\begin{cases}
(i \nabla + A_a^\alpha)^2 u = \lambda u,  &\text{in }\Omega,\\
u = 0, &\text{on }\partial \Omega,
\end{cases}
\end{equation}
in a weak sense, and  we say that $\lambda\in\C$ is an eigenvalue of
problem \eqref{eq:eige_equation_a} if there exists $u\in
H^{1,a}_{0}(\Omega,\C)\setminus\{0\}$ (called eigenfunction) such that
\[
\int_\Omega (i\nabla + A_a^\alpha) u 
\cdot \overline{(i \nabla + A_a^\alpha) v} \, dx 
= \lambda \int_\Omega u \overline{v} \,dx 
\quad \text{ for all } v \in H^{1,a}_{0}(\Omega,\C).
\]
From classical spectral theory (using the self-adjointness of the operator 
and the compactness of the inverse operator), for every $(a,\alpha) \in \Omega 
\times \R$, the eigenvalue problem \eqref{eq:eige_equation_a} admits a diverging 
sequence of real and positive eigenvalues $\{\lambda_k^{(a,\alpha)}\}_{k\geq 1}$ 
with finite multiplicity. In the enumeration
\[
0 < \lambda_1^{(a,\alpha)} \leq \lambda_2^{(a,\alpha)} \leq \dots 
\leq \lambda_k^{(a,\alpha)} \leq \dots
\]
we repeat each eigenvalue as many times as its multiplicity. Those eigenvalues 
also have a variational characterization given by 
\begin{equation} \label{eq:var-char}
\begin{split}
\lambda_k^{(a, \alpha)} = \min \Big\{ \sup_{u \in W_k \setminus\{0\}} 
\frac{ \int_{\Omega} |(i \nabla + A_a^\alpha) u|^2 }{\int_{\Omega} |u|^2 }
\, : \, & W_k \text{ is a linear subspace of } H^{1,a}_0(\Omega, \C), \\
&  \, \text{dim} W_k = k \Big\}.  
\end{split}
\end{equation}
We denote by $\varphi_k^{(a,\alpha)} \in H^{1,a}_0(\Omega,\C)$ the 
corresponding eigenfunctions orthonormalized in $L^2(\Omega,\C)$. 
We note that if $\varphi_k^{(a,\alpha)}$ is an eigenfunction of 
$(i \nabla + A_a^\alpha)^2$ of eigenvalue $\lambda_k^{(a,\alpha)}$, 
it is also an eigenfunction of $[(i \nabla + A_a^\alpha)^2]^{-1} 
\circ \text{ Im}_{H^{1,a}_0(\Omega,\C) \to ( H^{1,a}_0(\Omega,\C))^\star}$ with 
eigenvalue $(\lambda_k^{(a,\alpha)})^{-1}$.

\section{The gauge invariance} \label{sec:gauge}

Among all the circulations $\alpha \in \R$, the case $\alpha \in \{1/2\} + \Z $ 
presents very special features. For the reader's convenience, in this Section we 
are recalling some basic facts about eigenfunctions of Aharonov--Bohm operators. We 
gain them partially as they are stated in \cite[Section 3]{AbatangeloFelliLena2016}.

\subsection{General facts on the gauge invariance} \label{subsec:general}

\begin{definition}\label{def:gauge-invariance}
We call \emph{gauge function} a smooth complex valued function $\psi \, : \, \Omega 
\to \C$ such that $|\psi| \equiv 1$. To any gauge function $\psi$, we associate a 
\emph{gauge transformation} acting on the pairs magnetic potential -- function as 
$(A,u) \mapsto (A^*,u^*)$, with
\begin{align*}
& A^* = A + i \frac{\nabla\psi}{\psi}, \\
& u^* = \overline{\psi} u,
\end{align*}
where $\nabla \psi = \nabla(\Re\psi)+i \nabla(\Im\psi)$. We notice that since $|\psi| 
= 1$, $i \frac{\nabla\psi}{\psi}$ is a real vector field. Two magnetic potentials are 
said to be gauge equivalent if one can be obtained from the other by a gauge transformation 
(this is an equivalence relation).
\end{definition}

The following result is a consequence, see \cite[Theorem 1.2]{Leinfelder1983}.
\begin{proposition}\label{prop:unieq}
If $A$ and $A^*$ are two gauge equivalent vector potentials, the operators 
$(i\nabla + A)^2$ and $(i\nabla + A^*)^2$ are unitarily equivalent, that 
is
\[
\overline{\psi} \, (i \nabla + A)^2 \, {\psi} = (i \nabla + A^*)^2.
\]
\end{proposition}
We immediately see that if $A$ and $A^*$ are gauge equivalent, then the corresponding 
operators are spectrally equivalent, i.e.\ they have the same spectrum, and in particular 
they have the same eigenvalues with the same multiplicity. The equivalence between two 
vector potentials (which is equivalent to the fact that their difference is gauge equivalent 
to $0$) can be determined using the following criterion.

\begin{lemma}\label{lemma:EquivPot}
Let $A$ be a vector potential in $\Omega$. It is gauge equivalent to $0$ 
if and only if
\[
\frac{1}{2\pi}\oint_{\gamma} A(s) \cdot d\mathbf{s} \in \Z
\]  
for every closed path $\gamma$ contained in $\Omega$.
\end{lemma}

Whenever the vector potential $A$ is gauge equivalent to 0, i.e.\ there is 
a gauge function $\psi$ such that $A = - i \frac{\nabla \psi}{\psi}$, we 
can define the antilinear antiunitary operator $K$ by
\begin{equation} \label{eq:K-A}
K u= \psi \bar{u}.
\end{equation}

\begin{definition}
We say that a function $u \in L^2(\Omega,\C)$ is \emph{$K$-real} when $K u = u$.
\end{definition}

\subsection{Aharonov--Bohm potentials}

When the circulation $\alpha = n$ is an integer, i.e.\
\[
\frac{1}{2\pi} \oint_{\gamma} A_a^n \cdot d\mathbf{s} \in \mathbb{Z},
\]
for any closed path $\gamma$ contained in $\Omega$, it directly holds 
by Lemma~\ref{lemma:EquivPot} that $A_a^n$ is gauge equivalent to $0$. 
Moreover, in that particular case, we can give an explicit expression 
to the gauge function of Definition~\ref{def:gauge-invariance}. For any 
$a \in \mathbb{R}^2$, we define $\theta_a \, : \, \mathbb{R}^2 \setminus 
\{ a\} \to [0,2\pi)$ the polar angle centred at $a$ such that
\begin{equation} \label{eq:polar-angle-a}
\theta_a(a + r( \cos t, \sin t)) = t, \quad \text{ for } t \in [0, 2 \pi).
\end{equation}
We remark that such an angle is regular except on the half-line
\[
\{ (x_1, x_2) \in \mathbb{R}^2 \, : \, x_2 = a_2, \, x_1 > a_1 \}.
\]
From relation \eqref{eq:polar-angle-a} we immediately observe that for 
any $n \in \mathbb{Z}$
\begin{equation} \label{eq:gauge-n}
A_a^n = -i e^{- i n \theta_a} \nabla e^{i n \theta_a} = n \nabla \theta_a,
\end{equation}
almost everywhere. Therefore the gauge function is given by the phase 
$e^{i n \theta_a}$ and such a phase is well defined and smooth thanks 
to the fact that the circulation $\alpha = n$ is an integer. 
Proposition~\ref{prop:unieq} tells us then that, for any $n \in \mathbb{Z}$, 
$(i\nabla + A_a^n)^2$ and $-\Delta$ are unitarily equivalent, i.e.\ 
the spectrum of $(i \nabla + A_a^n)^2$ coincides with the spectrum of $-\Delta$.

Moreover the same gauge transformation tells us that, for any $\alpha \in (0,1)$ 
and $n \in \mathbb{Z}$, $(i \nabla + A_a^\alpha)^2$ and $(i \nabla + A_a^{\alpha + n})^2$ 
are unitarily equivalent, i.e.\ the spectrum of $(i \nabla + A_a^\alpha)^2$ coincides 
with the spectrum of $(i \nabla + A_a^{\alpha + n})^2$. Those observations tell us 
that it is sufficient to consider magnetic potentials with circulations $\alpha \in 
(0,1)$ since the other ones can be recovered from them, and $\alpha$ integer does 
not present any interest. This will be the case in the rest of the paper.

\medbreak

However, amoung the circulations $\alpha \in (0,1)$, the case $\alpha = 1/2$ 
presents special features. We refer to 
\cite{FelliFerreroTerracini2011,HHOO1999} for details. For any magnetic potential 
$A_a^{\sfrac{1}{2}}$ defined in \eqref{eq:magnetic-pot} it holds that
\[
\frac{1}{2\pi} \oint_{\gamma} 2 A_a^{\sfrac{1}{2}} \cdot d \mathbf{s} 
= \frac{1}{2\pi} \oint_{\gamma} A_a^{1} \cdot d \mathbf{s} 
= 1 \in \mathbb{Z},
\] 
for any closed path $\gamma$ containing $a$, so that, by Lemma~\ref{lemma:EquivPot}, 
$2 A_a^{\sfrac{1}{2}}$ is gauge equivalent to $0$. Therefore, 
by Definition~\ref{def:gauge-invariance} and \eqref{eq:polar-angle-a}--\eqref{eq:gauge-n}
\[
2 A_a^{\sfrac{1}{2}} = - i e^{- i \theta_a} \nabla e^{i \theta_a} = \nabla \theta_a.
\]

We write the antilinear and antiunitary operator of \eqref{eq:K-A}, which 
depends on the position of the pole $a\in \Omega$ through the angle $\theta_a$, 
as
\begin{equation} \label{eq:K-a}
K_a u = e^{i \theta_a} \overline{u}.
\end{equation}
For all $u \in C^{\infty}_0({\Omega\setminus\{a\}},\C)$ we have 
\begin{equation} \label{eq:relation-1}
\begin{split}
( i \nabla + A_a^{\sfrac{1}{2}} ) (K_a u) 
& = (i \nabla + A_a^{\sfrac{1}{2}}) (e^{i \theta_a} \bar{u}) 
= e^{i \theta_a} \left( i \nabla + i \frac{\nabla e^{i \theta_a} }{e^{i \theta_a}} 
+ A_a^{\sfrac{1}{2}} \right) \bar{u} \\
& = e^{i \theta_a} \left( i \nabla - A_a^{\sfrac{1}{2}} \right) \bar{u} 
= - e^{i \theta_a} \overline{ \left( i \nabla + A_a^{\sfrac{1}{2}} \right) u } 
= - K_a ((i \nabla + A_a^{\sfrac{1}{2}}) u),
\end{split}
\end{equation}
and therefore $(i \nabla + A_a^{\sfrac{1}{2}})^2$ and $K_a$ commute
\begin{equation} \label{eq:commutation}
(i \nabla + A_a^{\sfrac{1}{2}})^2 \circ K_a = K_a \circ (i \nabla + A_a^{\sfrac{1}{2}})^2.
\end{equation}
Let us denote 
\[
L^{2}_{K_a}(\Omega,\C) := \{ u \in L^2(\Omega,\C) \, : \, K_a u = u \}.
\] 
The restriction of the scalar product to $L^2_{K_a}(\Omega,\C)$ gives 
it the structure of a \emph{real} Hilbert space, instead of a complex 
space. Relation \eqref{eq:commutation} implies that $L^2_{K_a}(\Omega,\C)$ 
is stable under the action of $(i\nabla + A_a^{\sfrac{1}{2}})^2$; we denote 
by $(i\nabla+A_a^{\sfrac{1}{2}})^2_{L^2_{K_a}(\Omega,\C)}$ the restriction 
of $(i\nabla+A_a^{\sfrac{1}{2}})^2$ to $L^2_{K_a}(\Omega,\C)$, which is a 
\emph{real} operator. There exists an orthonormal basis of $L^2_{K_a}(\Omega,\C)$ 
formed by eigenfunctions of $(i\nabla+A_a^{\sfrac{1}{2}})^2_{L^2_{K_a}(\Omega,\C)}$.

We notice that if $u \in C^\infty_0(\Omega \setminus \{a\}, \C)$ is $K_a$-real, 
i.e.\ $K_a u = u$, relation \eqref{eq:relation-1} becomes
\begin{equation} \label{eq:relation-2}
( i \nabla + A_a^{\sfrac{1}{2}} ) u = - e^{i \theta_a} 
\overline{(i \nabla + A_a^{\sfrac{1}{2}}) u}.
\end{equation}

Being allowed to consider $K_a$-real eigenfunctions of $(i \nabla + A_a^{\sfrac{1}{2}})^2$ 
means to work with the \emph{real} operator $(i \nabla + A_a^{\sfrac{1}{2}})^2_{L^2_{K_a}(\Omega, \C)}$ 
in the \emph{real} space $L^2_{K_a}(\Omega, \C)$. This leads to the special characterisation 
of the eigenfunctions for $\alpha = 1/2$ mentioned in \eqref{eq:expansion-1/2}. Indeed, 
let $u$ be a $K_a$-real eigenfunction of $(i \nabla + A_a^{\sfrac{1}{2}})^2$ of eigenvalue 
$\lambda$. If we consider the double covering manifold, already introduced in \cite{HHOO1999}, 
(where we use the equivalence $\R^2 \cong \C$) given by
\[
\Omega_a := \{ y \in \C  \, : \,  y^2 + a = x \in \Omega \},
\]
and if we define for $y \in \Omega_a$ the function
\begin{equation} \label{eq:weighted-eigenf}
v(y) := e^{-i \frac{\theta_a}{2}(y^2 + a)} u(y^2 + a) = e^{- i \theta_0(y)} u(y^2 + a),
\end{equation}
we have that $v$ is well defined in $\Omega_a$ since
\[
e^{-i \frac{\theta_a}{2}(y^2 + a)} = e^{- i \theta_0(y)} 
\text{ is well defined on } \Omega_a.
\]
Morever, $v$ is \emph{real} (this comes directly from the $K_a$-reality of $u$) 
and it is a weighted eigenfunction of $- \Delta$ in $\Omega_a$, i.e.\
\[
- \Delta v = 4 |y|^2 \lambda v \quad \text{ in } \Omega_a.
\] 
To this aim see also \cite[Lemma 2.3]{BonnaillieNorisNysTerracini2014} and references therein.

Finally, from \eqref{eq:weighted-eigenf} it follows that $v$ is antisymmetric 
with respect to the transformation $y \mapsto - y$. From the above facts, we 
conclude that the nodal set of $u$, $u^{-1}(\{0\})$ (which coincides with the 
nodal set of $v$), is made of curves. Moreover, from the antisymmetry of $v$, 
we deduce that $u$ always has an odd number of nodal lines at $a$, and then at 
least one. As showed in \cite[Theorem 6.3]{FelliFerreroTerracini2011} if we 
denote by $h \in \mathbb{N}$, $h$ odd, the number of nodal lines of $u$ ending at 
$a \in \Omega$, there exist $c$, $d \in \mathbb{R}$ with 
$c^2 + d^2 \neq 0$, and
\[
r^{- h/2} u(a + r(\cos t, \sin t)) \to e^{i \frac{t}{2}} \left( c \cos \frac{h t}{2} 
+ d \sin \frac{h t}{2} \right),
\]
as $r \to 0^+$ in $C^{1,\tau}([0,2\pi], \C)$ for any $\tau \in (0,1)$. 
Similarly, we can write
\[
u( a + r(\cos t, \sin t)) = r^{h/2} e^{i \frac{t}{2}} \left( c \cos \frac{h t}{2} 
+ d \sin \frac{h t}{2} \right) + f(r, t), 
\]
as $r \to 0^+$, where $f(r,t) = O(r^{h/2 + 1})$ uniformly in $t \in [0, 2 \pi]$. 
The fact that $c$ and $d \in \mathbb{R}$ comes from the $K_a$-reality of $u$.

\medbreak

When $A_a^\alpha$ defined in \eqref{eq:magnetic-pot} has circulation 
$\alpha \in (0,1) \setminus\{1/2\}$, we still have that
\[
\frac{1}{\alpha} A_a^\alpha = - i e^{- i \theta_a} \nabla e^{i \theta_a},
\]
for $\theta_a$ defined in \eqref{eq:polar-angle-a}. However, the main difference 
is that we loose the commutation property between $(i \nabla + A_a^\alpha)^2$ and 
$K_a$ (given in \eqref{eq:K-a}). This means that we cannot consider a basis of 
$K_a$-real eigenfunctions of $L_{K_a}^2(\Omega, \C)$ and we must consider 
$(i \nabla + A_a^\alpha)^2$ as a complex operator, and the special expression 
\eqref{eq:expansion-1/2} with \emph{real} coefficients does not hold.

\section{Abstract result} \label{sec:abstract}

In order to prove our results, we follow the strategy of \cite{LupoMicheletti1993}.
The proof therein relies on a strong abstract result obtained by means of 
transversality theorems, see e.g.\ \cite[p28]{GuilleminPollackBook}. We need a 
slightly different version of their abstract theorem, that we enounce and prove here.

\begin{theorem}\label{theorem:abstract-complex}
Consider a Banach space $B$ and a complex Hilbert space $X$ with scalar product 
$(\cdot,\cdot)_X$. Let $V$ be a neighborhood of 0 in $B$ and consider a family 
of self-adjoint compact linear operators $\mathcal{T}_b \, : \, X \to X$ 
parametrized in $V$. Let $\lambda_0$ be an eigenvalue of $\mathcal{T}_0$ of 
multiplicity $\mu > 1$ and denote by $\{x_1^0, \ldots, x_\mu^0\}$ an orthonormal 
system of eigenvectors associated to $\lambda_0$. Suppose that the following two 
conditions hold
\begin{itemize}
\item[$(i)$] the map $b \mapsto \mathcal{T}_b$ is $C^1$ in $V$;
\item[$(ii)$] the application $F \, : \, B \to L_h(\R^n,\R^n)$ defined as 
$b \mapsto ( ( \mathcal{T}'(0)[b]) x_i^0, x_j^0)_{X \, \{i,j=1,\ldots,\mu\}}$ is such that 
\[
\mathrm{Im} F + [I] =  L_h(\R^\mu,\R^\mu).
\]
\end{itemize}
Then 
\[
\{ b \in B \, : \, \text{ there exists } \lambda_b \text{ eigenvalue of } 
\mathcal{T}_b \text{ near } \lambda_0 \text{ of multiplicity } \mu \}
\]
is a manifold in $B$ of codimension $\mu^2 - 1$.
\end{theorem}

\begin{proof}[Proof of Theorem~\ref{theorem:abstract-complex}]
For sake of clarity we sketch here the proof of the result, which follows the 
guidelines of its real counterpart contained in \cite[Theorem 1]{LupoMicheletti1993}. 

Let us denote by $L_h(X,X)$ the set of all linear continuous hermitian operators 
from $X$ to $X$ and by $\Phi_0(X,X)$ the set of all Fredholm operators of index 
$0$. We define the set 
\[
F_{\mu,\mu} := \{ L \in \Phi_0 \cap L_h(X,X) \, : \, \textrm{dim}\,\textrm{ker}L = 
\textrm{codim}\,\textrm{rk}L = \mu  \}.
\]
We remind that for any $L \in \Phi_0$ we have $X = \textrm{ker}L \oplus 
\textrm{rk}L $. We now fix $L_0 \in F_{\mu,\mu}$ and define the orthogonal projections 
\[
P \, : \, X \to \textrm{ker}L_0 \qquad Q \, : \, X \to \textrm{rk}L_0
\]
and the spaces 
\begin{align*}
V_0 := \{ H \in L_h(X,X) \,: \,  H(\textrm{ker}L_0) \subset \textrm{rk}L_0 \}
\qquad M := \{ PHP \in L_h(X,X) \, : \, H \in L_h(X,X) \}.
\end{align*}
Then \cite[Lemma 1]{LupoMicheletti1993} applies providing $L_h(X,X) = V_0 \oplus M$.

\begin{remark}
We remark that $\mathrm{dim} M = \mu^2$ since $L_h(X,X)$ is a \emph{real} vector 
space and not a complex vector space, and not $\mu(\mu+1)/2$ as in 
\cite[Remark 1]{LupoMicheletti1993}, where real Hilbert spaces $X$ were considered. 
\end{remark}

Then \cite[Lemma 2]{LupoMicheletti1993} applies providing that $F_{\mu,\mu}$ is 
an analytic manifold of real codimension $\mu^2$ in $L_h(X,X)$; moreover, the tangent 
plane at $L_0$ to $F_{\mu,\mu}$ is $V_0$. 

The next step consists in \cite[Lemma 3]{LupoMicheletti1993}. Let $L_0 \in F_{\mu,\mu}$ be
such that $\mathrm{ker}L_0$ is not contained in $\mathrm{rk}L_0$. Then there exists 
a neighborhood $U$ of $L_0$ in $L_h(X,X)$ such that
\[
\widetilde F_{\mu,\mu}:= \{ L + \lambda I_X \, : \, L \in U \cap F_{\mu,\mu},\ \lambda\in \R \}
\]
is a manifold, the tangent plane at $L_0$ to $\widetilde F_{\mu,\mu}$ is $V_0 \oplus [I_X]$
and the codimension of $\widetilde F_{\mu,\mu}$ in $L_h(X,X)$ is $\mu^2-1$. The proof follows 
as in \cite[Lemma 3]{LupoMicheletti1993}.

\begin{remark} 
We stress that in the definition of $\widetilde F_{\mu,\mu}$ we need $\lambda$ to be real in 
such a way that $\widetilde F_{\mu,\mu}$ is a subspace in $L_h(X,X)$ (indeed, if $\lambda\in 
\C \setminus \R$ then $\lambda I_X$ is not hermitian). This produces $-1$ in the codimension 
of $\widetilde F_{\mu,\mu}$ in $L_h(X,X)$ since $V_0$ does not contain $I_X$.
\end{remark}

The manifold $\widetilde F_{\mu,\mu}$ is the analytic manifold to which we can apply the 
transversality theorem, see e.g.\ \cite{GuilleminPollackBook}. The fundamental assumption 
$(ii)$ in Theorem~\ref{theorem:abstract-complex} implies that the map $b \mapsto \mathcal{T}_b$ 
is transversal at $0 \in B$ to $\widetilde F_{\mu,\mu}$, i.e. 
\[ 
\mathrm{rk}\mathcal{T}'(0) + \left([I_X] \oplus V_0 \right) = L_h(X,X).
\]
The end of the proof of Theorem~\ref{theorem:abstract-complex} follows as in 
\cite[Theorem 1]{LupoMicheletti1993}. 
\end{proof}

\section{The modified operator} \label{sec:modified-operator}

\subsection{The local perturbation}

Let us fix $\overline{R} > 0$ sufficiently small and such that 
$D_{2 \overline{R}}(0) \subset \Omega$. Let $\xi \in C^{\infty}
_c(\R^2)$ be a cut-off function such that
\begin{equation}\label{eq:xi}
0 \leq \xi \leq 1,\quad \xi \equiv 1 
\text{ on } D_{\overline{R}}(0), 
\quad \xi \equiv 0 
\text{ on } \R^2\setminus D_{2 \overline{R}}(0), 
\quad |\nabla \xi| \leq \frac4{\overline{R}} 
\text{ on }\R^2.
\end{equation}
We define for $a \in D_{\overline{R}}(0)$ the local transformation 
$\Phi_a \in C^\infty (\R^2, \R^2)$ by 
\begin{equation}\label{eq:Phi_a}
\Phi_a (x) = x + a \xi(x).
\end{equation}
Notice that $\Phi_a(0) = a$ and that $\Phi_a'$ is a perturbation 
of the identity 
\[
\Phi_a' = I + a \otimes \nabla\xi = 
\begin{pmatrix}
1 + a_1 \frac{\partial \xi}{\partial x_1} 
& a_1 \frac{\partial \xi}{\partial x_2} \\
a_2 \frac{\partial \xi}{\partial x_1} 
& 1 + a_2 \frac{\partial \xi}{\partial x_2}
\end{pmatrix},
\] 
so that 
\begin{equation}\label{eq:det}
J_a(x):=\det(\Phi_a')= 1 + a_1 \frac{\partial \xi}{\partial x_1} 
+ a_2 \frac{\partial \xi}{\partial x_2}
= 1 + a \cdot \nabla \xi.
\end{equation}
Let $R = \overline{R} /32$. Then, if $a \in D_R(0)$, $\Phi_a$ 
is invertible, its inverse $\Phi_a^{-1}$ is also $C^\infty(\R^2, 
\R^2)$, see e.g. \cite[Lemma 1]{Micheletti1972}, and it can be 
written as
\begin{equation}\label{eq:inversePhi_a}
\Phi_a^{-1}(y) = y - \eta_a(y).
\end{equation}
Moreover, from \eqref{eq:Phi_a} and \eqref{eq:inversePhi_a} we 
have
\[
\eta_a(y) = a \, \xi (y - \eta_a(y)) 
\quad \text{ or equivalently } \quad 
\eta_a( \Phi_a(x)) = a \, \xi(x).
\]
From this relation we deduce that
\begin{equation} \label{eq:derivative-eta-a}
\frac{\partial \eta_{a,i}}{\partial y_j}(\Phi_a(x)) 
= \frac{1}{J_a(x)} a_i \frac{\partial \xi}{\partial x_j}(x).
\end{equation}

\begin{lemma}\label{lemma:smooth-map}
Let $J_a$ be defined as in \eqref{eq:det}. The maps $(a, \alpha) 
\mapsto J_a$, $(a,\alpha) \mapsto \sqrt{J_a}$ and $(a, \alpha) 
\mapsto 1/\sqrt{J_a}$ are of class $C^\infty(D_R(0) \times (0,1), 
C^\infty(\R^2))$.
\end{lemma}

\begin{proof}
We first notice that $J_a$ does not depend on the variable $\alpha$. 
Therefore we only need to study the regularity with respect to $a$. 
By \eqref{eq:det}, we read that $J_a$ is a polynomial in the variable 
$a$, whose coefficients are $C^\infty(\R^2)$. Thus $J_a$ is an analytic 
function with respect to $a$ into the space $C^\infty(\R^2)$. Moreover, 
as $a \in D_R(0)$, there exists a positive constant $C$ such that $J_a(x) 
\geq C > 0$ uniformly with respect to $x \in \R^2$. This implies that 
even $a \mapsto \sqrt{J_a}$ and $a \mapsto 1/\sqrt{J_a}$ are of class 
$C^\infty(D_R(0), C^\infty(\R^2))$. From this we conclude.
\end{proof}

\subsection{The perturbed operator}

\begin{lemma} \label{lemma:new-operator}
Let $(a,\alpha) \in D_R(0) \times (0,1)$. If $u \in H^{1,a}_0(\Omega,\C)$, then 
$v = u \circ \Phi_a \in H^{1,0}_0(\Omega,\C)$, and the following relation holds
\begin{equation} \label{eq:L}
[(i \nabla + A_a^\alpha)^2 u ] \circ \Phi_a 
= [ (i \nabla + A_0^\alpha)^2 + \mathcal{L}_{(a,\alpha)} ]
(u \circ \Phi_a),
\end{equation}
where the operator $[ (i \nabla + A_a^\alpha)^2 u ] \circ \Phi_a \, : \, 
H^{1,0}_0(\Omega,\C) \to ( H^{1,0}_0(\Omega,\C) )^\star$ is defined by acting as 
\begin{equation} \label{eq:compos}
\begin{split}
& \phantom{a}_{(H^{1,0}_0(\Omega,\C))^\star}\left\langle 
[ (i \nabla + A_a^\alpha)^2 u ] \circ \Phi_a , 
w \right\rangle_{H^{1,0}_0(\Omega,\C)} \\
& \qquad \qquad \qquad \qquad \qquad
= \phantom{a}_{(H^{1,a}_0(\Omega,\C))^\star}\left\langle
(i \nabla + A_a^\alpha)^2 u, (w J_a^{-1}) \circ \Phi_a^{-1} 
\right\rangle_{H^{1,a}_0(\Omega,\C)}
\end{split}
\end{equation}
and 
where the linear operator $\mathcal{L}_{(a,\alpha)} \, : \, H^{1,0}_0(\Omega, \C) 
\to ( H^{1,0}_0(\Omega, \C) )^\star$ acts as
\begin{align*}
& \phantom{a}_{(H^{1,0}_0(\Omega,\C))^\star}\left\langle 
[ (i \nabla + A_0^\alpha)^2 + \mathcal{L}_{(a,\alpha)} ] v, 
w \right\rangle_{H^{1,0}_0(\Omega,\C)} \\
& = \int_{\Omega} \left[ (i \nabla + A_0^\alpha) v + F(a,\alpha) v \right] \cdot 
\overline{\left[ (i \nabla + A_0^\alpha) w + F(a,\alpha) w 
- i J_a^{-1} w \nabla J_a + i J_a^{-2} w ( a \cdot \nabla J_a ) \nabla \xi \right] },
\end{align*}
where
\begin{equation}\label{eq:F}
F(a,\alpha) v := (A_a^\alpha \circ \Phi_a - A_0^\alpha) v
- i J_a^{-1} (a \cdot \nabla v) \nabla \xi.  
\end{equation}
Finally the map $(a,\alpha)\in D_R(0) \times (0,1) \mapsto (i \nabla + A_0^\alpha)^2 +\mathcal{L}_{(a,\alpha)} 
\in BL( H^{1,0}_0(\Omega, \C), ( H^{1,0}_0(\Omega, \C) )^\star $ is of class 
$C^\infty$.
\end{lemma}

\begin{proof}
The fact that $u \circ \Phi_a \in H^{1,0}_0(\Omega,\C)$ if $u \in 
H^{1,a}_0(\Omega,\C)$ follows easily using the definition of the 
functional space in Subsection~\ref{subsec:functional} and \eqref{eq:xi}, 
\eqref{eq:det}.

Using the definitions of $\mathcal{L}_{(a,\alpha)}$ in \eqref{eq:L}, 
\eqref{eq:compos} and the way $(i \nabla + A_a^\alpha)^2$ acts in 
\eqref{eq:magn-lap}, it holds that
\begin{equation} \label{eq:formal}
\phantom{a}_{(H^{1,0}_0(\Omega,\C))^\star}\left\langle 
[ (i \nabla + A_0^\alpha)^2 + \mathcal{L}_{(a,\alpha)} ] v, 
w \right\rangle_{H^{1,0}_0(\Omega,\C)} 
= \int_{\Omega} (i \nabla + A_a^\alpha) u \cdot \overline{(i \nabla + A_a^\alpha) f},
\end{equation}
where we set $u = v \circ \Phi_a^{-1}$ and $f = (w J_a^{-1}) \circ \Phi_a^{-1}$. 
Performing a change of variables in the right hand side of \eqref{eq:formal} 
and using the relation
\begin{equation*}
[ (i \nabla + A_a^\alpha) u ] \circ \Phi_a 
= i \nabla v - i J_a^{-1} (a \cdot \nabla v) \nabla \xi 
+ A_0^\alpha v + (A_a^\alpha \circ \Phi - A_0^\alpha) v 
= (i \nabla + A_0^\alpha) v + F(a,\alpha) v,
\end{equation*}
which holds true because of \eqref{eq:derivative-eta-a}, we obtain that
\[
\eqref{eq:formal} =
\int_{\Omega} \left[ (i \nabla + A_0^\alpha) v + F(a,\alpha) v \right] \cdot 
\overline{\left[ (i \nabla + A_0^\alpha) z + F(a,\alpha) z \right] } \, J_a,
\]
where $z = f \circ \Phi_a = w J_a^{-1}$. Finally, the claim follows using
\[
J_a (i \nabla + A_0^\alpha) z = (i \nabla + A_0^\alpha) w - i J_a^{-1} w \nabla J_a
\]
and
\[
J_a F(a,\alpha) z = F(a,\alpha) w + i J_a^{-2} w (a \cdot \nabla J_a) \nabla \xi.
\]

Since $\xi(x)\equiv 1$ for every $x\in D_R(0)$, then 
for every $a\in D_R(0)$ and every $\alpha\in (0,1)$
$A_a^\alpha \circ \Phi_a - A_0^\alpha = 0$ and $J_a\equiv 1$ for every $x\in D_R(0)$. 
Therefore we have that 
\begin{align*}
\eqref{eq:formal} &= 
\int_{D_R(0)} (i \nabla + A_0^\alpha) v \cdot \overline{(i \nabla + A_0^\alpha) z}  \\ 
&\quad + 
\int_{\Omega\setminus D_R(0)} \left[ (i \nabla + A_0^\alpha) v + F(a,\alpha) v \right] \cdot 
\overline{\left[ (i \nabla + A_0^\alpha) z + F(a,\alpha) z \right] } \, J_a;
\end{align*}
we then see that $A_a^\alpha \circ \Phi_a - A_0^\alpha $ is a $C^1(\Omega)$ function which depends 
$C^\infty$-regularly from the parameter $(a,\alpha)\in D_R(0)\times (0,1)$
(see also the argument in Lemma~\ref{lemma:smooth-map}).
In this way, the $\mathcal L_{(a,\alpha)}$ contribution is out of the disk $D_R(0)$ where singularities may occur and
the fact that the map $(a, \alpha) \mapsto \mathcal{L}_{(a,\alpha)}$ 
is $C^\infty$ follows from this, \eqref{eq:F}
and Lemma~\ref{lemma:smooth-map}. 
\end{proof}

First we notice that 
\[
\mathcal{L}_{(0,\sfrac{1}{2})} = 0,
\]
since $J_0 = 1$ and $F(0,\tfrac{1}{2}) = 0$. In fact, it also holds that 
$\mathcal{L}_{(0,\alpha)} = 0$ for the same reasons. Since the map $(a,\alpha) 
\mapsto \mathcal{L}_{(a,\alpha)}$ is smooth, in the following we write 
for every $(b,t) \in \R^2 \times \R$
\begin{equation} \label{eq:Lprim}
\mathcal{L}'(0,\tfrac{1}{2})[(b,t)] \, : \, H^{1,0}_0(\Omega,\C) \to 
( H^{1,0}_0(\Omega,\C) )^\star,
\end{equation}
the derivative of $\mathcal{L}_{(a,\alpha)}$ at the point $(0,\tfrac{1}{2})$, 
applied to $(b,t)$. Letting $\varepsilon:= \alpha - \tfrac{1}{2}$ it holds 
that 
\[
\mathcal{L}_{(a,\alpha)} = \mathcal{L}'(0,\tfrac{1}{2})[(a,\varepsilon)] 
+ o(|(a,\varepsilon)|),
\]
as $|(a,\varepsilon)| \to 0$, so that $\mathcal{L}_{(a,\alpha)} = O(|(a,
\varepsilon)|)$ in $BL( H^{1,0}_0(\Omega, \C), ( H^{1,0}_0(\Omega, \C) )^\star )$ 
as $|(a,\varepsilon)| \to 0$.

\section{Continuity of eigenvalues with respect to \texorpdfstring{$(a,\alpha)$}{alpha}} \label{sec:continuity}

\subsection{Proof of Theorem~\ref{thm:reg-a-alpha}: continuity} \label{subsec:cont}

The proof is based on the variational characterization of the magnetic 
eigenvalues given by \eqref{eq:var-char}. We follow the same outline as 
in \cite[Theorem 3.4]{BonnaillieNorisNysTerracini2014}.

\smallskip

\begin{claim} \label{claim:limsup}
We aim at proving that if $(a_0,\alpha_0) \in \Omega \times (0,1)$ then
\[
\limsup_{(a,\alpha) \to (a_0,\alpha_0)} \lambda_k^{(a,\alpha)} \leq 
\lambda_k^{(a_0,\alpha_0)}.
\]
\end{claim}

\begin{proofclaim}
It will be sufficient to find a $k$-dimensional linear subspace 
$E_k \subset H^{1,a}_0(\Omega,\C)$ such that 
\[
\int_{\Omega} | (i \nabla + A_a^\alpha) \Phi |^2 
\leq ( \lambda_k^{(a_0,\alpha_0)} 
+ \eps(a,\alpha) ) \| \Phi \|_{L^2(\Omega)}^2 
\quad \text{ for every } \Phi \in E_k,
\]
where $\eps(a,\alpha) \to 0$ as $(a,\alpha) \to (a_0,\alpha_0)$.

Let $\{ \varphi_1, \ldots, \varphi_k\}$ be a set of eigenfunctions respectively 
related to $\lambda_1^{(a_0,\alpha_0)}, \ldots, \lambda_k^{(a_0,\alpha_0)}$. 
Given $r \in (0,1)$, $r := 2 |a - a_0|$, let $\eta $ be a smooth cut-off 
function given by 
\begin{equation*}
\eta(x):= 
 \begin{cases}
0 & \text{ for } 0 \leq |x| \leq r \\
\frac{\log r - \log |x|}{\log\sqrt{r}} & \text{ for } r \leq |x| \leq \sqrt{r} \\
1 & \text{ for }|x| \geq \sqrt{r}
 \end{cases}.
\end{equation*}
We denote $\eta_a(x):= \eta(x-a)$. By \cite[Lemma 3.1]{BonnaillieNorisNysTerracini2014} 
we have that 
\[
 \int_{\R^2} (|\nabla\eta_a|^2 + (1-{\eta_a}^2)) \to 0 \quad \text{ as } a \to a_0.
\]
We define 
\[
E_k := \text{span} \{ \tilde \varphi_1, \ldots, \tilde \varphi_k \} \quad 
\text{ where } \tilde \varphi_j := e^{ i \alpha ( \theta_a -\theta_{a_0})} 
\eta_a \varphi_j.
\]
By \cite[Lemma 3.3]{BonnaillieNorisNysTerracini2014}, it holds that $\tilde{\varphi}_j \in 
H^{1,a}_0(\Omega,\C)$. We consider an arbitrary combination $\Phi:= \sum_{j=1}^k 
\beta_j \tilde\varphi_j$ for $\beta_j\in\C$. We compute
\begin{align*}
 \int_{\Omega} | (i\nabla + A_a^\alpha) \Phi |^2 
& = \int_{\Omega} \Big| \sum_{j=1}^k \beta_j (i\nabla + A_a^\alpha) \tilde \varphi_j \Big|^2 
= \sum_{j,l=1}^k \beta_j \overline{\beta_l} \int_{\Omega} (i\nabla + A_{a_0}^\alpha)^2 (\eta_a\varphi_j) 
\, (\eta_a \overline{\varphi_l} ) \\
& = \sum_{j,l=1}^k \beta_j \overline{\beta_l} \int_\Omega  
 \Big( \eta_a (i\nabla + A_{a_0}^\alpha)^2 \varphi_j 
 + 2i \nabla \eta_a \cdot (i \nabla + A_{a_0}^\alpha) \varphi_j 
 - \varphi_j \Delta \eta_a \Big) \, (\eta_a \overline{\varphi_l}).
\end{align*}
Letting $\delta := \alpha - \alpha_0$, we can rewrite it as
\begin{align}  \label{eq:cont1}
\int_{\Omega} | (i\nabla + A_a^\alpha) \Phi |^2  
& = \sum_{j,l=1}^k \beta_j \overline{\beta_l} \int_{\Omega} 
\Big( \eta_a (i \nabla + A_{a_0}^{\alpha_0})^2 \varphi_j 
+ 2i \nabla \eta_a \cdot (i \nabla + A_{a_0}^{\alpha_0}) \varphi_j 
- \varphi_j \Delta \eta_a \Big) ( \eta_a \overline{\varphi_l}) \\ \label{eq:cont2}
& + \delta \sum_{j,l=1}^k \beta_j \overline{\beta_l} \int_{\Omega} 
\Big( 2 i \eta_a A_{a_0}^1 \cdot \nabla \varphi_j 
+ (2\alpha_0 + \delta) \eta_a |A_{a_0}^1|^2 \varphi_j 
+ 2 i \nabla \eta_a \cdot A_{a_0}^1 \varphi_j \Big) (\eta_a \overline{\varphi_l}).
\end{align}
From \cite[Theorem 3.4, Step 1]{BonnaillieNorisNysTerracini2014} 
it follows that the first term 
\begin{equation}\label{eq:cont11}
\eqref{eq:cont1} \leq ( \lambda_k^{(a_0,\alpha_0)} + \eps'(a) ) \| \Phi \|^2_{L^2(\Omega)} ,
\end{equation}
where $\eps'(a)\to0$ as $a \to a_0$. For what concerns the second term 
\eqref{eq:cont2}, it will be sufficient to show that the integral appearing 
in \eqref{eq:cont2} is uniformly bounded with respect to $a$. To this aim, 
we estimate
\begin{align*}
& \left| \int_\Omega {\eta_a}^2 A_{a_0}^{1} \overline{\varphi_l} \cdot \nabla \varphi_j  \right|
\leq C \left\| \frac{\varphi_l}{|x-a_0|} \right\|_{L^2(\Omega)} \, 
\| |x - a_0| A_{a_0}^{1} \|_{L^\infty(\Omega)} \, 
\| \nabla \varphi_j \|_{L^2(\Omega)}, \\
& \left| \int_\Omega {\eta_a}^2 \varphi_j \overline{\varphi_l} |A^1_{a_0}|^2 \right| 
\leq C \left\| \frac{\varphi_j}{|x-a_0|} \right\|_{L^2(\Omega)} \, 
\left\| \frac{\varphi_l}{|x-a_0|} \right\|_{L^2(\Omega)}, \\
& \left| \int_\Omega {\eta_a} A_{a_0}^1 \cdot \nabla \eta_a \varphi_j \overline{\varphi_l} \right|
\leq \left\| \frac{\varphi_j}{|x-a_0|} \right\|_{L^2(\Omega)} \, 
\| |x-a_0| A^1_{a_0} \|_{L^\infty(\Omega)} \, 
\| \nabla \eta_a \|_{L^2(\Omega)} \, \| \varphi_l \|_{L^2(\Omega)}.
\end{align*}
Those three terms are uniformly bounded with respect to $a$.
Therefore from \cite[Lemma 3.3]{BonnaillieNorisNysTerracini2014}
\begin{equation} \label{eq:cont12}
\eqref{eq:cont2} \leq \delta C \| \Phi \|^2_{L^2(\Omega)}  \to 0 \quad \text{ as } \delta = \alpha - \alpha_0 \to 0,
\end{equation}
for a constant $C > 0$ independent of $(a, \delta)$.

The proof of the first step is concluded by \eqref{eq:cont11}--\eqref{eq:cont12}.
\end{proofclaim}

\begin{claim} \label{claim:liminf}
We aim at proving that if $(a_0,\alpha_0) \in \Omega \times (0,1)$ then
\[
 \liminf_{(a,\alpha)\to(a_0,\alpha_0)} \lambda_k^{(a,\alpha)} \geq \lambda_k^{(a_0,\alpha_0)}. 
\]
\end{claim}

\begin{proofclaim}
Consider $\{ \varphi_1^{(a,\alpha)}, \ldots, \varphi_k^{(a,\alpha)} \}$ be a set 
of orthonormalized eigenfunctions in $L^2(\Omega,\C)$ and 
respectively related to $\lambda_1^{(a,\alpha)}, \ldots, \lambda_k^{(a,\alpha)}$.

We first observe that Claim~\ref{claim:limsup} implies for 
$j = 1, \ldots, k$,
\[
\| \varphi_j^{(a,\alpha)} \|_{H^{1}_0(\Omega,\C)} 
\leq C \int_{\Omega} | (i\nabla + A_a^{\alpha}) \varphi_j|^2 
\leq C \lambda_j^{(a_0,\alpha_0)},
\]
for $(a,\alpha)$ sufficiently close to $(a_0,\alpha_0)$. 
Therefore there exist a sequence $(a_n,\alpha_n) \to (a_0,\alpha_0)$, as $n \to + \infty$, 
and functions $\varphi_j^* \in H^1_0(\Omega,\C)$ such that 
\begin{align*}
\begin{aligned}
\varphi_j^{(a_n,\alpha_n)} \rightharpoonup \varphi_j^* 
\quad & \text{ in } H^1_0(\Omega,\C) \\
\varphi_j^{(a_n,\alpha_n)} \to \varphi_j^* 
\quad & \text{ in } L^2(\Omega,\C)
\end{aligned} 
\quad \text{ as } n \to + \infty,
\end{align*}
and $\int_\Omega \varphi_j^*\overline\varphi_l^* = 0$ for $j\neq l$. 
Moreover, by Fatou's Lemma and Claim~\ref{claim:limsup} we have for 
any $j = 1, \ldots, k$,
\begin{align*}
\left\| \frac{\varphi_j^*}{|x - a_0|} \right\|^2_{L^2(\Omega,\C)} 
& \leq \liminf_{n \to + \infty} 
   \left\| \frac{\varphi_j^{(a_n,\alpha_n)}}{|x - a_n|} \right\|^2_{L^2(\Omega,\C)} 
\leq C \liminf_{n \to + \infty} \int_{\Omega} |(i \nabla + A_{a_n}^{\alpha_n}) 
\varphi_j^{(a_n,\alpha_n)} |^2 
\leq C \lambda_j^{(a_0,\alpha_0)},
\end{align*}
so that $\varphi_j^* \in H^{1,a_0}_0(\Omega,\C)$. 

Up to a diagonal process, with a little abuse of notation, let us assume that 
for any $j = 1, \ldots, k$ 
\[
\lim_{n \to + \infty} \lambda_j^{(a_n,\alpha_n)} = \lambda_j^* 
= \liminf_{(a,\alpha) \to (a_0,\alpha_0)} \lambda_j^{(a,\alpha)}. 
\]
Thus, given a test function $\phi \in C^\infty_0(\Omega\setminus\{a_0\})$, 
if $n$ is large enough to have $a_n \not\in \text{supp}\phi$, 
we can pass to the limit along the above subsequence in the following expression
\[
\int_{\Omega} \varphi_j^{(a_n,\alpha_n)} \, \overline{ (i\nabla + A_{a_n}^{\alpha_n})^2 \phi}
= \lambda_j^{(a_n,\alpha_n)} \int_\Omega \varphi_j^{(a_n,\alpha_n)} \overline{\phi}
\]
to obtain
\[
\int_\Omega (i\nabla + A_{a_0}^{\alpha_0})^2 \varphi_j^* \overline{\phi} 
= \lambda_j^* \int_\Omega \varphi_j^* \overline{\phi}. 
\]
By density, this is also valid for every $\phi \in H^{1,a_0}_0(\Omega,\C)$, and 
therefore the orthogonality between the $\varphi_j^*$ follows.

Thus, we obtain that
\begin{align*}
\lambda_k^{(a_0,\alpha_0)} 
& \leq \sup_{(c_1, \ldots ,c_k) \in \C \setminus\{0\} }
\frac{\int_\Omega \left| (i \nabla + A_{a_0}^{\alpha_0}) 
\Big( \sum_{j = 1}^k c_j \varphi_j^* \Big) \right|^2}
{\int_\Omega \left| \sum_{j=1}^k c_j \varphi_j^* \right|^2} 
=  \sup_{(c_1, \ldots ,c_k)\in \C \setminus\{0\} } 
\frac{ \sum_{j=1}^k |c_j|^2 \lambda_j^*}{ \sum_j |c_j|^2} \leq \lambda_k^*,
\end{align*}
and Claim~\ref{claim:liminf} follows.
\end{proofclaim}
The proof is thereby completed combining Claim~\ref{claim:limsup} and \ref{claim:liminf}.

\resetclaim

\begin{remark}
Following the scheme of \cite[Section~4]{BonnaillieNorisNysTerracini2014} it is also possible to prove
that for any $k\in \N$ the map $(a,\alpha)\mapsto \lambda_k^{(a,\alpha)}$ is continuous up to the boundary of
$\Omega\times (0,1)$.
\end{remark}

\subsection{Proof of Theorem~\ref{thm:reg-a-alpha}: higher regularity for simple eigenvalues} \label{subsec:reg}

Fix $\alpha_0 \in (0,1)$ such that $\lambda_k^{(0,\alpha_0)}$ 
is a simple eigenvalue. Throughout this subsection, we will 
treat for simplicity the space $H^{1,0}_0(\Omega,\C)$ as a 
\emph{real} Hilbert space endowed with the scalar product
\[
(u,v) := \Re \bigg( \int_\Omega (i\nabla + A_0^{\alpha_0}) u 
\cdot \overline{ (i\nabla + A_0^{\alpha_0}) v} \bigg).
\]
To emphasize the fact that $H^{1,0}_0(\Omega,\C)$ is meant as a 
vector space over $\R$, we denote it as $H^{1,0}_{0,\R}(\Omega,\C)$. 
The main difference lies in the fact that if $u \in H^{1,0}_{0,\R}
(\Omega,\C)$, then $u$ and $iu$ are linearly independent, which 
was not the case in the complex vector space. We also write 
$(H^{1,0}_{0,\R}(\Omega,\C))^\star$ the \emph{real} dual space of 
$H^{1,0}_{0,\R}(\Omega,\C)$.

Let us consider the function $F \, : \, \Omega \times (0,1) \times 
H^{1,0}_{0,\R}(\Omega,\C) \times \C  \to (H^{1,0}_{0,\R}(\Omega,\C))^\star 
\times  \R \times \R$ sending $(a, \alpha, \varphi, \lambda)$ on
\begin{equation} \label{def_operatore_F}
\left(
(i\nabla + A_0^\alpha)^2 \varphi 
+ \mathcal L_{(a,\alpha)} \varphi 
- \lambda \varphi, 
\, \Re \int_{\Omega} \varphi \, 
\big( \overline{\varphi_k^{(0,\alpha)} \circ \Phi_a } \big) \, J_a - 1 ,
\, \mathfrak{Im} \int_{\Omega} \varphi \, 
\big( \overline{\varphi_k^{(0,\alpha)} \circ \Phi_a } \big) \, J_a
\right),
\end{equation}
where $(i\nabla + A_0^\alpha)^2 \varphi - \lambda \varphi \in 
(H^{1,0}_{0,\R}(\Omega,\C))^\star$ acts as 
\[
\phantom{a}_{(H^{1,0}_{0,\R}(\Omega,\C))^\star}\Big\langle 
(i \nabla + A_0^\alpha)^2 \varphi - \lambda \varphi 
, u \Big\rangle_{\! H^{1,0}_{0,\R}(\Omega,\C)} \! \! 
= \Re \left( {\textstyle{ \int_{\Omega}(i \nabla + A_0^\alpha ) \varphi 
\cdot \overline{ (i \nabla + A_0^\alpha) u }
- \lambda \int_{\Omega} \varphi \overline{u} }} \right)
\]
for all $\varphi\in H^{1,0}_{0,\R}(\Omega,\C)$. We notice that in 
\eqref{def_operatore_F} $\C$ is also meant as a vector space over 
$\R$. 

We have that for any $\alpha_0 \in (0,1)$
\[
F(0,  \alpha_0, \varphi_k^{(0,\alpha_0)}, \lambda_k^{(0,\alpha_0)}) 
= (0,0,0),
\] 
since $\Phi_0$ is the identity, $J_0 = 1$ and $\mathcal{L}_{(0,\alpha_0)}=0$. 
Moreover, by direct calculations it is easy to verify that $F$ is $C^\infty$ 
with respect to $(\varphi,\lambda)$, at $(0,\alpha_0,\varphi_k^{(0,\alpha_0)}, 
\lambda_k^{(0,\alpha_0)})$ and moreover the explicit derivative of $F$ at 
$(0,\alpha_0,\varphi_k^{(0,\alpha_0)},\lambda_k^{(0,\alpha_0)})$, applied to 
$(\varphi,\lambda)$, is given by
\begin{align*}
& dF_{(\varphi,\lambda)}(0,\alpha_0,\varphi_k^{(0,\alpha_0)},
\lambda_k^{(0,\alpha_0)})[(\varphi,\lambda)] \\
& \qquad =  \left( 
(i\nabla + A_0^{\alpha_0})^2 \varphi 
- \lambda_k^{(0,\alpha_0)} \varphi 
- \lambda \varphi_k^{(0,\alpha_0)}, \,
\Re \int_{\Omega} \varphi \, \overline{\varphi_k^{(0,\alpha_0)}}, \, 
\mathfrak{Im} \int_{\Omega} \varphi \, \overline{\varphi_k^{(0,\alpha_0)}} 
\right)
\end{align*}
for every $(\varphi,\lambda)\in  H^{1,0}_{0,\R}(\Omega,\C) \times \C$.

It remains to prove that $dF_{(\varphi,\lambda)}(0,\alpha_0,\varphi_k^{(0,\alpha_0)},
\lambda_k^{(0,\alpha_0)}) \, : \, H^{1,0}_{0,\R}(\Omega,\C) \times \C \to 
(H^{1,0}_{0,\R}(\Omega),\C)^\star \times \R \times \R$ is invertible. For any 
$\lambda \in \C$, we define 
\[
T_\lambda \, : \, H^{1,0}_{0,\R}(\Omega,\C) \to (H^{1,0}_{0,\R}(\Omega,\C))^\star 
\quad \text{ as } \quad 
T_\lambda u = \lambda u.
\]
We define as well the Riesz isomorphism $\mathcal R \, : \, 
(H^{1,0}_{0,\R}(\Omega,\C))^\star \to H^{1,0}_{0,\R}(\Omega,\C)$, 
and $\mathcal I$ the standard identification of $\R \times \R$ 
onto $\C$. By exploiting the compactness of $T_\lambda$, it is 
easy to prove that $(\mathcal R \times \mathcal I) \circ dF_{(\varphi,\lambda)}
(0, \alpha_0, \varphi_k^{(0,\alpha_0)}, \lambda_k^{(0,\alpha_0)}) \in BL
(H^{1,0}_{0,\R}(\Omega) \times \C,H^{1,0}_{0,\R}(\Omega) \times \C)$ is a 
compact perturbation of the identity. Indeed, since by definition
\[
\phantom{a}_{(H^{1,0}_{0,\R}(\Omega))^\star} \big\langle 
(i\nabla + A_0^{\alpha_0})^2 \varphi , 
u \big\rangle_{H^{1,0}_{0,\R}(\Omega)} 
= \Re \left( \int_{\Omega} (i \nabla + A_0^{\alpha_0}) \varphi 
\cdot \overline{ (i \nabla + A_0^{\alpha_0})u} \right) 
= \big( \varphi ,u \big)_{H^{1,0}_{0,\R}(\Omega,\C)},
\]
we have that
$\mathcal{R} \big( (i\nabla + A_0^{\alpha_0})^2 \varphi 
- \lambda_k^{(0,\alpha_0)} \varphi - \lambda \varphi_k^{(0,\alpha_0)} \big) 
= \varphi - ( \mathcal{R} \circ T_{\lambda_k^{(0,\alpha_0)}} ) \varphi 
- ( \mathcal{R} \circ T_\lambda ) \varphi_k^{(0,\alpha_0)}$, which has the form 
identity plus a compact perturbation (composition of the Riesz isomorphism and 
the compact operator $T_\lambda$). The Fredholm alternative tells us then 
that $dF_{(\varphi,\lambda)}(0, \alpha_0, \lambda_k^{(0,\alpha_0)}, 
\varphi_k^{(0,\alpha_0)})$ is invertible if and only if it is injective. 
Therefore to conclude the proof, it is enough to prove that
$\ker ( dF_{(\varphi,\lambda)} (0, \alpha_0, \lambda_k^{(0,\alpha_0)}, 
\varphi_k^{(0,\alpha_0)})) = \{(0,0)\}$.

Let $(\varphi,\lambda) \in H^{1,0}_{0,\R}(\Omega,\C) \times \C$ be such 
that 
\begin{equation} \label{eq:inv}
\begin{split}
& \qquad (i \nabla + A_0^{\alpha_0})^2 \varphi 
- \lambda_k^{(0,\alpha_0)} \varphi 
- \lambda \varphi_k^{(0,\alpha_0)} = 0, \\
& \Re \int_{\Omega} \varphi \, \overline{\varphi_k^{(0,\alpha_0)}} = 0 
\quad \text{ and } \quad 
\mathfrak{Im} \int_{\Omega} \varphi \, \overline{\varphi_k^{(0,\alpha_0)}} = 0.
\end{split}
\end{equation}
The first equation means that
\[
\Re \int_{\Omega} (i \nabla + A_0^{\alpha_0} ) \varphi 
\cdot \overline{ (i \nabla + A_0^{\alpha_0} ) u } 
- \lambda_k^{(0,\alpha_0)} \varphi \, \overline{u} 
- \lambda \varphi_k^{(0,\alpha_0)} \, \overline{u} = 0
\] 
for all $u\in H^{1,0}_{0,\R}(\Omega,\C)$. Considering in turn 
$u = \varphi_k^{(0,\alpha_0)}$ and $u = i \varphi_k^{(0,\alpha_0)}$ 
into the previous identity leads respectively to $\Re \lambda = 0$ 
and $\mathfrak{Im}\lambda = 0$. Then the first equation in \eqref{eq:inv} 
becomes $( i \nabla + A_0^{\alpha_0})^2 \varphi - \lambda_k^{(0,\alpha_0)} 
\varphi = 0$ in $(H^{1,0}_{0,\R}(\Omega,\C))^\star$, which, 
by assumption of simplicity of $\lambda_k^{(0,\alpha_0)}$, implies that 
$\varphi = \gamma \varphi_k^{(0,\alpha_0)}$ for some $\gamma \in \C$. 
The second and third equations in \eqref{eq:inv} imply respectively that 
$\Re \gamma = 0$ and $\mathfrak{Im} \gamma = 0$, so that $\varphi=0$. Then 
we conclude that the only element in the kernel of $dF_{(\varphi,\lambda)}
(0,\alpha_0,\lambda_k^{(0,\alpha_0)},\varphi_k^{(0,\alpha_0)})$ is $(0,0) 
\in H^{1,0}_{0,\R}(\Omega,\C) \times \C$.

The Implicit Function Theorem therefore applies and the maps $(a, \alpha) 
\mapsto (\varphi_k^{(a,\alpha)} \circ \Phi_a, \lambda_k^{(a,\alpha)}) \in 
H^{1,0}_{0,\R}(\Omega,\C) \times \C $ are of class $C^{\infty}$ locally in 
a neighborhood of $(0,\alpha_0)$.

\section{The spectrally equivalent operators} \label{sec:equivalent-operator}

As in \cite{LupoMicheletti1993} we define $\gamma_a \, : \, L^2(\Omega,\C) 
\to L^2(\Omega,\C)$ by
\begin{equation}\label{eq:gamma_a}
\gamma_a (u) = \sqrt{J_a} (u \circ \Phi_a),
\end{equation}
where $J_a$ is defined in \eqref{eq:det}. Such a transformation $\gamma_a$ 
defines an isomorphism preserving the scalar product in $L^2(\Omega,\C)$. 
Indeed,
\[
\int_{\Omega} u(y) \overline{v(y)} \, \mathrm{d}y 
= \int_{\Omega} u(\Phi_a(x)) \overline{v(\Phi_a(x))} J_a(x) \, \mathrm{d}x 
= \int_{\Omega} \gamma_a(u)(x) \overline{\gamma_a(v)(x)} \, \mathrm{d}x.
\]
Since $\Phi_a$ and $\sqrt{J_a}$ are $C^\infty$, $\gamma_a$ defines an algebraic 
and topological isomorphism of $H^{1,a}_0(\Omega,\C)$ in $H^{1,0}_0(\Omega,\C)$ 
and inversely with $\gamma_a^{-1}$, see \cite[Lemma 2]{Micheletti1972}, 
\cite{LupoMicheletti1995-2}. We notice that $\gamma_a^{-1}$ writes
\[
\gamma_a^{-1} (u) = \left(\sqrt{J_a \circ \Phi_a^{-1}} \right)^{-1} 
( u \circ \Phi_a^{-1}).
\]

With a little abuse of notation we define the application $\gamma_a \, : \, 
( H^{1,a}_0(\Omega,\C) )^\star \to (H^{1,0}_0(\Omega,\C))^\star$ in 
such a way that 
\begin{equation} \label{eq:gamma-a-dual}
\phantom{a}_{( H^{1,0}_0(\Omega,\C) )^\star }\langle 
\gamma_a(f), v \rangle_{H^{1,0}_0(\Omega,\C)} 
= \phantom{a}_{( H^{1,0}_0(\Omega,\C) )^\star }\langle 
f, \gamma_a^{-1}(v)
\rangle_{H^{1,a}_0(\Omega,\C)},
\end{equation}
for any $f \in ( H^{1,a}_0(\Omega,\C) )^\star$, and inversely for 
$\gamma_a^{-1} \, : \, ( H^{1,0}_0(\Omega,\C) )^\star \to (H^{1,a}_0(\Omega,\C))^\star$.

\subsection{Spectral equivalent operator to \texorpdfstring{$(i\nabla + A_a^\alpha)^2$}{direct operator}}

We would like to find an operator spectrally equivalent to 
$(i \nabla + A_a^\alpha)^2$ but having a domain of definition 
independent of $(a,\alpha)$. The parameter $\alpha$ does not 
create any problem since the functional spaces introduced in 
Subsection~\ref{subsec:functional} are independent of $\alpha$. 
We therefore need only to perform a transformation moving the 
pole $a$ to the fixed point $0$. For this, for every $(a,\alpha) 
\in D_R(0) \times (0,1)$, we define the new operator $G_{(a,\alpha)} 
\, : \, H^{1,0}_0(\Omega,\C) \to (H^{1,0}_0(\Omega,\C))^\star$ 
by the following relation
\begin{equation} \label{eq:operator-G-a}
G_{(a,\alpha)} \circ \gamma_a = \gamma_a \circ (i \nabla + A_a^\alpha)^2,
\end{equation}
being $\gamma_a$ defined in \eqref{eq:gamma_a} and \eqref{eq:gamma-a-dual}.
By \cite[Lemma 3]{Micheletti1972} the domain of definition of 
$G_{(a,\alpha)}$ is given by $\gamma_a(H^{1,a}_0(\Omega, \C))$, 
it coincides with $H^{1,0}_0(\Omega,\C)$. Moreover, $G_{(a,\alpha)}$ 
and $(i \nabla + A_a^\alpha)^2$ are \emph{spectrally equivalent}, 
in particular they have the same eigenvalues with the same multiplicity. 

The following lemma gives a more explicit expression to the operator $G_{(a,\alpha)}$.

\begin{lemma} \label{lemma:operator-G-a}
Let $(a, \alpha) \in D_R(0) \times (0,1)$ and let $G_{(a,\alpha)}$ be defined 
in \eqref{eq:operator-G-a}. Then 
\[
G_{(a,\alpha)} v = \sqrt{J_a} \, [ (i \nabla + A_0^{\alpha})^2 
+ \mathcal{L}_{(a,\alpha)} ] 
( v \, (\sqrt{J_a})^{-1}),
\]
meaning that
\begin{equation*}
\begin{split}
& \phantom{a}_{( H^{1,0}_0(\Omega,\C))^\star}\langle 
G_{(a,\alpha)} v, w \rangle_{H^{1,0}_0(\Omega,\C)}  \\
& \qquad \qquad \qquad = \phantom{a}_{( H^{1,0}_0(\Omega,\C))^\star}\left\langle 
[ (i \nabla + A_0^{\alpha})^2 + \mathcal{L}_{(a,\alpha)} ] 
( v \, (\sqrt{J_a})^{-1}),
\sqrt{J_a} w \right\rangle_{H^{1,0}_0(\Omega,\C)} .
\end{split}
\end{equation*}
Moreover, $(a,\alpha)  \mapsto G_{(a,\alpha)}$ is $C^\infty (D_R(0) 
\times (0,1), BL( H^{1,0}_0(\Omega,\C), (H^{1,0}_0(\Omega,\C))^\star )$.
\end{lemma}

\begin{proof}
Let $u, \, v \in H^{1,a}_0(\Omega,\C)$. Using Lemma~\ref{lemma:new-operator} 
and equation \eqref{eq:operator-G-a}, we have that
\begin{align*}
& \phantom{a}_{( H^{1,0}_0(\Omega,\C))^\star} \left\langle  
G_{(a,\alpha)} (\gamma_a(u)) ,\, \gamma_a(v) 
\right\rangle_{H^{1,0}_0(\Omega,\C)} \\
& \qquad \qquad \qquad
= \phantom{a}_{( H^{1,0}_0(\Omega,\C))^\star} \left\langle 
\gamma_a( (i  \nabla + A_a^\alpha)^2 u), \, \gamma_a(v)
\right\rangle_{H^{1,0}_0(\Omega,\C)} \\
& \qquad \qquad \qquad
= \phantom{a}_{( H^{1,0}_0(\Omega,\C))^\star} \left\langle  
[ ((i \nabla + A_a^\alpha)^2 u) \circ \Phi_a ] ,\,
(v \circ \Phi_a) \, J_a 
\right\rangle_{H^{1,0}_0(\Omega,\C)} \\
& \qquad \qquad \qquad
= \phantom{a}_{( H^{1,0}_0(\Omega,\C))^\star} \left\langle 
[ (i \nabla + A_0^\alpha)^2 + \mathcal{L}_{(a,\alpha)} ] (u \circ \Phi_a), \, 
(v \circ \Phi_a) \, J_a 
\right\rangle_{H^{1,0}_0(\Omega,\C)} \\
& \qquad \qquad \qquad
= \phantom{a}_{( H^{1,0}_0(\Omega,\C))^\star} \left\langle 
[ (i \nabla + A_0^\alpha)^2 + \mathcal{L}_{(a,\alpha)} ]
( \gamma_a(u) (\sqrt{J_a})^{-1}) ,\, \gamma_a(v) \, \sqrt{J_a} 
\right\rangle_{H^{1,0}_0(\Omega,\C)}.
\end{align*}
This proves the first claim. 

When $(a,\alpha) \in D_R(0) \times (0,1)$, the regularity of 
$(a,\alpha) \mapsto G_{(a,\alpha)}$ follows from Lemmas~
\ref{lemma:smooth-map} and \ref{lemma:new-operator}.
\end{proof}

We first notice that
\[
G_{(0,\sfrac{1}{2})} = (i \nabla + A_0^{\sfrac{1}{2}})^2
\]
since $J_0 = 1$ and $\mathcal{L}_{(0,\sfrac{1}{2})} = 0$. Because 
of its regularity, in the following we write for every $(b,t) \in 
\R^2 \times \R$
\begin{equation} \label{eq:Gprim}
G'(0,\tfrac{1}{2})[(b,t)] \, : \, H^{1,0}_0(\Omega,\C) \to 
(H^{1,0}_0(\Omega,\C))^\star
\end{equation}
the derivative operator of $G_{(a,\alpha)}$ at the point 
$(0,\tfrac{1}{2})$ applied to $(b,t)$. Therefore letting 
$\varepsilon := \alpha - \tfrac{1}{2}$
\[
G_{(a,\alpha)} = (i \nabla + A_0^{\sfrac{1}{2}})^2 
+ G'(0,\tfrac{1}{2})[(a,\varepsilon)] 
+ o(|(a,\varepsilon)|)
\]
as $|(a,\varepsilon)| \to 0$.

\subsection{Spectral equivalent operator to \texorpdfstring{$\left[ (i\nabla + A_a^\alpha)^2 \right]^{-1} \circ Im_{H^{1,a}_0(\Omega,\C) \to (H^{1,a}_0(\Omega,\C))^\star}$ }{inverse operator}}

In order to use the abstract Theorem~\ref{theorem:abstract-complex} we 
would like to define a family of compact operators spectrally equivalent to 
$[(i \nabla + A_a^\alpha)^{2}]^{-1} \circ \text{Im}_{H^{1,a}_0(\Omega,\C) \to 
( H^{1,a}_0(\Omega,\C))^\star}$, but having a fixed domain of definition. 
We proceed as in \cite{LupoMicheletti1993} and define the Hermitian form 
$E_{(a,\alpha)} \, : \, H^{1,0}_0(\Omega,\C) \times H^{1,0}_0(\Omega,\C) \to \C$ 
\begin{equation} \label{eq:bilinear-form}
E_{(a,\alpha)}(u, v) = \int_{\Omega} (i\nabla + A_a^\alpha) \gamma_a^{-1}(u) 
\cdot \overline{(i\nabla + A_a^\alpha) \gamma_a^{-1}(v)}.
\end{equation}
Since $\gamma_a$ defines an algebraic and topological isomorphism of 
$H^{1,a}_0(\Omega,\C)$ in $H^{1,0}_0(\Omega,\C)$, and inversely for 
$\gamma_a^{-1}$, the Hermitian form $E_{(a,\alpha)}$ is easily proved 
to be continuous and coercive. Then, via Lax-Milgram and Riesz Theorems, 
it defines a scalar product equivalent to the standard one on 
$H^{1,0}_0(\Omega,\C)$, i.e.\ there exists $c_{(a,\alpha)} > 0$ and 
$d_{(a,\alpha)} > 0$ such that
\[
c_{(a,\alpha)} \| u \|^2_{H^{1,0}_0(\Omega,\C)} 
\leq E_{(a,\alpha)} (u,u) 
\leq d_{(a,\alpha)} \| u \|^2_{H^{1,0}_0(\Omega,\C)}
\quad \forall u \in H^{1,0}_0(\Omega,\C).
\] 
In a standard way, $E_{(a,\alpha)}$ uniquely defines uniquely a 
self-adjoint compact linear operator $B_{(a,\alpha)} \, : \, 
H^{1,0}_0(\Omega,\C) \to H^{1,0}_0(\Omega,\C)$ by
\begin{equation} \label{eq:inverse-operator}
E_{(a,\alpha)} (B_{(a,\alpha)} \big(\gamma_a(u) \big), \gamma_a(v)) 
= \int_{\Omega} \gamma_a(u) \overline{\gamma_a(v)} 
= \int_{\Omega} u \overline{v}.
\end{equation}

\begin{lemma}\label{l:BaC1}
Let $\mathcal{W} \subset (0,1)$ be any neighborhood of $\{\tfrac{1}{2}\}$ 
such that $\mathcal{W} \subset\subset (0,1)$. The map $(a,\alpha) 
\mapsto B_{(a,\alpha)}$ is $C^1\big(D_R(0) \times \mathcal{W}, BL( H^{1,0}_0
(\Omega,\C), H^{1,0}_0(\Omega,\C)) \big)$.
\end{lemma}

\begin{proof}
The proof is similar to the one in \cite{LupoMicheletti1993}. For completeness 
we refer to the Section~\ref{app:proof} in the Appendix.
\end{proof}

Since \eqref{eq:bilinear-form} and \eqref{eq:inverse-operator} hold 
we have that
\begin{equation}\label{eq:Ba}
G_{(a,\alpha)} \circ B_{(a,\alpha)} = \text{Im}_{H^{1,0}_0(\Omega,\C) 
\to (H^{1,0}_0(\Omega,\C))^\star},
\end{equation}
where $\text{Im}_{H^{1,0}_0(\Omega,\C) \to (H^{1,0}_0(\Omega,\C))^\star}$ 
is the compact immersion from $H^{1,0}_0(\Omega,\C)$ to $(H^{1,0}_0(\Omega,\C))^\star$. 
Moreover, it is worthwhile noticing that since $G_{(0,\sfrac{1}{2})} = 
(i \nabla + A_0^{\sfrac{1}{2}})^2$ 
\[
B_{(0,\sfrac{1}{2})} = \left[ (i\nabla + A_0^{\sfrac{1}{2}})^2\right]^{-1} 
\circ \text{ Im}_{H^{1,0}_0(\Omega,\C) \to (H^{1,0}_0(\Omega,\C))^\star},
\]
i.e.\ the unperturbed compact inverse operator. Moreover, because of its 
regularity, we write for every $(b,t) \in \R^2 \times\R$
\begin{equation} \label{eq:Bprim}
B'(0,\tfrac{1}{2})[(b,t)] \, : \, H^{1,0}_0(\Omega,\C) \to H^{1,0}_0(\Omega,\C)
\end{equation}
the derivative of $B_{(a,\alpha)}$ at the point $(0,\tfrac{1}{2})$, applied 
to $(b,t)$. Therefore, letting $\varepsilon := \alpha - \tfrac{1}{2}$
\[
B_{(a,\alpha)} = \left[ (i\nabla + A_0^{\sfrac{1}{2}})^2\right]^{-1} \circ 
\text{ Im}_{H^{1,0}_0(\Omega,\C) \to (H^{1,0}_0(\Omega,\C))^\star} 
+ B'(0,\tfrac{1}{2})[(a,\varepsilon)] + o(|(a,\varepsilon)|)
\]
as $|(a,\varepsilon)| \to 0$.

\begin{remark}
We also remark that by \cite[Lemma 3]{Micheletti1972}, since $B_{(a,\alpha)}$ 
can be rewriten from \eqref{eq:Ba} and \eqref{eq:operator-G-a} as
\begin{equation*}
\begin{split}
B_{(a,\alpha)} & = \gamma_a \circ [(i \nabla + A_a^\alpha)^2]^{-1} 
\circ \gamma_a^{-1} \circ \textrm{ Im}_{H^{1,0}_0(\Omega,\C) 
\to ( H^{1,0}_0(\Omega,\C))^\star} \\
& = \gamma_a \circ [(i \nabla + A_a^\alpha)^2]^{-1} 
\circ \textrm{ Im}_{H^{1,a}_0(\Omega,\C) 
\to ( H^{1,a}_0(\Omega,\C))^\star} \circ \gamma_a^{-1},
\end{split}
\end{equation*}
it holds that $B_{(a,\alpha)}$ and $\left[ (i\nabla + A_a^\alpha)^2\right]^{-1} 
\textrm{ Im}_{H^{1,a}_0(\Omega,\C) \to ( H^{1,a}_0(\Omega,\C))^\star}$ 
are spectrally equivalent, so that they have the same eigenvalues with the same multiplicity. 
Morever, those eigenvalues are the inverse of the eigenvalues of $G_{(a,\alpha)}$ 
and $(i \nabla + A_a^\alpha)^2$ (which are also spectrally equivalent).
\end{remark}

\section{Proof of Theorem~\ref{t:main}} \label{sec:proof}

\subsection{The first order terms}

In this section, we assume to have an eigenvalue $\lambda \in 
\R^+$ of $(i \nabla + A_0^{\sfrac{1}{2}})^2$ of multiplicity 
$\mu \geq 1$, and we denote by $\varphi_j$, $j = 1, \ldots, \mu$, 
the corresponding eigenfunctions orthonormalized in $L^2(\Omega,\C)$. 
Moreover, from Section~\ref{sec:gauge} we know that we can consider 
a system of $K_0$-real eigenfunctions, and that we can write for 
$j = 1, \ldots, \mu$
\begin{equation} \label{eq:coeff}
\varphi_j(r (\cos t, \sin t)) = e^{i \frac{t}{2}} r^{1/2} 
\left( c_j \cos \frac{t}{2} + d_j \sin \frac{t}{2} \right) 
+ f_j(r,t) \quad \text{ as } r \to 0^+, 
\end{equation}
where $f_j(r,t) = O(r^{3/2})$ uniformly in $t \in [0,2\pi]$ and 
$c_j, d_j \in \mathbb{R}$ can possibly be zero. We also recall that 
\begin{equation} \label{eq:regul}
\text{the eigenfunctions of the operator \eqref{eq:magnetic-pot}
are of class } C^\infty(\omega,\C) \text{ for any } \omega\subset
\subset\Omega\setminus\{0\}
\end{equation}
by the results in \cite{FelliFerreroTerracini2011} and standard elliptic estimates.

It is only in the next section dedicated to the proof of Theorem~\ref{t:main} 
that we restrict ourselves to the case of multiplicity $\mu = 2$.

\subsubsection{First order terms of \texorpdfstring{$\mathcal{L}_{(a,\alpha)}$}{L}, \texorpdfstring{$G_{(a,\alpha)}$}{G} and \texorpdfstring{$B_{(a,\alpha)}$}{B}}

To use Theorem~\ref{theorem:abstract-complex} we need to consider the 
derivative $B'(0,\tfrac{1}{2})$ applied to eigenfunctions $\varphi_j$. 
However, this object is difficult to calculate explicitely since 
$B_{(a,\alpha)}$ is defined in an implicit way, through $E_{(a,\alpha)}$, 
see \eqref{eq:inverse-operator}. Nevertheless, \eqref{eq:Ba} will allow 
us to find a relation with $G'(0,\tfrac{1}{2})$. As a first step, we 
need the expression of the derivative $\mathcal{L}'(0,\tfrac{1}{2})$ 
applied to eigenfunctions.

\begin{lemma} \label{lemma:first-order-La}
Let $\Omega \subset \R^2$ be open, bounded, simply connected and Lipschitz. 
Let $\mathcal{L}'(0,\tfrac{1}{2})$ be defined as in \eqref{eq:Lprim}. 
Let $\lambda \in \R^+$ be an eigenvalue of $(i\nabla + A_0^{\sfrac{1}{2}})^2$ 
of multiplicity $\mu \geq 1$, and let $\varphi_j \in H^{1,0}_0(\Omega,\C)$, 
$j = 1, \ldots , \mu$, be the corresponding eigenfunctions orthonormalized 
in $L^2(\Omega,\C)$. Then, for every $(b,t) \in \R^2 \times \R$ and $j,k 
= 1, \ldots, \mu$
\begin{align*}
\phantom{a}_{(H^{1,0}_0(\Omega,\C))^\star} \left\langle 
\mathcal{L}'(0,\tfrac{1}{2})[(b,t)] \varphi_j , 
\varphi_k \right \rangle_{H^{1,0}_0(\Omega,\C)} = 
\int_{\partial \Omega} ( b \cdot \nu) \, \frac{\partial \varphi_j}{\partial \nu} 
\, \overline{\frac{\partial \varphi_k}{\partial\nu}},
\end{align*}
where $\nu \, : \, \partial \Omega \to \mathbb{S}^1$ is the exterior 
normal to $\partial \Omega$.
\end{lemma}

\begin{proof}
\begin{claim} \label{claim:general-expression-L-prim}
We first prove that for every $(b,t) \in \R^2 \times \R$
\begin{equation} \label{eq:first-order-L-a}
\begin{split}
& \phantom{a}_{(H^{1,0}_0(\Omega,\C))^\star} \left\langle 
\mathcal{L}'(0,\tfrac{1}{2})[(b,t)] \varphi_j , 
\varphi_k \right \rangle_{H^{1,0}_0(\Omega,\C)} \\
& \qquad \qquad \qquad 
= \int_{\Omega} (\xi - 1) \, b \cdot \nabla 
\left[ (i \nabla + A_0^{\sfrac{1}{2}})^2 \varphi_j \right] 
\, \overline{\varphi_k}
- (i \nabla + A_0^{\sfrac{1}{2}})^2 
\left[ (\xi - 1) b \cdot \nabla \varphi_j \right] 
\, \overline{\varphi_k},
\end{split}
\end{equation}
and $\mathcal{L}'(0,\tfrac{1}{2})[(b,t)] = O(|(b,t)|)$ in 
$BL( H^{1,0}_0(\Omega, \mathbb{C}), (H^{1,0}_0(\Omega,\C))^\star)$
as $(b,t)\to(0,0)$. 
\end{claim}

\begin{proofclaim}
The proof being quite technical, we report it in Section~\ref{app:Lprim} 
in the Appendix.
\end{proofclaim}

\resetclaim

The Lemma follows by an integration by parts and the facts that 
for any $j = 1, \ldots, \mu$ the eigenfunctions $\varphi_j = 0$ 
on $\partial\Omega$ and $(\xi - 1) = -1$ on $\partial \Omega$, 
in addition the $\varphi_j$ are eigenfunctions of the same eigenvalue.
\end{proof}

We can now give an expression of $G'(0,\tfrac{1}{2})$.

\begin{lemma} \label{lemma:first-order-G-a}
Let $\Omega \subset \R^2$ be open, bounded, simply connected and Lipschitz. 
Let $G'(0,\tfrac{1}{2})$ be defined as in \eqref{eq:Gprim}. 
Let $\lambda \in \R^+$ be an eigenvalue of $(i\nabla + A_0^{\sfrac{1}{2}})^2$ 
of multiplicity $\mu \geq 1$, and let $\varphi_j \in H^{1,0}_0(\Omega,\C)$, 
$j = 1, \ldots , \mu$, be the corresponding eigenfunctions orthonormalized 
in $L^2(\Omega,\C)$. Then, for every $(b,t) \in \R^2 \times \R$ and $j,k 
= 1, \ldots, \mu$
\begin{align*}
\phantom{a}_{(H^{1,0}_0(\Omega,\C))^\star} \left\langle 
G'(0,\tfrac{1}{2})[(b,t)] \varphi_j , 
\varphi_k \right \rangle_{H^{1,0}_0(\Omega,\C)} = 
\int_{\partial \Omega} ( b \cdot \nu) \, \frac{\partial \varphi_j}{\partial \nu} 
\, \overline{\frac{\partial \varphi_k}{\partial\nu}} 
+ 4 t \int_{\Omega} (i \nabla + A_0^{\sfrac{1}{2}}) \varphi_j 
\cdot A_0^{\sfrac{1}{2}} \, \overline{\varphi_k}.
\end{align*}
\end{lemma}

\begin{proof}
\begin{claim} \label{claim:general-expression-G-prim}
We first prove that for every $(b,t) \in \R^2 \times \R$
\begin{equation} \label{eq:first-order-G-a}
\begin{split}
& \phantom{a}_{(H^{1,0}_0(\Omega,\C))^\star} \left\langle 
G'(0,\tfrac{1}{2})[(b,t)] \varphi_j, 
\varphi_k \right\rangle_{H^{1,0}_0(\Omega,\C)} \\
& \qquad \qquad \qquad = \phantom{a}_{(H^{1,0}_0(\Omega,\C))^\star} \left\langle 
\mathcal{L}'(0,\tfrac{1}{2})[(b,t)] \varphi_j, \varphi_k 
\right\rangle_{H^{1,0}_0(\Omega,\C)} \\
& \qquad \qquad \qquad + \frac{1}{2} \int_{\Omega}  [ (b \cdot \nabla \xi ) 
(i \nabla + A_0^{\sfrac{1}{2}})^2 \varphi_j  
- (i \nabla + A_0^{\sfrac{1}{2}})^2 ( (b \cdot \nabla \xi) \varphi_j) ] 
\, \overline{\varphi_k} \\
& \qquad \qquad \qquad + 4 t \int_{\Omega} (i \nabla + A_0^{\sfrac{1}{2}}) \varphi_j 
\cdot A_0^{\sfrac{1}{2}} \, \overline{\varphi_k}, 
\end{split}
\end{equation}
being $\mathcal{L}'(0,\tfrac{1}{2})$ as in \eqref{eq:first-order-L-a}, 
and $G'(0,\tfrac{1}{2})[(b,t)] = O(|(b,t)|)$ in 
$BL( H^{1,0}_0(\Omega, \mathbb{C}), (H^{1,0}_0(\Omega,\C))^\star)$
as $(b,t)\to(0,0)$.
\end{claim}

\begin{proofclaim}
Again, the proof being technical, we report it in Section~\ref{app:Gprim} 
in the Appendix.
\end{proofclaim}

\resetclaim

An integration by parts in Claim~\ref{claim:general-expression-G-prim} and 
the facts that $\varphi_j$, $j=1,\ldots,\mu$, are eigenfunctions of the same 
eigenvalue vanishing on $\partial \Omega$, and $\nabla \xi = 0$ on $\partial 
\Omega$, tell us that the second and third terms in \eqref{eq:first-order-G-a} 
cancel. Therefore the Lemma follows using Lemma~\ref{lemma:first-order-La}.
\end{proof}

\begin{remark} \label{rem:welldef}
Although the eigenfunctions $\varphi_j$ are in $H^{1,0}_0(\Omega,\C)$, 
expressions \eqref{eq:first-order-L-a} and \eqref{eq:first-order-G-a} 
are well defined. Indeed, this follows from the presence of the cut-off 
function $(\xi - 1)$, which vanishes in a neighborhood of $0$, and \eqref{eq:regul}.
\end{remark}

The next lemma gives us the relation between $G'(0,\tfrac{1}{2})$ and 
$B'(0,\tfrac{1}{2})$.

\begin{lemma} \label{lemma:relation-matrices}
Let $G'(0,\tfrac{1}{2})$ and $B'(0,\tfrac{1}{2})$ be defined respectively 
in \eqref{eq:Gprim} and \eqref{eq:Bprim}. Let $\lambda \in \mathbb{R}^+$ 
be an eigenvalue of $(i \nabla + A_0^{\sfrac{1}{2}})^2$ of multiplicity 
$\mu \geq 1$ and $\varphi_j$, $j = 1, \ldots, \mu$, be the corresponding 
eigenfunctions orthonormalized in $L^2(\Omega,\C)$. Then for any $j,k = 1, 
\ldots, \mu$ and $(b,t) \in \R^2 \times \R$
\begin{equation*}
\begin{split}
( B'(0,\tfrac{1}{2})[(b,t)] \varphi_j, \varphi_k)_{H^{1,0}_0(\Omega,\C)} 
& := \int_{\Omega} (i \nabla + A_0^{\sfrac{1}{2}}) 
( B'(0,\tfrac{1}{2})[(b,t)] \varphi_j )
\cdot \overline{(i \nabla + A_0^{\sfrac{1}{2}}) \varphi_k} \\
& = - \lambda^{-1} \phantom{a}_{(H^{1,0}_0(\Omega,\C))^\star}\left\langle 
G'(0,\tfrac{1}{2})[(b,t)] \varphi_j, 
\varphi_k \right\rangle_{H^{1,0}_0(\Omega,\C)}
\end{split}
\end{equation*}
\end{lemma}

\begin{proof}
We denote again $\varepsilon = \alpha - \frac{1}{2}$. Since by \eqref{eq:Ba}
\[
G_{(a,\alpha)} \circ B_{(a,\alpha)} 
= \text{ Im}_{H^{1,0}_0(\Omega,\C) \to (H^{1,0}_0(\Omega,\C))^\star},  
\]
we have for $\varphi_j, \, \varphi_k \in H^{1,0}_0(\Omega, \C)$ 
\begin{align*}
& \phantom{a}_{(H^{1,0}_0(\Omega,\C))^\star}\left\langle
G_{(0,\sfrac{1}{2})} ( B'(0,\tfrac{1}{2})[(a,\varepsilon)] \varphi_j ), 
\varphi_k \right\rangle_{H^{1,0}_0(\Omega,\C)}  \\
& + \phantom{a}_{(H^{1,0}_0(\Omega,\C))^\star}\left\langle 
G'(0,\tfrac{1}{2})[(a,\varepsilon)] (B_{(0,\sfrac{1}{2})} \varphi_j ), 
\varphi_k  \right\rangle_{H^{1,0}_0(\Omega,\C)} = 0.
\end{align*}
Since by definition $G_{(0,\sfrac{1}{2})} = (i \nabla + A_0^{\sfrac{1}{2}})^2$
\begin{equation*}
\begin{split}
& \phantom{a}_{(H^{1,0}_0(\Omega,\C))^\star}\left\langle
G_{(0,\sfrac{1}{2})} ( B'(0,\tfrac{1}{2})[(a,\varepsilon)] \varphi_j ), 
\varphi_k \right\rangle_{H^{1,0}_0(\Omega,\C)}  \\
& \qquad \qquad \qquad \qquad
= \int_{\Omega} (i \nabla + A_0^{\sfrac{1}{2}}) 
( B'(0,\tfrac{1}{2})[(a,\varepsilon)] \varphi_j )
\cdot \overline{(i \nabla + A_0^{\sfrac{1}{2}}) \varphi_k},
\end{split}
\end{equation*}
the scalar product in $H^{1,0}_0(\Omega,\C)$ and since $\varphi_j$ is 
an eigenfunction of $B_{(0,\sfrac{1}{2})}$ of eigenvalue $\lambda^{-1}$, 
the claim follows. This holds also true for any $(b,t) \in \R^2 \times \R$ 
by linearity.
\end{proof}

Therefore, an immediate consequence of Lemmas~\ref{lemma:first-order-G-a} 
and \ref{lemma:relation-matrices} is the following.

\begin{lemma} \label{lemma:first-order-B-a}
Let $\Omega \subset \R^2$ be open, bounded, simply connected and Lipschitz. 
Let $B'(0,\tfrac{1}{2})$ be defined as in \eqref{eq:Bprim}. Let $\lambda 
\in \mathbb{R}^+$ be an eigenvalue of $(i \nabla + A_0^{\sfrac{1}{2}})^2$ 
of multiplicity $\mu \geq 1$ and $\varphi_j$, $j = 1, \ldots, \mu$, be 
the corresponding eigenfunctions orthonormalized in $L^2(\Omega,\C)$. 
Then for any $j,k = 1, \ldots, \mu$ and $(b,t) \in \R^2 \times \R$
\begin{equation} \label{eq:BON1}
\begin{split}
& \int_{\Omega} (i \nabla + A_0^{\sfrac{1}{2}}) 
( B'(0,\tfrac{1}{2})[(b,t)] \varphi_j )
\cdot \overline{(i \nabla + A_0^{\sfrac{1}{2}}) \varphi_k}  \\
& \qquad \qquad =  - \lambda^{-1} \left( \int_{\partial \Omega} ( b \cdot \nu ) 
\frac{\partial \varphi_j }{\partial \nu} 
\overline{ \frac{ \partial \varphi_k}{\partial \nu}} 
+ 4 t \int_{\Omega} (i \nabla + A_0^{\sfrac{1}{2}}) \varphi_j 
\cdot  A_0^{\sfrac{1}{2}} \overline{\varphi_k} \right).
\end{split}
\end{equation}
where $\nu \, : \, \partial \Omega \to \mathbb{S}^1$ is the exterior 
normal vector to $\partial\Omega$.
\end{lemma}

Expression \eqref{eq:BON1} is exactly the one we need to consider in 
$(ii)$ of Theorem~\ref{theorem:abstract-complex}.

\subsubsection{Expression of \eqref{eq:BON1} using the local properties of the eigenfunctions \eqref{eq:coeff}}

It happens that the first term in \eqref{eq:BON1} can be rewritten 
using the local properties of the eigenfunctions near $0$, i.e.\ 
as an expression involving the coefficients of $\varphi_j$, $j=1, 
\ldots, \mu$, in \eqref{eq:coeff}.

\begin{lemma} \label{lemma:symmetry}
Let $\Omega \subset \R^2$ be open, bounded, simply connected and Lipschitz. 
Let $\lambda \in \mathbb{R}^+$ be an eigenvalue of $(i \nabla + A_0^{\sfrac{1}{2}})^2$ 
of multiplicity $\mu \geq 1$ and $\varphi_j$, $j = 1, \ldots, \mu$, be the corresponding 
$K_0$-real eigenfunctions orthonormalized in $L^2(\Omega,\C)$. Let $c_j, \, d_j \in \R$ be the 
coefficients of $\varphi_j$ given in \eqref{eq:coeff}. Then for any $j,k = 1, \ldots,\mu$ 
and $b = (b_1, b_2) \in \R^2$
\[
\int_{\partial \Omega} ( b \cdot \nu ) \frac{\partial \varphi_j}{\partial \nu} 
\overline{\frac{\partial \varphi_k}{\partial \nu}} 
= \frac{\pi}{2} \left[ (c_j c_k - d_j d_k) b_1 + (c_j d_k + c_k d_j) b_2 \right], 
\]
where $\nu \, : \, \partial \Omega \to \mathbb{S}^1$ is the exterior normal 
to $\Omega$.
\end{lemma}

\begin{proof}
\begin{claim}
We first prove that
\begin{align*}
\int_{\partial \Omega} ( b \cdot \nu ) \frac{\partial \varphi_j}{\partial \nu} 
\overline{\frac{\partial \varphi_k}{\partial \nu}} 
= \lim_{\delta \to 0} 
& \Big( \int_{\partial D_\delta(0)} (i \nabla + A_0^{\sfrac{1}{2}}) 
\varphi_j \cdot \nu 
\, \overline{(i \nabla + A_0^{\sfrac{1}{2}}) \varphi_k \cdot b} \\ 
& \qquad \qquad + \int_{\partial D_\delta(0)} \varphi_j 
\, \overline{(i\nabla +A_0^{\sfrac{1}{2}}) 
[ (i \nabla + A_0^{\sfrac{1}{2}}) \varphi_k \cdot b]} \, \Big),
\end{align*}
where $\nu \, : \, \partial \Omega \to \mathbb{S}^1$ or $\nu \, : \, 
\partial D_\delta (0) \to \mathbb{S}^1$ are respectively the exterior 
normal to $\partial \Omega$ or to $\partial D_\delta (0)$.
\end{claim}

\begin{proofclaim}
We test the equation satisfied by $\varphi_j$ in $\Omega 
\setminus D_\delta(0)$ on $(i \nabla + A_0^{\sfrac{1}{2}}) 
\varphi_k \cdot b$ and take the limit $\delta \to 0$, and 
we notice that everything is well defined since we remove 
a small set containing the singular point $0$, see \eqref{eq:regul},
\begin{align*}
& 0 = \lim_{\delta \to 0} \int_{\Omega \setminus D_\delta(0)}
[ (i \nabla + A_0^{\sfrac{1}{2}})^2 - \lambda] \varphi_j \, 
\overline{(i \nabla + A_0^{\sfrac{1}{2}}) \varphi_k \cdot b}  \\
& = \lim_{\delta \to 0} \int_{\Omega \setminus D_\delta (0)} 
(i \nabla + A_0^{\sfrac{1}{2}}) \varphi_j \cdot 
\overline{(i \nabla + A_0^{\sfrac{1}{2}}) [ (i \nabla + A_0^{\sfrac{1}{2}}) \varphi_k \cdot b]} 
- \lambda \varphi_j \, \overline{(i \nabla + A_0^{\sfrac{1}{2}}) \varphi_k \cdot b} \\
& + \lim_{\delta \to 0 } \Big( i \int_{\partial \Omega} (i \nabla + A_0^{\sfrac{1}{2}}) \varphi_j \cdot \nu \, 
\overline{(i \nabla + A_0^{\sfrac{1}{2}}) \varphi_k \cdot b}
- i \int_{\partial D_\delta (0)} (i \nabla + A_0^{\sfrac{1}{2}}) \varphi_j \cdot \nu \, 
\overline{(i \nabla + A_0^{\sfrac{1}{2}}) \varphi_k \cdot b} \Big) \\
& = \lim_{\delta \to 0} \int_{\Omega \setminus D_\delta (0)}
\varphi_j \, \overline{(i \nabla + A_0^{\sfrac{1}{2}})^2 [ (i \nabla + A_0^{\sfrac{1}{2}}) \varphi_k \cdot b]} 
- \lambda \varphi_j \, \overline{(i \nabla + A_0^{\sfrac{1}{2}}) \varphi_k \cdot b} \\
& + \lim_{\delta \to 0 } \Big( i \int_{\partial \Omega} (i \nabla + A_0^{\sfrac{1}{2}}) \varphi_j \cdot \nu \, 
\overline{(i \nabla + A_0^{\sfrac{1}{2}}) \varphi_k \cdot b} 
- i \int_{\partial D_\delta(0)} (i \nabla + A_0^{\sfrac{1}{2}}) \varphi_j \cdot \nu \, 
\overline{(i \nabla + A_0^{\sfrac{1}{2}}) \varphi_k \cdot b} \Big) \\
& - \lim_{\delta \to 0} i \int_{\partial D_\delta (0)} 
\varphi_j \, \overline{(i \nabla + A_0^{\sfrac{1}{2}}) [ (i \nabla + A_0^{\sfrac{1}{2}}) \varphi_k \cdot b]} \cdot \nu.
\end{align*}
We have that 
\begin{align*}
(i \nabla + A_0^{\sfrac{1}{2}})^2 
& \left[ (i \nabla + A_0^{\sfrac{1}{2}}) \varphi_j \cdot b \right] \\
& = \sum_{k=1}^2 (i \nabla + A_0^{\sfrac{1}{2}})^2 
( i \partial_k \varphi_j b_k + A_{0,k}^{\frac{1}{2}} \varphi_j b_k ) 
= (i \nabla + A_0^{\sfrac{1}{2}}) 
\left[ (i \nabla + A_0^{\sfrac{1}{2}})^2 \varphi_j \right] \cdot b \\
& + \sum_{k,l =1}^2 \Big[ 2 (\partial_k A_{0,l}^{\frac{1}{2}} 
- \partial_l A_{0,k}^{\frac{1}{2}}) \partial_l \varphi_j b_k 
+ 2 i b_k A_{0,l}^{\frac{1}{2}} (\partial_l A_{0,k}^{\frac{1}{2}} 
- \partial_k A_{0,l}^{\frac{1}{2}}) \varphi_j 
- \partial_l^2 A_{0,k}^{\frac{1}{2}} b_k \varphi_j \Big]  \\
& = \lambda \, \left[ (i \nabla + A_0^{\sfrac{1}{2}}) \varphi_j \cdot b \right],
\end{align*}
since $\nabla \cdot A_0^{\sfrac{1}{2}} = 0$ and $\nabla \times A_0^{\sfrac{1}{2}} = 0$ in 
$\Omega \setminus \{0\}$. Therefore, the first two terms cancel and this proves the claim 
since $\varphi_j = \varphi_k = 0$ on $\partial \Omega$.
\end{proofclaim}

\smallskip

To prove the lemma, we use the explicit expression of \eqref{eq:coeff}. 
First we compute for $j = 1, \ldots, \mu$
\[
(i \nabla + A_0^{\sfrac{1}{2}}) \varphi_j 
= \frac{i}{2} e^{i \frac{t}{2}} r^{- 1/2} 
\Big( c_j \cos \frac{t}{2} - d_j \sin \frac{t}{2}, 
c_j \sin \frac{t}{2} + d_j \cos \frac{t}{2} \Big) 
+ R_1(r,t),
\]
where $R_1(r,t) = o(r^{-1/2})$ as $r\to 0^+$ uniformly with respect 
to $t\in[0,2\pi]$. Then, if $\nu = (\cos t, \sin t)$ is the exterior 
normal to $\partial D_{\delta}(0)$
\begin{equation} \label{eq:eq1}
(i \nabla + A_0^{\sfrac{1}{2}}) \varphi_j \cdot \nu 
= \frac{i}{2} e^{i \frac{t}{2}} r^{-1/2} 
\Big( c_j \cos \frac{t}{2} + d_j \sin \frac{t}{2} \Big) 
+ R_2(r,t),
\end{equation}
where $R_2(r,t)=o(r^{-1/2})$ as $r\to 0^+$ uniformly with respect 
to $t\in[0,2\pi]$, while if $b = (b_1, b_2)$ it holds that 
\begin{equation} \label{eq:eq2}
(i \nabla + A_0^{\sfrac{1}{2}}) \varphi_j \cdot b 
= \frac{i}{2} e^{i \frac{t}{2}} r^{-1/2} 
\Big( c_j b_1 \cos \frac{t}{2} - d_j b_1 \sin \frac{t}{2} 
+ c_j b_2 \sin \frac{t}{2} + d_j b_2 \cos \frac{t}{2} \Big) 
+  R_3(r,t),
\end{equation}
where $R_3(r,t) = o(r^{-1/2})$ as $r\to 0^+$ uniformly with respect 
to $t\in[0,2\pi]$. Finally,
\begin{align*}
(i \nabla + A_0^{\sfrac{1}{2}}) 
\left[ (i\nabla + A_0^{\sfrac{1}{2}}) \varphi_j \cdot b \right] 
& = \frac{1}{4} e^{i \frac{t}{2}} r^{-3/2} 
\Big( c_j b_1 \cos \frac{3 t}{2} - d_j b_1 \sin \frac{3 t}{2} 
+ c_j b_2 \sin \frac{3 t}{2} + d_j b_2 \cos \frac{3 t}{2},\\
& c_j b_1 \sin \frac{3 t}{2} + d_j b_1 \sin \frac{3 t}{2} 
- c_j b_2 \cos \frac{3 t}{2} + d_j b_2 \sin \frac{3 t}{2} \Big) 
+ R_4(r,t),
\end{align*}
where $R_4(r,t) = o(r^{-3/2})$ as $r\to 0^+$ uniformly with respect 
to $t\in[0,2\pi]$, and
\begin{equation} \label{eq:eq3}
\begin{split}
(i \nabla + A_0^{\sfrac{1}{2}}) 
& [ (i\nabla + A_0^{\sfrac{1}{2}}) \varphi_j \cdot b] \cdot \nu  \\
& = \frac{1}{4} e^{i \frac{t}{2}} r^{-3/2} 
\Big( c_j b_1 \cos \frac{t}{2} - d_j b_1 \sin \frac{t}{2}
+ c_j b_2 \sin \frac{t}{2} + d_j b_2 \cos \frac{t}{2} \Big) 
+ R_5(r,t),
\end{split}
\end{equation}
where $R_5(r,t) = o(r^{-3/2})$ as $r\to 0^+$ uniformly with respect 
to $t\in[0,2\pi]$. Then using \eqref{eq:eq1} and \eqref{eq:eq2} and 
elementary calculations we have that
\begin{equation} \label{eq:eq4}
\lim_{\delta \to 0} \int_{\partial D_\delta(0)} (i \nabla + A_0^{\sfrac{1}{2}}) 
\varphi_j \cdot \nu \, 
\overline{ ( i \nabla + A_0^{\sfrac{1}{2}}) \varphi_k \cdot b} 
= \frac{\pi}{4} \left[ (c_j c_k - d_j d_k) b_1 + (c_j d_k + c_k d_j) b_2 \right],
\end{equation}
and using \eqref{eq:coeff} and \eqref{eq:eq3}
\begin{equation} \label{eq:eq5}
\lim_{\delta \to 0} \int_{\partial D_\delta(0)} \varphi_j \, 
\overline{(i \nabla + A_0^{\sfrac{1}{2}}) [ (i\nabla + A_0^{\sfrac{1}{2}}) 
\varphi_k \cdot b] \cdot \nu} 
= \frac{\pi}{4} \left[ (c_j c_k - d_j d_k) b_1 + (c_j d_k + c_k d_j) b_2 \right].
\end{equation}
Summing \eqref{eq:eq4} and \eqref{eq:eq5} gives the lemma.
\end{proof}

\resetclaim

We are not able to give an explicit expression of the second term in 
\eqref{eq:BON1}, as we have for the first one, see Lemma~\ref{lemma:symmetry}. 
However, we can say something using explicitly the \emph{real} structure 
of the operator, and more precisely the $K_0$-reality of the eigenfunctions in 
\eqref{eq:coeff}.

\begin{lemma} \label{lemma:antisymmetry}
Let $\lambda \in \R^+$ be an eigenvalue of $(i \nabla + A_0^{\sfrac{1}{2}})^2$ 
of multiplicity $\mu \geq 1$, and let $\varphi_j$, $j = 1, \ldots, \mu$, be the 
corresponding $K_0$-real eigenfunctions orthonormalized in $L^2(\Omega,\C)$. Let
\[
i R_{jk} := 4 \int_{\Omega} (i \nabla + A_0^{\sfrac{1}{2}}) \varphi_j \cdot A_0^{\sfrac{1}{2}} \overline{\varphi_k}.
\]
Then for any $j,k = 1, \ldots, \mu$
\[
\overline{R_{jk}} = R_{jk} \quad \text{ i.e.\ } R_{jk} \text{ is real valued}
\]
and 
\[
R_{jk} = - R_{kj}.
\]
\end{lemma}

\begin{proof}
The proof of this lemma relies strongly on the $K_0$-reality of the 
eigenfunctions. Using first \eqref{eq:relation-2} and next an integration 
by part and the fact that $\nabla \cdot A_0^{\sfrac{1}{2}} = 0$ in 
$\Omega \setminus\{0\}$, we have that
\[
\int_{\Omega} (i \nabla + A_0^{\sfrac{1}{2}}) \varphi_j 
\cdot A_0^{\sfrac{1}{2}} \overline{\varphi_k} 
 = - \int_{\Omega} \overline{(i \nabla + A_0^{\sfrac{1}{2}}) \varphi_j} 
 \cdot A_0^{\sfrac{1}{2}} \varphi_k 
 = - \int_{\Omega}  (i \nabla + A_0^{\sfrac{1}{2}}) \varphi_k 
 \cdot A_0^{\sfrac{1}{2}} \overline{\varphi_j}. 
\]
This proves the lemma.
\end{proof}

From Lemma~\ref{lemma:antisymmetry} we immediately see that 
$R_{jj} = 0$ for $j=1,\ldots,\mu$.

\subsection{Proof of Theorem~\ref{t:main}} \label{subsec:proofmain}

In Theorem~\ref{theorem:abstract-complex}, the Banach space $B$ is given by 
$\R^2 \times \R$ and the fix point in $B$ is $(0,\tfrac{1}{2})$. The Hilbert 
space $X$ is $H^{1,0}_0(\Omega,\C)$ and the family of compact self-adjoint 
linear operators is given by $\{ B_{(a,\alpha)} \, : \,  H^{1,0}_0(\Omega,\C) \to 
H^{1,0}_0(\Omega,\C) \, : \, (a,\alpha) \in D_R(0) \times \mathcal W\}$, 
being $\mathcal W$ a small neighborhood of $1/2$. The non perturbed operator 
is $B_{(0,\sfrac{1}{2})} = [(i \nabla + A_0^{\sfrac{1}{2}})^2]^{-1} 
\circ \text{Im}_{H^{1,0}_0(\Omega,\C) \to ( H^{1,0}_0(\Omega,\C) )^\star}$. 
We assume to have an eigenvalue $\lambda \in \R^+$ of $(i \nabla + A_0^{\sfrac{1}{2}})^2$ 
(and therefore an eigenvalue $\lambda^{-1} \in \R^+$ of $B_{(0,\sfrac{1}{2})}$) 
of multiplicity $\mu = 2$, and two corresponding $K_0$-real eigenfunctions 
$\varphi_j$, $j = 1,2$, orthonormalized in $L^2(\Omega,\C)$ and verifying \eqref{eq:coeff}. 

\smallskip

Lemma~\ref{l:BaC1} tells us that condition $(i)$ of Theorem~\ref{theorem:abstract-complex} 
is satisfied. To prove condition $(ii)$ of Theorem~\ref{theorem:abstract-complex} 
it will be sufficient to prove that the function
$F \, : \, \R^2 \times \R \to L_h(\R^2,\R^2)$ given by 
\[
(b, t) \mapsto 
\left( \int_{\Omega} (i \nabla + A_0^{\sfrac{1}{2}}) 
( B'(0,\tfrac{1}{2})[(b,t)] \varphi_j ) \, 
\cdot \overline{(i \nabla + A_0^{\sfrac{1}{2}}) \varphi_k} 
\right)_{j,k=1,2}
\]
is such that
\[
\text{Im} F + [I] = L_h(\R^2,\R^2).
\]
This expression is exactly the one given by \eqref{eq:BON1}. Using 
\eqref{eq:BON1}, Lemmas~\ref{lemma:symmetry} and \ref{lemma:antisymmetry}, 
forgetting some non zero constants ($- \lambda^{-1}$ and $\tfrac{\pi}{4}$) 
for better readibility (this can be done through a renormalization of the 
parameters), we need to show that the application sending $(b, t, \mu) \in 
\R^2 \times \R \times \R$ on
\begin{align*}
\begin{pmatrix}
(c_1^2 - d_1^2) b_1 + 2 c_1 d_1 b_2 + \mu 
& (c_1c_2 - d_1 d_2) b_1 + (c_1 d_2 + c_2 d_1) b_2 + i t R_{12} \\
(c_1c_2 - d_1 d_2) b_1 + (c_1 d_2 + c_2 d_1) b_2 - i t R_{12} 
& (c_2^2 - d_2^2) b_1 + 2 c_2 d_2 b_2 + \mu
\end{pmatrix}
\end{align*}
gives all the  $2 \times 2$ hermitian matrices; or equivalently that 
the application sending $(b,\mu) \in \R^2 \times \R$ on
\begin{align*}
\begin{pmatrix}
(c_1^2 - d_1^2) b_1 + 2 c_1 d_1 b_2 + \mu 
& (c_1c_2 - d_1 d_2) b_1 + (c_1 d_2 + c_2 d_1) b_2 \\
(c_1c_2 - d_1 d_2) b_1 + (c_1 d_2 + c_2 d_1) b_2 
& (c_2^2 - d_2^2) b_1 + 2 c_2 d_2 b_2 + \mu
\end{pmatrix}
\end{align*}
gives all the $2 \times 2$ symmetric matrices, since $c_j, \, d_j \in \R$ 
for $j = 1,2$ by \eqref{eq:coeff}, and the application sending $t \in \R$ 
on
\begin{align*}
\begin{pmatrix}
0 & t R_{12} \\
- t R_{12} & 0
\end{pmatrix}
\end{align*}
gives all the $2 \times 2$ antisymmetric matrices, since $R_{12} \in \R$ 
and $R_{11} = R_{22} = 0$ by Lemma~\ref{lemma:antisymmetry}.

Those matrices can be rewritten in a more suitable way. Equation 
\eqref{eq:coeff} also reads for $j = 1,2$
\[
\varphi_j(r (\cos t, \sin t)) = m_j e^{i \frac{t}{2}} r^{1/2} 
\cos \frac{t - \alpha_j}{2} + f_j(r,t), 
\]
where $f_j(r,t) = o(r^{1/2})$ as $r \to 0^+$ uniformly in $t \in [0,2\pi]$, 
and
\[
c_j = m_j \cos \frac{\alpha_j}{2} \quad \text{ and } 
\quad d_j = m_j \sin \frac{\alpha_j}{2},
\]
with $\alpha_j \in [0, 2\pi)$ and $m_j \in \R$ possibly zero. 
We notice that if $m_j \neq 0$, then $c_j^2 + d_j^2 \neq 0$ and the 
eigenfunction $\varphi_j$ has a zero of order $1/2$ at $0$, i.e.\ a 
unique nodal line ending at $0$. The angle of such a nodal line is 
related to $\alpha_j$ by
\[
\text{angle of the nodal line of } \varphi_j = \alpha_j + \pi + 2 k \pi, 
\quad k \in \mathbb{Z}.
\]
Using this new expression, our first $2 \times 2$ symmetric matrix writes 
\[
\begin{pmatrix}
m_1^2 ( \cos \alpha_1 b_1 + \sin \alpha_1 b_2 ) + \mu 
& m_1 m_2 ( \cos \frac{\alpha_1 + \alpha_2}{2} b_1 
+ \sin \frac{\alpha_1 + \alpha_2}{2} b_2 ) \\
m_1 m_2 (\cos \frac{\alpha_1 + \alpha_2}{2} b_1  
+ \sin \frac{\alpha_1 + \alpha_2}{2} b_2 ) 
& m_2^2 ( \cos \alpha_2 b_1 + \sin \alpha_2 b_2 ) + \mu
 \end{pmatrix}.
\]
Asking that such a matrix gives all $2\times2$ symmetric real 
matrices is equivalent to ask the following matrix to be 
surjective in $\R^3$ 
\[
M:=\begin{pmatrix}
m_1^2 \cos\alpha_1 
& m_1^2 \sin\alpha_1 
& 1 \\
m_2^2 \cos \alpha_2 
& m_2^2 \sin\alpha_2 
& 1 \\
m_1 m_2 \cos \frac{\alpha_1 + \alpha_2}{2} 
& m_1 m_2 \sin \frac{\alpha_1 + \alpha_2}{2} 
& 0
\end{pmatrix}.
\]
This will be the case if and only if $\mathrm{det}M \neq 0$, that is
\[
\mathrm{det}M =  m_1 m_2 (m_1^2 + m_2^2) 
\sin \frac{(\alpha_1-\alpha_2)}{2} \neq 0. 
\]
This happens only if the following conditions are satisfied
\begin{itemize}
\item[$(1)$] $m_1 \neq 0$ and $m_2 \neq 0$,
\item[$(2)$] $\alpha_1 \neq \alpha_2 + 2k \pi$, $k \in \mathbb{Z}$.
\end{itemize}
Those conditions mean that they do not exist a system of orthonormal 
eigenfunctions such that at least one has a zero of order strictly 
greater than $1/2$ at $0$, i.e.\ more than one nodal line ending at $0$.
In term of the coefficients $c_j$ and $d_j$, $j=1,2$, the above 
conditions can be rewritten as
\begin{itemize}
\item[$(1')$] $c_j^2 + d_j^2 \neq 0$, $j=1,2$,
\item[$(2')$] there does not exist $\gamma \in \R$ such that $(c_1,d_1) = \gamma (c_2,d_2)$.
\end{itemize}

To prove that the second matrix gives all $2 \times 2$ antisymmetric 
real matrices, it is sufficient to ask 
\begin{itemize}
\item[$(3)$] $R_{12} \neq 0$.
\end{itemize} 
Therefore, Theorem~\ref{theorem:abstract-complex} may be applied if 
conditions $(1)$--$(3)$ are all satisfied, or equivalently $(1')$--$(3)$.

\appendix

\section{Proof of Lemma~\ref{l:BaC1}} \label{app:proof}

This Lemma is proved in five claims. For this, we follow closely the argument 
presented in \cite{LupoMicheletti1993}.

We call here $\mathcal{W} \subset (0,1)$ any neighborhood of $\{\tfrac{1}{2}\}$ 
such that $\overline{\mathcal{W}} \subset\subset (0,1)$.
\begin{claim} 
Let $(a,\alpha) \in D_R(0) \times \mathcal{W}$. We claim that the map 
$(a,\alpha) \mapsto E_{(a,\alpha)}$ is $C^1 \Big(D_R(0) \times \mathcal{W}, 
BL \big (H^{1,0}_0(\Omega,\C) \times H^{1,0}_0(\Omega,\C), \C \big) \Big)$.
\end{claim}

\begin{proofclaim}
We consider any $\hat{u}, \, \hat{v} \in H^{1,0}_0(\Omega,\C)$. By definition 
of $G_{(a,\alpha)}$ in \eqref{eq:operator-G-a} and the fact that $\gamma_a$, 
defined in \eqref{eq:gamma_a}, is an isomorphism in $L^2(\Omega,\C)$ we see that 
\begin{align*}
E_{(a,\alpha)} (\Hat{u},\Hat{v})
& = \int_{\Omega} (i\nabla + A_a^\alpha)(\gamma_a^{-1} (\hat{u})) \cdot 
\overline{(i\nabla + A_a^\alpha)(\gamma_a^{-1}(\hat{v}))} 
= \int_{\Omega} (i\nabla + A_a^\alpha)^2(\gamma_a^{-1}(\Hat u)) \, 
\overline{ (\gamma_a^{-1} (\Hat v)) } \\
& = \int_{\Omega} \gamma_a^{-1} \circ G_{(a,\alpha)} \Hat u \, 
\overline{ (\gamma_a^{-1}(\Hat v)) } 
= \int_{\Omega} G_{(a,\alpha)} \Hat u \,\overline{\Hat v}
\end{align*}
From Lemma \ref{lemma:operator-G-a}, the conclusion follows. 

Moreover we know from Lemma~\ref{lemma:operator-G-a} that $G_{(a,\alpha)}$ 
is $C^1$ in $D_R(0) \times (0,1)$. Therefore, for any $\alpha \in (0,1)$, 
there exists $G'(0,\alpha)$ such $G_{(a,\alpha)} = (i \nabla + A_0^{\alpha})^2 
+ G'(0,\alpha)[(a,0)] + o(|(a,0)|)$, for $|(a,0)| \to 0$. If $\alpha$ is 
sufficiently far from the integers $0$ and $1$, that is if $\alpha \in 
\mathcal{W}$, there exists $K > 0$ independent of $(a,\alpha) \in D_R(0) 
\times \mathcal{W}$ such that 
\begin{equation}\label{eq:rem5}
K \|\Hat u\|_{H^{1,0}_0(\Omega,\C)}^2 \leq E_{(a,\alpha)}(\Hat u,\Hat u) 
\quad \forall \Hat u\in H^{1,0}_0(\Omega,\C) .
\end{equation}
\end{proofclaim}

Therefore, we denote by $E'(a_0,\alpha_0)[(a,\omega)]$ the Fréchet derivative 
of $E_{(a,\alpha)}$ at $(a_0,\alpha_0) \in D_R(0) \times (0,1)$ applied to 
$(a,\omega)$, and by $R(a_0,\alpha_0)[(a,\omega)] = o(|(a,\omega)|)$ the 
remainder.

\smallskip
 
\begin{claim}
For any $(a,\alpha) \in D_R(0) \times \mathcal{W}$, we claim that 
\begin{equation}\label{eq:rem6}
\| B_{(a,\alpha)} \Hat u \|_{H^{1,0}_0(\Omega,\C)} \leq C \|\Hat u\|_{H^{1,0}_0(\Omega,\C)},
\end{equation}
for some constant $C > 0$ independent of $(a,\alpha)$.
\end{claim}

\begin{proofclaim}
By definition of $B_{(a,\alpha)}$, $E_{(a,\alpha)}$ and \eqref{eq:rem5}, 
for $a\in D_R(0) \in (0,1)$ we have
\begin{align*}
K \|B_{(a,\alpha)} \Hat u \|_{H^{1,0}_0(\Omega,\C)}^2 
\leq E_{(a,\alpha)} (B_{(a,\alpha)} \Hat u , B_{(a,\alpha)} \Hat u) 
& = (\Hat u , B_{(a,\alpha)} \Hat u)_{L^2(\Omega,\C)} \\
& \leq \|\Hat u\|_{H^{1,0}_0(\Omega,\C)} \, \|B_{(a,\alpha)} \Hat u \|_{H^{1,0}_0(\Omega,\C)}.
\end{align*}
The claim follows immediately from it.
\end{proofclaim}

\smallskip

\begin{claim}
Let $(a,\alpha) \in D_R(0) \times \mathcal{W}$. We claim that the map 
$(a,\alpha) \mapsto B_{(a,\alpha)}$ is $C^0 \big( D_R(0) \times \mathcal{W}, 
BL( H^{1,0}_0(\Omega,\C), H^{1,0}_0(\Omega,\C) \big)$.
\end{claim}

\begin{proofclaim} 
We follow \cite[Lemma 5]{LupoMicheletti1993}.
For any $\Hat u,\,\Hat v \in H^{1,0}_0(\Omega,\C)$ we have for 
$(a_0,\alpha_0) \in D_R(0) \times (0,1)$ and $(a_0 + a, \alpha_0 + \omega) 
\in D_R(0) \times \mathcal{W}$ and by \eqref{eq:inverse-operator}
\begin{align*}
(\Hat u,\Hat v)_{L^2(\Omega,\C)} 
& = E_{(a_0+a,\alpha_0 + \omega)} (B_{(a_0+a,\alpha_0+\omega)} \Hat u,\Hat v) \\
& = E_{(a_0,\alpha_0)} (B_{(a_0+a,\alpha_0+\omega)} \Hat u, \Hat v) 
+ E'(a_0,\alpha_0)[(a,\omega)] (B_{(a_0+a,\alpha_0+\omega)} \Hat u,\Hat v) 
+ o(|(a,\omega)|) \\
& = E_{(a_0,\alpha_0)} (B_{(a_0,\alpha_0)} \Hat u,\Hat v) 
+ E_{(a_0,\alpha_0)}((B_{(a_0+a,\alpha_0+\omega)} - B_{(a_0,\alpha_0)}) \Hat u,\Hat v) \\
& + E'(a_0,\alpha_0)[(a,\omega)](B_{(a_0+a,\alpha_0+\omega)}\Hat u,\Hat v) 
+ o(|(a,\omega)|).
\end{align*}
Since by definition $(\Hat u,\Hat v)_{L^2(\Omega,\C)} = 
E_{(a_0,\alpha_0)}(B_{(a_0,\alpha_0)}\Hat u,\Hat v)$ we obtain
\begin{equation}\label{eq:lem5}
E_{(a_0,\alpha_0)}((B_{(a_0+a,\alpha_0+\omega)} - B_{(a_0,\alpha_0)}) 
\Hat u,\Hat v) = - E'(a_0,\alpha_0)[(a,\omega)](B_{(a_0+a,\alpha_0+\omega)} 
\Hat u,\Hat v) + o(|(a,\omega)|).
\end{equation}
Considering $\Hat v = (B_{(a_0+a,\alpha_0+\omega)} - B_{(a_0,\alpha_0)})\Hat u$ 
and using \eqref{eq:rem5} and \eqref{eq:rem6}, the latter relation reads
\[
\|(B_{(a_0+a,\alpha_0+\omega)} - B_{(a_0,\alpha_0)}) 
\Hat u\|_{{H^{1,0}_0(\Omega,\C)}} \leq c(a_0,\alpha_0) |(a,\omega)| 
\|\Hat u\|_{{H^{1,0}_0(\Omega,\C)}},
\]
for some $c(a_0,\alpha_0) > 0$ depending only on $(a_0,\alpha_0)$.
\end{proofclaim}

\smallskip

\begin{claim}
For any $(a_0,\alpha_0) \in D_R(0) \times \mathcal{W}$, the map $(a,\alpha) 
\mapsto B_{(a,\alpha)}$ is Fréchet differentiable at $(a_0,\alpha_0)$. 
Moreover, if we write $B'(a_0,\alpha_0)[(a,\omega)]$ the Fréchet derivative 
of $B_{(a,\alpha)}$ at $(a_0,\alpha_0)$ applied to $(a,\omega)$, it 
holds for any $\Hat{u}, \Hat{v} \in H^{1,0}_0(\Omega,\C)$
\[
E_{(a_0,\alpha_0)} (B'(a_0,\alpha_0)[(a,\omega)] \Hat{u}, \Hat{v} ) 
= - E'(a_0,\alpha_0)[(a,\omega)](B_{(a_0,\alpha_0)} \Hat{u},\Hat{v}).
\]
\end{claim}

\begin{proofclaim}
We follow the proof of \cite[Lemma 6]{LupoMicheletti1993}. Let us consider 
$(a_0, \alpha_0), \, (a_0 + a, \alpha_0 + \omega) \in D_R(0) \times \mathcal{W}$, 
and $\Hat u \in H^{1,0}_0(\Omega,\C)$. For any $\Hat v \in H^{1,0}_0(\Omega,\C)$, 
we consider the map $\Hat v \mapsto E'(a_0,\alpha_0)[(a,\omega)](B_{(a_0,\alpha_0)} 
\Hat u,\Hat v) \in \C$. By the properties of $E_{(a,\alpha)}$ and Riesz's 
Theorem, it is defined a sesquilinear and continuous map $L_{(a_0,\alpha_0)} 
\, : \, \R^2 \times \R \times H^{1,0}_0(\Omega,\C) \to 
H^{1,0}_0(\Omega,\C)$ such that
\begin{equation}\label{eq:La}
- E'(a_0,\alpha_0)[(a,\omega)](B_{(a_0,\alpha_0)} \Hat u,\Hat v) = 
E_{(a_0,\alpha_0)}(L_{(a_0,\alpha_0)}(a,\omega,\Hat u),\Hat v). 
\end{equation}
We are now proving that for every fixed $(a_0,\alpha_0) \in D_R(0) \times 
\mathcal{W}$ and fixed a normalized $\Hat u \in H^{1,0}_0(\Omega,\C)$ we have
\[
\lim_{|(a,\omega)| \to 0} \frac{\| (B_{(a_0+a,\alpha_0+\omega)} 
- B_{(a_0,\alpha_0)}) \Hat u - L_{(a_0,\alpha_0)}(a,\omega,\Hat u) 
\|_{H^{1,0}_0(\Omega,\C)}}{|(a,\omega)|} = 0
\]
uniformly with respect to $\Hat u$. Indeed, denoting $\Hat w:= 
(B_{(a_0+a,\alpha_0+\omega)} - B_{(a_0,\alpha_0)}) \Hat u 
- L_{(a_0,\alpha_0)}(a,\omega,\Hat u)$, by \eqref{eq:rem5}, \eqref{eq:lem5} 
and \eqref{eq:La} we have
\begin{align*}
& K \| ( B_{(a_0+a,\alpha_0+\omega)} - B_{(a_0,\alpha_0)}) \Hat u 
- L_{(a_0,\alpha_0)} ( a,\omega,\Hat u ) \|_{ H^{1,0}(\Omega,\C)}^2 \\
& \leq E_{(a_0,\alpha_0)}( ( B_{(a_0+a,\alpha_0+\omega)} - B_{(a_0,\alpha_0)}) \Hat u 
- L_{(a_0,\alpha_0)}(a,\omega,\Hat u) , \Hat w) \\
& = - E'(a_0,\alpha_0)[(a,\omega)]( B_{(a_0+a,\alpha_0+\omega)} \Hat u, \Hat w) 
+ E'(a_0,\alpha_0)[(a,\omega)]( B_{(a_0,\alpha_0)} \Hat u,\Hat w) + o(|(a,\omega)|) \\
& = - E'(a_0,\alpha_0)[(a,\omega)]( ( B_{(a_0+a,\alpha_0+\omega)} - B_{(a_0,\alpha_0)} ) 
\Hat u,\Hat w) + o(|(a,\omega)|),
\end{align*}
from which the thesis follows. We then have that
\[
L_{(a_0,\alpha_0)}(a,\omega, \Hat{u}) = B'(a_0,\alpha_0)[(a,\omega)] \Hat{u}.
\]
\end{proofclaim}

\smallskip

\begin{claim}
We claim that the map $(a,\alpha) \mapsto B_{(a,\alpha)}$ is 
$C^1 \big( D_R(0)\times \mathcal{W} ,  BL( H^{1,0}_0(\Omega,\C), H^{1,0}_0(\Omega,\C) ) \big)$.
\end{claim}

\begin{proofclaim}
Here we follow \cite[Lemma 7]{LupoMicheletti1993}. Let us consider $(a_0,\alpha_0), 
\, (a_0 + a_1, \alpha_0 + \alpha_1) \in D_R(0) \times \mathcal{W}$
and denote $\Hat w:= ( B'(a_0 + a_1,\alpha_0 + \omega_1) - B'(a_0,\alpha_0))[(a,\omega)] 
\Hat u$ for a fixed $\Hat u \in H^{1,0}_0(\Omega,\C)$.
As before, by \eqref{eq:rem5}, \eqref{eq:La} we estimate
\begin{align*}
& K \| ( B'(a_0 + a_1,\alpha_0 + \omega_1) - B'(a_0,\alpha_0) ) [(a,\omega)] 
\Hat u \|_{H^{1,0}_0(\Omega,\C)}^2 \\
& \leq E_{(a_0+a_1,\alpha_0+\omega_1)}( ( B'(a_0 + a_1,\alpha_0+\omega_1)) 
- B'(a_0,\alpha_0) [(a,\omega)] \Hat u , \Hat w ) \\
& = E_{(a_0+a_1,\alpha_0+\omega_1)} (B'(a_0+a_1,\alpha_0+\omega_1)[(a,\omega)] 
\Hat u,\Hat w) - E_{(a_0+a_1,\alpha_0+\omega_1)} (B'(a_0,\alpha_0)[(a,\omega)] 
\Hat u,\Hat w) \\
& = - E'(a_0+a_1,\alpha_0+\omega_1)[(a,\omega)]( B_{(a_0+a_1,\alpha_0+\omega_1)} 
\Hat u, \Hat w) - E_{(a_0+a_1,\alpha_0+\omega_1)} (B'(a_0,\alpha_0)[(a,\omega)] 
\Hat u,\Hat w) \\
& = - E'(a_0+a_1,\alpha_0+\omega_1)[(a,\omega)] (B_{(a_0+a_1,\alpha_0+\omega_1} 
\Hat u,\Hat w) 
- E'(a_0,\alpha_0)[(a_1,\omega_1)] (B'(a_0,\alpha_0)[(a,\omega)] 
\Hat u,\Hat w)   \\
& + E'(a_0,\alpha_0)[(a,\omega)] (B_{(a_0,\alpha_0)} \Hat u,\Hat w) + o(|(a_1,\omega_1)|).
\end{align*}
By Claim 4 it holds 
\[
 B_{(a_0+a_1,\alpha_0+\omega_1)} \Hat u 
 = B_{(a_0,\alpha_0)} \Hat u + B'(a_0,\alpha_0)[(a_1,\omega_1)] \Hat u 
 + o(|(a_1,\omega_1)|), 
\]
so that we can proceed with
\begin{align*} 
& K \| ( B'(a_0+a_1,\alpha_0+\omega_1) - B'(a_0,\alpha_0) ) [(a,\omega)] 
\Hat u\|_{H^{1,0}_0(\Omega,\C)}^2 \\
& \leq - ( E'(a_0+a_1,\alpha_0+\omega_1)- E'(a_0,\alpha_0) )[(a,\omega)] 
(B_{(a_0,\alpha_0)} \Hat u, \Hat w) \\
& - E'(a_0 + a_1,\alpha_0+\omega_1) [(a,\omega)] ( B'(a_0,\alpha_0) [(a_1,\omega_1)] 
\Hat u, \Hat w) \\
& - E'(a_0,\alpha_0)[(a_1,\omega_1)] (B'(a_0,\alpha_0)[(a,\omega)] 
\Hat u, \Hat w) + o(|(a_1,\omega_1)|),
\end{align*}
from which the thesis follows.
\end{proofclaim}
The proof of Claim 5 concludes the proof of the whole lemma.

\resetclaim

\section{Proof of Claim~\ref{claim:general-expression-L-prim} in Lemma~\ref{lemma:first-order-La}} \label{app:Lprim}

Fix $\varepsilon := \alpha - \tfrac{1}{2}$. As a first step, we note that 
with simple calculations we can prove that
\begin{equation} \label{eq:Aa-A0}
A_a^{\alpha} \circ \Phi_a - A_0^\alpha 
= (\xi - 1) \nabla (a \cdot A_0^{\sfrac{1}{2}}) 
+ o(|(a,\varepsilon)|),
\end{equation}
as $|(a,\varepsilon)| \to 0$, in $L^\infty(\Omega)$.

\smallskip

\begin{claim} \label{claim:intermediate}
As an intermediate step, we prove that
\begin{align*}
\phantom{a}_{(H^{1,0}_0(\Omega,\C))^\star} & \left\langle 
\mathcal{L}'(0,\tfrac{1}{2})[(a,\varepsilon)] \varphi_j 
, \varphi_k \right\rangle_{H^{1,0}_0(\Omega,\C)} \\
& = \int_{\Omega} - \nabla \varphi_j \cdot \nabla (a \cdot \nabla \xi) 
\, \overline{\varphi_k} 
- ( \nabla \varphi_j \cdot \nabla \xi ) 
\, ( a \cdot \nabla \overline{\varphi_k} ) 
- (a \cdot \nabla \varphi_j ) 
\, ( \nabla \xi \cdot \nabla \overline{\varphi_k} ) \\
& + \int_{\Omega} i A_0^{\sfrac{1}{2}} \cdot 
\nabla ( a \cdot \nabla \xi) \, \varphi_j \, \overline{\varphi_k} 
+ i ( A_0^{\sfrac{1}{2}} \cdot \nabla \xi ) 
\, ( a \cdot \nabla \overline{\varphi_k} ) \, \varphi_j
- i ( A_0^{\sfrac{1}{2}} \cdot \nabla \xi ) 
\, (a \cdot \nabla \varphi_j) \, \overline{\varphi_k} \\
& + \int_{\Omega} i (\xi -1 ) \nabla \varphi_j 
\cdot \nabla (a \cdot A_0^{\sfrac{1}{2}}) \, \overline{\varphi_k} 
- i (\xi - 1) \nabla \overline{\varphi_k} 
\cdot \nabla (a \cdot A_0^{\sfrac{1}{2}} ) \, \varphi_j  \\
& + \int_{\Omega} 2 (\xi-1) A_0^{\sfrac{1}{2}} 
\cdot \nabla (a \cdot A_0^{\sfrac{1}{2}} ) \, \varphi_j \, \overline{\varphi_k} . 
\end{align*}
\end{claim}

\begin{proofclaim}
We look at every possible combinations of the terms appearing in 
Lemma~\ref{lemma:new-operator}, except for the first term
\[
\int_{\Omega} (i \nabla + A_0^\alpha) \varphi_j \cdot 
\overline{(i \nabla + A_0^\alpha) \varphi_k},
\]
which is not part of $\mathcal{L}_{(a,\alpha)}$ but represents the 
operator $(i \nabla + A_0^\alpha)^2$. The first term to consider is
\begin{equation*}
\int_{\Omega} (i \nabla + A_0^\alpha) v \cdot \overline{F(a,\alpha) w} 
= \int_{\Omega} (i \nabla + A_0^\alpha) v \cdot 
\overline{ \left[ (A_a^\alpha \circ \Phi_a - A_0) w - i J_a^{-1} (a \cdot \nabla w) \nabla \xi 
\right]}
\end{equation*}
The left part writes as
\[
(i \nabla + A_0^{\sfrac{1}{2}}) v + 2 \varepsilon A_0^{\sfrac{1}{2}} v,
\]
while the right parts is
\[
\overline{ \left[ (\xi - 1) \nabla (a \cdot A_0^{\sfrac{1}{2}} ) w 
- i (a \cdot \nabla w) \nabla \xi - i (1 - J_a) J_a^{-1} (a \cdot \nabla w) \nabla \xi 
+ o(|(a,\varepsilon)|) \right] },
\]
where the last $o(|(a,\varepsilon)|)$ is in $L^\infty(\Omega)$ as in \eqref{eq:Aa-A0}. 
From this we get the first order terms
\begin{equation} \label{eq:second}
\int_{\Omega} (i \nabla + A_0^{\sfrac{1}{2}}) v \cdot \nabla (a \cdot A_0^{\sfrac{1}{2}} ) (\xi - 1) \overline{w} 
+ i (i \nabla + A_0^{\sfrac{1}{2}}) v \cdot \nabla \xi \, (a \cdot \nabla \overline{w}),
\end{equation}
and the remainder terms are all bounded by $o(|(a,\varepsilon)|) \|w\|_{H^{1,0}_0(\Omega,\C)}$ 
as $|(a,\varepsilon)| \to 0$ because of \eqref{eq:xi} and \eqref{eq:det}.

\smallskip

The third term in Lemma~\ref{lemma:new-operator} to look at is
\begin{equation*}
\begin{split}
& \int_{\Omega} i (i \nabla + A_0^\alpha) v \cdot \nabla J_a \, J_a^{-1} \overline{w} \\
& \qquad \qquad \qquad = \int_{\Omega} i \left[ (i \nabla + A_0^{\sfrac{1}{2}}) v 
+ 2 \varepsilon A_0^{\sfrac{1}{2}} v \right] 
\cdot \left[ \nabla ( a \cdot \nabla \xi ) \overline{w} 
+ \nabla (a \cdot \nabla \xi ) (1 - J_a) J_a^{-1} \overline{w} \right].
\end{split}
\end{equation*}
The first order term is
\begin{equation} \label{eq:trois}
\int_{\Omega} i (i \nabla + A_0^{\sfrac{1}{2}} ) v \cdot 
\nabla (a \cdot \nabla \xi) \overline{w},
\end{equation}
and the rest is bounded by $o(|(a,\varepsilon)|) \|w\|_{H^{1,0}_0(\Omega,\C)}$ as before.

\smallskip

The fourth term in Lemma~\ref{lemma:new-operator} is
\[
\int_{\Omega} F(a,\alpha) v \cdot 
\overline{(i \nabla + A_0^\alpha) w} .
\]
Exactly as for the first term, the left part gives
\[
(\xi - 1) \nabla (a \cdot A_0^{\sfrac{1}{2}} )v - i J_a^{-1} (a \cdot \nabla v) \nabla \xi + o(|(a,\varepsilon)|),
\] 
being the $o(|(a,\varepsilon)|)$ in $L^\infty$, while 
the right part is
\[
\overline{ \left[ (i \nabla + A_0^{\sfrac{1}{2}}) w + 2 \varepsilon A_0^{\sfrac{1}{2}} w \right]}.
\]
The first order term is then
\begin{equation} \label{eq:quatre}
\int_{\Omega} (\xi - 1)  \nabla (a \cdot A_0^{\sfrac{1}{2}} ) v\cdot 
\overline{(i \nabla + A_0^{\sfrac{1}{2}}) w} 
- i (a \cdot \nabla v) \nabla \xi \cdot 
\overline{(i \nabla + A_0^{\sfrac{1}{2}}) w},
\end{equation}
and the rest is still bounded by the same quantity. 
Finally, all the remaining terms are also negligeable with respect to $|(a,\eps)|\to0$
using again \eqref{eq:xi} and \eqref{eq:det}. 
A combination of \eqref{eq:second}--\eqref{eq:quatre} 
gives Claim~\ref{claim:intermediate}.
\end{proofclaim}

\smallskip

\begin{claim} \label{claim:intermediate2}
We now prove that
\begin{align*}
\phantom{a}_{(H^{1,0}_0(\Omega,\C))^\star} & \left\langle 
\mathcal{L}'(0,\tfrac{1}{2})[(a,\varepsilon)] \varphi_j 
, \varphi_k \right\rangle_{H^{1,0}_0(\Omega,\C)} \\
& = \int_{\Omega} \Delta \xi \, (a \cdot \nabla \varphi_j) \, 
\overline{\varphi_k} 
+ 2 \nabla \xi \cdot \nabla ( a \cdot \nabla \varphi_j) \, 
\overline{\varphi_k} \\
& + \int_{\Omega} - 2i \, ( A_0^{\sfrac{1}{2}} \cdot \nabla \xi ) 
(a \cdot \nabla \varphi_j) \, \overline{\varphi_k} 
+ 2 i \, (\xi -1) \nabla \varphi_j \cdot (a \cdot  A_0^{\sfrac{1}{2}}) 
\, \overline{\varphi_k} \\
& + \int_{\Omega} (\xi-1) \, a \cdot \nabla (|A_0^{\sfrac{1}{2}}|^2) \, \varphi_j 
\overline{\varphi_k}. 
\end{align*}
\end{claim}

\begin{proofclaim}
First of all, we integrate by parts in Claim~\ref{claim:intermediate} 
all terms containing a derivative of $\varphi_k$ to move it on the 
other terms thanks to the regularity of eigenfunctions. Next we 
use in turn the following identities 
\[
a \cdot \nabla ( \nabla \varphi_j \cdot \nabla \xi) 
= \nabla \varphi_j \cdot \nabla (a \cdot \nabla \xi) 
+ \nabla \xi \cdot (a \cdot \nabla \varphi_j)
\]
and
\begin{equation} \label{eq:help}
A_0^{\sfrac{1}{2}} \cdot \nabla (a \cdot \nabla \xi) 
- a \cdot \nabla (A_0^{\sfrac{1}{2}} \cdot \nabla \xi) 
+ \nabla \xi \cdot \nabla ( a \cdot A_0^{\sfrac{1}{2}} ) = 0,
\end{equation}
\[
\Delta (a \cdot A_0^{\sfrac{1}{2}} ) = 0,
\]
\[
2 A_0^{\sfrac{1}{2}} \cdot \nabla (a \cdot A_0^{\sfrac{1}{2}} )
= a \cdot \nabla (|A_0^{\sfrac{1}{2}}|^2),
\]
which hold true in $\Omega \setminus\{0\}$ since 
$\nabla \times A_0^{\sfrac{1}{2}} = 0$ and 
$\nabla \cdot A_0^{\sfrac{1}{2}} = 0$ in $\Omega \setminus\{0\}$.
\end{proofclaim}

\smallskip

To prove the lemma, it remains us to show that 
\[
\int_{\Omega} (\xi - 1) \, a \cdot \nabla 
\left[ (i \nabla + A_0^{\sfrac{1}{2}})^2 \varphi_j \right] 
\, \overline{\varphi_k}
- (i \nabla + A_0^{\sfrac{1}{2}})^2 \left[ (\xi - 1) a 
\cdot \nabla \varphi_j \right] \, \overline{\varphi_k}
\]
is the same as in Claim~\ref{claim:intermediate2}. We can write 
the first part as
\begin{align*}
(\xi - 1) \, a \cdot \nabla  
\left[ (i \nabla + A_0^{\sfrac{1}{2}})^2 \varphi_j \right] 
& = (\xi - 1) (i \nabla + A_0^{\sfrac{1}{2}})^2 (a \cdot \nabla \varphi_j) \\
& + 2 i (\xi - 1) \left[ a \cdot \nabla ( A_0^{\sfrac{1}{2}} \cdot  \nabla \varphi_j )
- A_0^{\sfrac{1}{2}} \cdot \nabla (a \cdot \nabla \varphi_j ) \right] \\
& + (\xi - 1) a \cdot \nabla ( | A_0^{\sfrac{1}{2}} |^2 ) \varphi_j.
\end{align*}
Using a similar identity to \eqref{eq:help} for the second term, this gives
\begin{align} \label{eq:aaa1}
(\xi - 1) (i \nabla + A_0^{\sfrac{1}{2}})^2 (a \cdot \nabla \varphi_j) 
+ 2 i (\xi - 1) \nabla \varphi_j \cdot \nabla ( a \cdot A_0^{\sfrac{1}{2}} )
+ (\xi - 1) a \cdot \nabla ( | A_0^{\sfrac{1}{2}} |^2 ) \varphi_j.
\end{align}
The second part is
\begin{align*}
- (i \nabla + A_0^{\sfrac{1}{2}})^2 \left[ (\xi - 1) a 
\cdot \nabla \varphi_j \right] 
& = - (\xi-1) (i\nabla + A_0^{\sfrac{1}{2}})^2 (a \cdot \nabla \varphi_j) 
+ \Delta \xi (a \cdot \nabla \varphi_j) \\
& - 2 i \nabla \xi \cdot (i \nabla + A_0^{\sfrac{1}{2}} ) (a \cdot \nabla \varphi_j).
\end{align*}
This gives us
\begin{equation} \label{eq:aaa2}
- (\xi-1) (i\nabla + A_0^{\sfrac{1}{2}})^2 (a \cdot \nabla \varphi_j) 
+ \Delta \xi (a \cdot \nabla \varphi_j)
+ 2 \nabla \xi \cdot \nabla ( a \cdot \nabla \varphi_j) 
- 2i \, (\nabla \xi \cdot A_0^{\sfrac{1}{2}} ) \,  (a \cdot \nabla \varphi_j).
\end{equation}
By summing \eqref{eq:aaa1} and \eqref{eq:aaa2} we recognize the final claim.

\resetclaim

\section{Proof of Claim~\ref{claim:general-expression-G-prim} in Lemma~\ref{lemma:first-order-G-a}} \label{app:Gprim}

Let $\varepsilon := \alpha - \tfrac{1}{2}$. We first look at
\begin{align*}
& \phantom{a}_{(H^{1,0}_0(\Omega,\C))^\star}\left\langle 
\mathcal{L}_{(a,\alpha)} \left( \frac{\varphi_j}{\sqrt{J_a}} \right), 
\sqrt{J_a} \varphi_k \right\rangle_{H^{1,0}_0(\Omega,\C)} 
= \phantom{a}_{(H^{1,0}_0(\Omega,\C))^\star}\left\langle 
\mathcal{L}_{(a,\alpha)} \varphi_j, \varphi_k 
\right\rangle_{H^{1,0}_0(\Omega,\C)} \\
& \hspace{4cm} + \phantom{a}_{(H^{1,0}_0(\Omega,\C))^\star}\left\langle 
\mathcal{L}_{(a,\alpha)} v, (\sqrt{J_a} - 1) \varphi_k
\right\rangle_{H^{1,0}_0(\Omega,\C)} \\
& \hspace{4cm} + \phantom{a}_{(H^{1,0}_0(\Omega,\C))^\star}\left\langle 
\mathcal{L}_{(a,\alpha)} \left( v \frac{1 - \sqrt{J_a}}{\sqrt{J_a}} \right), 
\sqrt{J}_a \varphi_k 
\right\rangle_{H^{1,0}_0(\Omega,\C)} .
\end{align*}
Since, by Lemma~\ref{lemma:first-order-La}, $\mathcal{L}_{(a,\alpha)} 
= O(|(a,\varepsilon)|)$ for $(a,\varepsilon) \to 0$, and by \eqref{eq:xi} and 
\eqref{eq:det}  
\[
| 1 - \sqrt{J_a} | \leq C \, |(a,\varepsilon)|,
\]
it holds that the second and third terms are bounded by $o(|a,\varepsilon)|)$, 
as $|(a,\varepsilon)| \to 0$. In the first term, from Lemma~\ref{lemma:first-order-La} 
we conclude that the first order term is 
\begin{equation} \label{eq:bbb1}
\phantom{a}_{(H^{1,0}_0(\Omega,\C))^\star}\left\langle 
\mathcal{L}'(0,\tfrac{1}{2}) \varphi_j, \varphi_k 
\right\rangle_{H^{1,0}_0(\Omega,\C)}.
\end{equation}

Next, we look at 
\begin{align*}
& \phantom{a}_{(H^{1,0}_0(\Omega,\C))^\star}\left\langle 
(i \nabla + A_0^\alpha)^2 \left( \frac{\varphi_j}{\sqrt{J_a}} \right),
\sqrt{J_a} \varphi_k 
\right\rangle_{H^{1,0}_0(\Omega,\C)} \\
& \hspace{6cm} := \int_{\Omega} (i \nabla + A_0^\alpha) \left( \frac{\varphi_j}{\sqrt{J_a}} \right) 
\cdot \overline{ (i \nabla + A_0^\alpha) (\sqrt{J_a} \varphi_k )},
\end{align*}
by definition. Split it again in several pieces, the left part gives
\begin{align*}
(i \nabla + A_0^{\sfrac{1}{2}}) \varphi_j 
+ (i \nabla + A_0^{\sfrac{1}{2}}) \left(\varphi_j ( 1 - \sqrt{J_a} ) (\sqrt{J_a})^{-1} \right) 
+ 2 \varepsilon A_0^{\sfrac{1}{2}} \varphi_j 
+ 2 \varepsilon A_0^{\sfrac{1}{2}} \varphi_j ( 1 - \sqrt{J_a} ) (\sqrt{J_a})^{-1} ,
\end{align*}
while the right part reads
\begin{align*}
 \overline{ \left[ (i \nabla + A_0^{\sfrac{1}{2}}) \varphi_k 
+ (i \nabla + A_0^{\sfrac{1}{2}}) ((\sqrt{J_a} - 1)\varphi_k) 
+ 2 \varepsilon A_0^{\sfrac{1}{2}} \varphi_k 
+ 2 \varepsilon A_0^{\sfrac{1}{2}} (\sqrt{J_a} - 1) \varphi_k \right]}.
\end{align*}
Here, using \eqref{eq:xi} and \eqref{eq:det} we get
\[
| (\sqrt{J_a} - 1) (\sqrt{J_a})^{-1} + \tfrac{1}{2} a \cdot \nabla \xi| 
\leq C |(a,\varepsilon)|^2,
\]
and
\[
|\sqrt{J_a} - 1 - \tfrac{1}{2} a \cdot \nabla \xi| \leq C |(a,\varepsilon)|^2,
\]
as $|(a,\varepsilon)| \to 0$, for $C > 0$ independent of $|(a,\varepsilon)|$. 
Therefore the unperturbed term is
\[
\int_{\Omega} (i \nabla + A_0^{\sfrac{1}{2}}) \varphi_j 
\cdot \overline{(i\nabla + A_0^{\sfrac{1}{2}}) \varphi_k} 
:= \phantom{a}_{(H^{1,0}_0(\Omega,\C))^\star}\left\langle 
(i \nabla + A_0^{\sfrac{1}{2}})^2 \varphi_j, \varphi_k 
\right\rangle_{H^{1,0}_0(\Omega,\C)},
\]
while the first order terms are given by
\begin{equation} \label{eq:bbb2}
\begin{split}
& \frac{1}{2} \int_{\Omega}  (i \nabla + A_0^{\sfrac{1}{2}}) \varphi_j 
\cdot \overline{ (i \nabla + A_0^{\sfrac{1}{2}}) ( a \cdot \nabla \xi) \varphi_k } 
- (i \nabla + A_0^{\sfrac{1}{2}}) ((a\cdot \nabla \xi) \varphi_j) \cdot 
\overline{(i \nabla + A_0^{\sfrac{1}{2}}) \varphi_k} \\
& \qquad \qquad + 2 \varepsilon \int_{\Omega} (i \nabla + A_0^{\sfrac{1}{2}}) \varphi_j 
\cdot A_0^{\sfrac{1}{2}} \overline{\varphi_k} 
+ A_0^{\sfrac{1}{2}} \varphi_j \cdot 
\overline{(i \nabla + A_0^{\sfrac{1}{2}}) \varphi_k}.
\end{split}
\end{equation}
An integration by parts in \eqref{eq:bbb2} together with \eqref{eq:bbb1} gives us the result.

\section*{Acknowledgements}

The authors would like to thank their mentor prof. Susanna Terracini 
for her encouragement and useful discussions on the theme, as well as prof. Virginie
Bonnaillie-Noël for her courtesy to provide all the figures. 

The authors are
partially supported by the project ERC Advanced Grant
2013 n. 339958 : ``Complex Patterns for Strongly Interacting Dynamical
Systems -- COMPAT''. 
 L. Abatangelo is partially
supported by the PRIN2015 grant ``Variational methods, with
applications to problems in mathematical physics and geometry''
and by the 2017-GNAMPA project ``Stabilità e analisi
  spettrale per problemi alle derivate parziali''.


\end{document}